\documentclass[11pt]{amsart}
\usepackage[dvipsnames]{xcolor}
\usepackage{amsfonts,amssymb,amsmath,amscd,amstext}
\usepackage{mathrsfs}
\usepackage[colorlinks=true,linkcolor=teal,citecolor=purple,urlcolor=Plum]{hyperref}
\usepackage[utf8]{inputenc}
\usepackage[final]{graphicx}
\usepackage{aligned-overset}
\usepackage{float}
\usepackage{comment}
\usepackage[noabbrev,capitalize,nameinlink]{cleveref}
\usepackage{mathdots}
\usepackage[a4paper,top=2.3cm,bottom=2.3cm,left=1.4cm,right=1.4cm]{geometry}

\renewcommand{\geq}{\geqslant}

\newcommand{\hhh}{{\mathcal{H}}}
\newcommand{\hn}{\mathbb{H}^n}
\newcommand{\A}{\mathbf{A}}
\newcommand{\B}{\mathbf{B}}
\newcommand{\C}{\mathbf{C}}
\newcommand{\D}{\mathbf{D}}

\newcommand{\n}[1]{\nabla^\varepsilon_{ #1 }}
\newcommand{\tmc}{{\mathscr{H}}}

\newcommand{\vh}{\nu}

\newcommand{\rr}{{\mathbb{R}}}

\newcommand{\hh}{{\mathbb{H}}}

\newcommand{\eps}{\varepsilon}

\newcommand{\s}{\mathcal S}

\renewcommand{\j}{\mathcal{C}}

\newcommand{\W}{\mathbf Y}
\newcommand{\Y}{\mathcal X}

\newcommand{\N}{\mathcal N}

\newcommand{\E}{\mathbf{E}}
\newcommand{\X}{\mathbf{X}}

\newcommand{\Z}{\mathbf{Z}}
\renewcommand{\n}{\mathbf{N}}
\newcommand{\G}{\mathcal G}

\newcommand{\average}{{\mathchoice {\kern1ex\vcenter{\hrule height.4pt
				width 6pt
				depth0pt} \kern-9.7pt} {\kern1ex\vcenter{\hrule height.4pt width 4.3pt
				depth0pt}
			\kern-7pt} {} {} }}

\definecolor{champagne}{rgb}{0.97, 0.91, 0.81}

\definecolor{asparagus}{rgb}{0.53, 0.66, 0.42}
\DeclareMathOperator{\ric}{Ric}

\DeclareMathOperator{\divv}{div}

\DeclareMathOperator{\trace}{trace}

\DeclareMathOperator{\hess}{Hess}
\newcommand{\jacobi}{\mathcal{J}}

\DeclareMathOperator{\supp}{supp}
\DeclareMathOperator{\spann}{span}
\DeclareMathOperator{\jac}{Jac}

\DeclareMathOperator{\R}{R}
\DeclareMathOperator{\sym}{sym}

\newtheorem{theorem}{Theorem}[section]
\newtheorem{proposition}[theorem]{Proposition}

\newtheorem{lemma}[theorem]{Lemma}
\newtheorem{corollary}[theorem]{Corollary}
\theoremstyle{definition}

\theoremstyle{remark}

\numberwithin{equation}{section}

\author[S.~Verzellesi]{Simone Verzellesi}
\address[S.~Verzellesi]{Dipartimento di Matematica ``Tullio Levi-Civita'', Università degli Studi di Padova, via Trieste 63, 35131 Padova (PD), Italy}
\email{simone DOT verzellesi AT unipd DOT it}

\title[Variation formulas for mean curvature functionals]{Variation formulas for mean curvature functionals}

\date{\today}

\thanks{\textit{Acknowledgments.} S. Verzellesi thanks R. Monti, J. Pozuelo, M. Ritoré and D. Vittone for fruitful discussion about the addressed topics.
\textit{Memberships and funding information.} S. Verzellesi is member of the Istituto Nazionale di Alta Matematica (INdAM), Gruppo Nazionale per l'Analisi Matematica, la Probabilità e le loro Applicazioni (GNAMPA). 
S.~Verzellesi is supported by the University of Padova. S.~Verzellesi received funding through INdAM-GNAMPA 2026 Project \emph{Variational, Geometric, and Analytic Perspectives on Regularity}, CUP E53C25002010001}

\begin{document}
\begin{abstract}

We establish first and second variation formulas for general mean curvature functionals under arbitrary variations in arbitrary Riemannian manifolds. We obtain corresponding formulas in the sub-Riemannian Heisenberg group through a Riemannian approximation scheme.
\end{abstract}
\maketitle
\section{Introduction}

Let $M$ be a Riemannian manifold and let $S\subseteq M$ be a smooth, embedded, closed, two-sided hypersurface. In this paper, we consider geometric functionals of the form
\begin{equation}
    \label{intro_riemfun}
\mathcal H_f(S)=\int_S f(H)\,d S,
\end{equation}
where $H$ denotes the mean curvature of $S$ and $f:\mathbb R\to\mathbb R$ is a smooth function. Our main purpose is to establish first and second variation formulas of $\mathcal H_f$ under arbitrary smooth variations. We refer to \Cref{maintheoremappendixA} for a summary of the main outcomes in this regard.

\medskip
Under additional constraints on the ambient manifold, on the class of variations or on the curvature functional, such variation formulas are well-known and appear in several different contexts. When $f\equiv 1$, \eqref{intro_riemfun} boils down to the area functional, for which variation formulas are available in great generality \cite{MR4676392}. Other natural choices of $f$ specialize \eqref{intro_riemfun} to relevant curvature energies, such as the \emph{Willmore energy}, the \emph{total mean curvature} and the \emph{total inverse mean curvature}. Variation formulas for these and more general curvature functionals have typically been addressed within very specific frameworks, for instance allowing for restricted classes of variations \cite{MR3962031,MR4927775}, or imposing additional geometric constraints on the ambient manifold \cite{MR3666332,MR3962031,MR4927775,MR341351}. However, to the best of our knowledge, a general formula simultaneously allowing an arbitrary mean curvature functional, arbitrary variations, and an arbitrary ambient manifold does not seem to be available. Our first purpose is to provide such a formula, both at first and (most importantly) at second order.

\medskip
Variation formulas are natural tools when addressing geometric variational problems. The first variation identifies the corresponding Euler-Lagrange equation, while the second variation provides the relevant stability information. Besides the obvious role played by such variational tools in connection with the isoperimetric problem \cite{MR86338,MR0731682,MR0917854,MR3921314}, more general mean curvature functionals arise e.g. in connection with the establishment of sharp geometric inequalities of Minkowski \cite{MR3544938,MR2522433,MR1511220}, Willmore \cite{MR4037467,MR4941125,MR291994,MR278240}, and Heintze-Karcher \cite{MR533065, MR1173047,MR996826,MR3702549} type. Accordingly, general variation formulas provide a unified framework for these problems, while allowing the resulting identities to be adapted to the specific geometry under consideration.

\medskip
The derivation of the variation formulas, inspired by the approach proposed in \cite{MR4676392} and carried out in \Cref{appendiceriem}, is based on a direct analysis of the pointwise evolution of the relevant geometric quantities. After introducing the basic preliminaries about hypersurfaces and variations, we establish a computational lemma for iterated covariant derivatives, which provides the main technical tool for the subsequent calculations. We then describe how the basic geometric quantities associated with the evolving hypersurface change along the variation. More precisely, we derive first and second order formulas for the induced Jacobian, the unit normal, the second fundamental form, and the mean curvature. While the corresponding first-order identities are well-known, their second-order counterparts, in the general setting considered here, constitute the main computational step of the argument. These pointwise identities are then combined with the area formula and integrated over the reference hypersurface, leading to the first and second variation formulas for \eqref{intro_riemfun}. Throughout the argument, no assumptions are imposed on the curvature of the ambient manifold, and the variation is allowed to have both tangential and normal components.

\medskip
As a second main result, we exploit the Riemannian machinery to establish corresponding \emph{sub-Riemannian} variation formulas. Namely, we focus on the sub-Riemannian  \emph{Heisenberg group} $\mathbb H^n$, a suitable stratified Lie group which provides the prototypical model within sub-Riemannian \cite{MR3971262,MR2363343} and \emph{pseudohermitian} \cite{MR2165405} geometries. In this framework, a general mean curvature functional takes the form 
\begin{equation}\label{tmcfunsubriemderfgt5ryhythyh}
    \tmc^\hhh_f(S)=\int_S f\left(H^\hhh\right)\,d\sigma^\hhh,
\end{equation} 
where $S$ is an embedded, closed hypersurface, $H^\hhh$ is its \emph{horizontal mean curvature} and $\sigma^\hhh$ is its \emph{sub-Riemannian surface measure}. We refer to \Cref{appendicesubriem} for the related definitions, and to \Cref{teoremavariazionigeneralimainsubriem} for the statement of the variation formulas.

\medskip
Typically, sub-Riemannian geometry is determined by a distinguished distribution of admissible directions, the \emph{horizontal distribution} $\hhh$, endowed with a metric defined \emph{a priori} only along such directions. Accordingly, both intrinsic and extrinsic geometric quantities are built from such horizontal geometry. The horizontal geometry of a hypersurface may collapse when the horizontal distribution coincides with its tangent space, i.e. in the presence of \emph{characteristic points}. Nevertheless, a sub-Riemannian structure can be \emph{a posteriori} approximated by a suitable family of Riemannian metrics whose associated geometries converge, in a suitable sense, to the sub-Riemannian one.
Our approach to the sub-Riemannian problem is indeed inspired by the above construction. Namely, following a well-established approach \cite{MR2774306,MR2262784,MR4761954,MR2043961,MR5029815}, we approximate the sub-Riemannian structure $\left(\hh^n,\hhh,\langle\cdot,\cdot\rangle\right)$ by means of a sequence of collapsing Riemannian structures $$\big(\hh^n,\langle\cdot,\cdot\rangle_\eps\big),\qquad \eps>0.$$
For each approximating metric, the Riemannian variation formulas apply directly. The sub-Riemannian formulas are then obtained by a careful limiting argument. To this aim, we restrict to consider variations which do not move characteristic points. This procedure requires a detailed analysis of the convergence of the geometric quantities entering the Riemannian formulas. We therefore study both ambient objects, such as the curvature tensors of the approximating metrics, and quantities associated with the hypersurface, in particular those built from its second fundamental form. A notable feature of the approximation is that, although some of these quantities may diverge individually as $\varepsilon\to 0$, the specific combinations in which they enter the variation formulas remain convergent through appropriate cancellations. This phenomenon is well known in the Riemannian approximation of first and second variations of area \cite{MR2723818,MR5029815}. Here we show that it persists for the more involved quantities arising from the second variation of the mean curvature.
\section{Variation formulas for Riemannian mean curvature functionals}\label{appendiceriem}
In this section, we establish first and second variation formulas for \eqref{intro_riemfun}. Much of the notation and several results are borrowed from \cite{MR4676392}. We refer to \cite{MR1138207} for a general account of Riemannian geometry.
\subsection{Preliminaries} 
Here and hereafter, $M$ is a fixed $(n+1)$-dimensional Riemannian manifold for some $n\geq 1$, $\langle\cdot,\cdot\rangle$ is its Riemannian metric and $\nabla$ is its Levi-Civita connection.
In the following, Einstein's summation convention is assumed.
Denote by $\R$ both the $(3,1)$ and the $(4,0)$ Riemann tensor, namely
\begin{equation*}
    \R(\A ,\B )\C =\nabla_{\A}\nabla_{\B} \C-\nabla_{\B}\nabla_{\A} \C-\nabla_{[ \A,\B ]}\C,\quad \R(\A,\B,\C,\D)=\left\langle \R(\A,\B)\C,\D\right\rangle,\quad \A,\B,\C,\D\in\Gamma(TM).
\end{equation*}
If $\A,\C\in\Gamma(TM)$ are fixed, denote by $\j$ the $(2,0)$-tensor field defined by 
\begin{equation*}
   \j(\A ,\C)(\B,\D)\coloneqq\R(\A ,\B ,\C,\D),\qquad \B,\D\in\Gamma(TM).
\end{equation*}
We recall that, if $ \A,\B ,\C,\D \in\Gamma(TM)$, then
    \begin{align*}
        \left(\nabla_{\A}\ric\right)(\B,\C )=-\trace\left(\nabla_{\A} \R\right)(\B,\cdot,\C ,\cdot),\qquad
        \left(\nabla_{\A}\R\right)(\B,\C,\D,\D)=0.
    \end{align*}
 Let $p\in M$. Let $x_1,\ldots,x_{n+1}$ be local coordinates in $M$ near $p$.
 Write 
 \begin{equation*}
     \R\left(\frac{\partial }{\partial x_i},\frac{\partial}{\partial x_j}\right)\frac{\partial }{\partial x_k}=\R_{ijk}^l\frac{\partial}{\partial x_l},\qquad  \left(\nabla_{\frac{\partial}{\partial x_\alpha}}\R\right)\left(\frac{\partial }{\partial x_i},\frac{\partial}{\partial x_j}\right)\frac{\partial }{\partial x_k}=\left(\nabla \R\right)_{\alpha i j k}^l\frac{\partial}{\partial x_l},\qquad\alpha,i,j,k=1,\ldots,n+1.
 \end{equation*}
 Notice that 
 \begin{equation}\label{riemanincoordinate}
     \R_{ijk}^l=\frac{\partial\Gamma_{jk}^l}{\partial x_i}-\frac{\partial \Gamma_{ik}^l}{\partial x_j}+\Gamma_{jk}^m\Gamma_{im}^l-\Gamma_{ik}^m\Gamma_{jm}^l,\qquad i,j,k,l=1,\ldots,n+1,
 \end{equation}
where $\Gamma^k_{ij}$ are the Christoffel symbols. In particular, when $x_1,\ldots,x_{n+1}$ are normal coordinates centered at $p$,
\begin{equation}\label{normcoordcristo}
   \Gamma_{ij}^{k}(p)=0\qquad\text{for any $i,j,k=1,\ldots,n+1$,} 
\end{equation}
 so that 
\begin{equation}\label{rieminnorm}
    \R_{ijk}^l(p)=\frac{\partial \Gamma_{jk}^l}{\partial x_i}(p)-\frac{\partial \Gamma_{ik}^l}{\partial x_j}(p)\qquad\text{for any $i,j,k,l=1,\ldots,n+1$}
\end{equation}
and 
\begin{equation}\label{deririeminnorm}
    \left(\nabla \R\right)_{\alpha ijk}^l(p)=\frac{\partial ^2\Gamma_{jk}^l}{\partial x_\alpha\partial x_i}(p)-\frac{\partial^2 \Gamma_{ik}^l}{\partial x_\alpha\partial x_j}(p)\qquad\text{for any $\alpha,i,j,k,l=1,\ldots,n+1$}.
\end{equation}
    \subsection{Hypersurfaces}
    Throughout this section, $S\subseteq M$ is a smooth, embedded, closed, two-sided hypersurface. Denote by $\n$ a (globally defined) unit normal to $S$. If $\A\in\Gamma(TM)$, denote by $\A^T$ its orthogonal projection onto $TS$. Denote by $A$ and $h$ the shape operator and the second fundamental form, i.e.
    \begin{equation*}
        A\left(\A\right)=\nabla_\A\n,\qquad h(\A,\B)=\left\langle A(\A),\B\right\rangle,\qquad\A,\B\in\Gamma(TS).    \end{equation*}
    Denote by $H$ the (non-averaged) mean curvature of $S$, and by $\nabla^S$ the Levi-Civita connection of $S$. If $B$ is a $(2,0)$-tensor field on $S$, set   \begin{equation*}
       B^t(\A,\B)\coloneqq B(\B,\A),\qquad \langle L_B(\A),\B\rangle=B(\A,\B),\qquad \A,\B\in\Gamma(TS).
    \end{equation*} 
    For any $k\in\mathbb N_+$, define 
\begin{equation*}
    B^k(\A,\B)=\left\langle L_B^k(\A),\B\right\rangle,\qquad \A,\B\in \Gamma(TS).
\end{equation*}
If $B$ is symmetric, then $B^k$ is symmetric. If $p\in S$ and $e_1,\ldots,e_n$ is an orthonormal basis of $T_p S$, we write
\begin{equation*}\label{atok}
   \left(B^k\right)_{ij}\coloneqq B^k(e_i,e_j)=\sum_{l_1,\ldots,l_{k-1}=1}^nB_{il_1}B_{l_!l_2}\cdots B_{l_{k-1}l_{k-1}}B_{l_{k-1}j},\qquad i,j=1,\ldots, n.
    \end{equation*}
 If $\A,\B,\C\in\Gamma(TS)$,
 the \emph{traced Codazzi equation} reads as  
     \begin{align}
            \left(\divv^S h\right)(\A)&=\A H+\ric(\A,\n)\label{tracedcodazzi}
        \end{align}
        We will make use of the following consequence of the traced Codazzi equation \eqref{tracedcodazzi}.
    \begin{lemma}
        Let $\A\in\Gamma(TM)$. Then
        \begin{equation}\label{corollarioditracedcodazzi}
            \left\langle h,\nabla^S\A\right\rangle+\ric(\A,\n)=\divv^S A\left(\A^T\right)-\A^T H+\left\langle\A,\n\right\rangle\left(|h|^2+\ric(\n,\n)\right).
        \end{equation}
            \end{lemma}
        \begin{proof}
            Let $p\in S$. Let $\E_1,\ldots,\E_n$ be a geodesic frame of $S$ at $p$. Then
      \begin{equation*}              \begin{split}                 \left\langle h,\nabla^S\A\right\rangle
      &=\sum_{i,j=1}^n h(\E_i,\E_j)\E_i\left\langle\A,\E_j\right\rangle-\sum_{i,j=1}^n h(\E_i,\E_j)\left\langle\A,\nabla_{\E_i}\E_j\right\rangle\\
      &=\sum_{i,j=1}^n \E_i\left(h(\E_i,\E_j)\left\langle\A,\E_j\right\rangle\right)-\sum_{i,j=1}^n\E_i\left(h(\E_i,\E_j)\right)\left\langle\A,\E_j\right\rangle-\left\langle\A,\n\right\rangle\sum_{i,j=1}^n h(\E_i,\E_j)\left\langle\n,\nabla_{\E_i}\E_j\right\rangle\\
      &=\sum_{i=1}^n \E_i\left(h\left(\E_i,\A^T\right)\right)-\sum_{i,j=1}^n\left(\nabla^S_{\E_i}h\right)(\E_i,\E_j)\left\langle\A,\E_j\right\rangle+\left\langle\A,\n\right\rangle|h|^2\\
      &=\divv^S A\left(\A^T\right)-\left(\divv^S h\right)\left(\A^T\right)+\left\langle\A,\n\right\rangle|h|^2.
             \end{split}
      \end{equation*}  
      Therefore
      \begin{equation*}
          \begin{split}
               \left\langle h,\nabla^S\A\right\rangle+\ric(\A,\n)&= \divv^S A\left(\A^T\right)-\left(\divv^S h\right)\left(\A^T\right)+\ric\left(\A^T,\n\right)+\left\langle\A,\n\right\rangle\left(|h|^2+\ric(\n,\n)\right)\\
               \overset{\eqref{tracedcodazzi}}&{=}\divv^S A\left(\A^T\right)-\A^T H+\left\langle\A,\n\right\rangle\left(|h|^2+\ric(\n,\n)\right).
          \end{split}
      \end{equation*}
      \end{proof}
  \subsection{Variations}\label{subsec_variations}
    A variation is a smooth map $\Phi:I\times M\to M$, where $I\subseteq\rr$ is any open neighborhood of $0$, such that:
\begin{itemize}
    \item $p\mapsto \Phi(t,p)$ is a diffeomorphism for any $t\in I$;
    \item $\Phi(0,p)=p$ for any $p\in M$. 
\end{itemize}
We adopt the notation $\Phi_t(p)\coloneqq\Phi(t,p)$. A variation is \emph{compactly supported} if $\Phi_t(p)=p$ outside a compact set $K(\Phi)$. In the following, we tacitly assume that a variation is compactly supported. This is clearly not restrictive when computing variations of a closed hypersurface. Nevertheless, closedness is assumed just for the sake of simplicity: the forthcoming variation formulas would clearly hold as well provided the variation is supported far from the boundary of the reference hypersurface. Define the time-dependent vector field $\Y$ by
\begin{equation}\label{def_xechis}
    \Y(t,q)=\left.\frac{\partial}{\partial s}\right|_{s=0}\left(\Phi_{s+t}\circ\Phi_t^{-1}\right)(q)\qquad\text{for any $t\in I$ and $q\in M$.}
\end{equation}
Define
\begin{equation*}
    \X(q)=\Y(0,q),\qquad \X'(q)=\left.\frac{\partial }{\partial t}\right|_{t=0}\Y(t,q),\qquad\Z(q)=\left(\X'+\nabla_\X\X\right)(q)\qquad\text{for any $q\in M$.}
\end{equation*}
The vector fields $\X$ and $\Z$ are known as \emph{variational velocity field} and \emph{variational acceleration field} of $\Phi$. We point out that, while the velocity depends only on the underlying differential structure, the acceleration does depend on the metric via the connection term $\nabla_\X\X$. 
Fix $\varphi\in C^\infty(S)$ and $ \X^T\in\Gamma(TS)$ such that 
\begin{equation*}
    \X|_S=\ \X^T+\varphi \n.
\end{equation*}
When $X^T\equiv 0$, $\Phi$ is called \emph{normal variation}. 
     Set $S_t\coloneqq \Phi_t(S)$ for any $t\in I.$ If $p\in S$ is fixed, set $\beta_p(t)\coloneqq \Phi_t(p)$ for every $t\in I$. If there is no ambiguity, we write $\beta=\beta_p$. By definition, $\beta(t)\in S_t$ for every $t\in I$. Moreover,

\begin{equation*}\label{xtempodefinisceflusso}
    \begin{split}
        \Y(t,\beta(t))=\Y(t,\Phi_t(p))=\left.\frac{\partial}{\partial s}\right|_{s=0}\Phi_{s+t}(p)=\left.\frac{\partial}{\partial s}\right|_{s=t}\Phi_{s}(p)=\dot\beta(t)\qquad\text{for any $t\in I$}.
    \end{split}
\end{equation*}
Therefore, $\Phi$ is the (time-dependent) flow of $\Y$. 
For any $t\in I$, denote by $\N(t,\cdot)$ a smooth choice of arbitrary local extensions, around $S_t$, of unit normals to $S_t$, in such a way that $\N(0,q)=\n(q)$ locally around $S$. Moreover, if $p\in S$, denote by $H_t\left(\Phi_t(p)\right)$ the mean curvature of $S_t$ at $\Phi_t(p)$. When $p\in S$ is fixed and local coordinates $x_1,\ldots,x_{n+1}$ around $p$ are given, we will write 
  \begin{equation*}
        \Y(t,q)=a^i(t,q)\left.\frac{\partial}{\partial x_i}\right|_{q},\qquad \N(t,q)=c^k(t,q)\left.\frac{\partial}{\partial x_k}\right|_{q}
    \end{equation*}
    and \begin{equation*}
        \X(q)=X^i(q)\left.\frac{\partial}{\partial x_i}\right|_{q},\qquad \X'(q)=\left(X'\right)^i(q)\left.\frac{\partial}{\partial x_i}\right|_{q},\qquad \n(q)=N^i(q)\left.\frac{\partial}{\partial x_i}\right|_{q}.
    \end{equation*}
    Notice that, by definition,
    \begin{equation*}
        X^i(q)=a^i(0,q),\qquad(X')^i(q)=\left.\frac{\partial }{\partial t}\right|_{t=0}a^i(t,q),\qquad N^i(q)=c^i(0,q).
    \end{equation*}
    We denote by $\frac{D}{dt}$ the covariant derivative operator, along $\beta$, induced by $\nabla$. We recall that
\begin{equation}\label{comefarecovdev}
    \frac{D}{dt}\left(\left.f^j(t)\frac{\partial}{\partial x_j}\right|_{\beta(t)}\right)=\dot f^i(t)\left.\frac{\partial}{\partial x_j}\right|_{\beta(t)}+a^i(t,\beta(t))f^j(t)\Gamma_{ij}^k(\beta(t))\left.\frac{\partial}{\partial x_k}\right|_{\beta(t)}.
\end{equation}
  Finally, when $p\in S$ and $e\in T_p S$ are fixed, we set  
    \begin{equation}\label{estensionevettore}
        E(t)\coloneqq\left(d\Phi_t\right)|_p(e)\qquad\text{for every $t\in I$.} 
    \end{equation}
    $E$ is a vector field along $\beta$, and $E(t)\in T_{\beta(t)}S_t$ for every $t\in I$. In local coordinates around $p$, we write
      \begin{equation*}
        E(t)=b^j(t)\left.\frac{\partial}{\partial x_j}\right|_{\beta(t)}
    \end{equation*}
    Fix $p\in S$. Let $e_1,\ldots,e_n$ be an orthonormal basis of $T_p S$. The \emph{Jacobian} of $\Phi_t|_{S}$ at $p$ is defined by
    \begin{equation*}
        \jac\left(\Phi_t|_{S}\right)\left(p\right)\coloneqq\sqrt{\det \G(t)}
    \end{equation*}
    where the symmetric matrix $\G$ is defined by
    \begin{equation}\label{grammatrix}
         \G(t)_{ij}\coloneqq\langle E_i(t),E_j(t)\rangle,\qquad i,j=1,\ldots,n.
    \end{equation}
    If $p\in S$ and we fix local coordinates $x_1,\ldots, x_{n+1}$ in a neighborhood of $p$, say $U$, the continuity of $\Phi$ ensures that $\Phi_t(q)\in U$ for any $t$ small and any $q$ sufficiently close to $p$. We will tacitly assume this. 
\subsection{A computational lemma}
The following technical lemma is the main computational tool used below.
\begin{lemma}\label{fondamentalecomplemmavariazioniriemannianejhrnfi4rfjnfin}
    Let $p\in M$. Let $x_1,\ldots,x_{n+1}$ be local coordinates in $M$ near $p$.  Let $ \A,\B ,\C ,\D\in\Gamma(TM)$. Then, near $p$,
     \begin{equation}\label{nablatrecosenonnormale}
    \begin{split}
       \nabla_{\A}\nabla_{\B} \C
        &=A^\alpha B^\beta\left(\frac{\partial^2 C^\gamma}{\partial x_\alpha\partial x_\beta}\frac{\partial}{\partial x_\gamma}+\frac{\partial C^\gamma}{\partial x_\alpha}\Gamma_{\beta\gamma}^\delta\frac{\partial}{\partial x_\delta}+C^\gamma\frac{\partial \Gamma_{\alpha\gamma}^\delta}{\partial x_\beta}\frac{\partial}{\partial x_\delta}+\frac{\partial C^\gamma}{\partial x_\beta}\Gamma_{\alpha\gamma}^\delta\frac{\partial}{\partial x_\delta}+C^\gamma\Gamma_{\alpha\gamma}^\delta\Gamma_{\beta\delta}^\eta\frac{\partial}{\partial x_\eta}\right)\\
        &\quad+A^\alpha\frac{\partial B^\beta}{\partial x_\alpha}\left(\frac{\partial C^\gamma}{\partial x_\beta}\frac{\partial}{\partial x_\gamma}+C^\gamma\Gamma_{\beta\gamma}^\delta\frac{\partial}{\partial x_\delta}\right)+\R(\A ,\B )\C .
    \end{split}
\end{equation}
Assume in addition that $x_1,\ldots,x_{n+1}$ are normal coordinates centered at $p$. Then, at $p$, 
\begin{equation}\label{nablatrecose}
        \begin{split}
            \nabla_{\A}\nabla_{\B}\C&=\R(\A ,\B )\C +A^\alpha B^\beta C^\gamma\frac{\partial \Gamma_{\alpha\gamma}^\delta}{\partial x_\beta} \frac{\partial}{\partial x_\delta}+A^\alpha B^\beta\frac{\partial^2 C^\gamma}{\partial x_\alpha \partial x_\beta} \frac{\partial}{\partial x_\gamma}+A^\alpha\frac{\partial B^\beta}{\partial x_\alpha}\frac{\partial C^\gamma}{\partial x_\beta}\frac{\partial}{\partial x_\gamma}
        \end{split}
    \end{equation}
    and
    \begin{equation}\label{nablaquattrocose}
    \begin{split}          &\nabla_{\A}\nabla_{\B}\nabla_{\C}\D=\left(\nabla_{\B} \R\right)(\A ,\C )\D+\R(\A ,\B )\nabla_{\C} \D+\R\left(\A ,\C \right)\nabla_{\B} \D+\R\left(\A ,\nabla_{\B} \C\right)\D+\R\left(\nabla_\A \B,\C\right)\D\\
          &\quad+A^\alpha B^\beta\left(\frac{\partial ^2C^\delta}{\partial x_\alpha\partial x_\beta}\frac{\partial D^\gamma}{\partial x_\delta} +\frac{\partial C^\delta}{\partial x_\beta}\frac{\partial^2 D^\gamma}{\partial x_\alpha\partial x_\delta}+\frac{\partial C^\delta}{\partial x_\alpha}\frac{\partial^2 D^\gamma}{\partial x_\beta\partial x_\delta} +  C^\delta\frac{\partial^3 D^\gamma}{\partial x_\alpha\partial x_\beta\partial x_\delta} \right)\frac{\partial}{\partial x_\gamma}\\
          &\quad+A^\alpha B^\beta\left(C^\eta\frac{\partial D^\delta}{\partial x_\eta}\frac{\partial \Gamma_{\alpha\delta}^\gamma}{\partial x_\beta} +\frac{\partial C^\delta  }{\partial x_\beta}D^\eta\frac{\partial \Gamma^\gamma_{\alpha\eta}}{\partial x_\delta} + C^\delta\frac{\partial D^\eta}{\partial x_\beta}\frac{\partial \Gamma^\gamma_{\alpha\eta}}{\partial x_\delta}+\frac{\partial C^\delta }{\partial x_\alpha}D^\eta\frac{\partial \Gamma^\gamma_{\delta \eta}}{\partial x_\beta} +C^\delta \frac{\partial D^\eta}{\partial x_\alpha}\frac{\partial \Gamma^\gamma_{\delta \eta}}{\partial x_\beta} \right)\frac{\partial}{\partial x_\gamma}\\
          &\quad+A^\alpha B^\beta C^\delta D^\eta\frac{\partial^2 \Gamma^\gamma_{\alpha \eta}}{\partial x_\beta\partial x_\delta} \frac{\partial}{\partial x_\gamma}+A^\alpha\frac{\partial B^\beta}{\partial x_\alpha}\left(\frac{\partial C^\delta}{\partial x_\beta}\frac{\partial D^\gamma}{\partial x_\delta}+ C^\delta\frac{\partial^2 D^\gamma}{\partial x_\beta\partial x_\delta}+ C^\delta D^\eta\frac{\partial \Gamma^\gamma_{\beta \eta}}{\partial x_\delta}\right)\frac{\partial}{\partial x_\gamma}.
        \end{split}
    \end{equation}
\end{lemma}
\begin{proof}
    By a direct computation, near $p$,
    \begin{equation*}
    \begin{split}
       \nabla_{\A}\nabla_{\B} \C&=A^\alpha\nabla_{\frac{\partial}{\partial x_\alpha}}\left(B^\beta\nabla_{\frac{\partial}{\partial x_\beta}} \left(C^\gamma\frac{\partial}{\partial x_\gamma}\right)\right)\\
        &=A^\alpha B^\beta\nabla_{\frac{\partial}{\partial x_\alpha}}\left(\nabla_{\frac{\partial}{\partial x_\beta}} \left(C^\gamma\frac{\partial}{\partial x_\gamma}\right)\right)+A^\alpha\frac{\partial B^\beta}{\partial x_\alpha}\nabla_{\frac{\partial}{\partial x_\beta}} \left(C^\gamma\frac{\partial}{\partial x_\gamma}\right)\\
        &=A^\alpha B^\beta \nabla_{\frac{\partial}{\partial x_\alpha}}\left(\frac{\partial C^\gamma}{\partial x_\beta} \frac{\partial}{\partial x_\gamma}+C^\gamma\Gamma_{\beta\gamma}^\delta \frac{\partial}{\partial x_\delta}\right)+A^\alpha\frac{\partial B^\beta}{\partial x_\alpha}\left(\frac{\partial C^\gamma}{\partial x_\beta}\frac{\partial}{\partial x_\gamma}+C^\gamma\Gamma_{\beta\gamma}^\delta\frac{\partial}{\partial x_\delta}\right)\\
        &=A^\alpha B^\beta\left(\frac{\partial^2 C^\gamma}{\partial x_\alpha\partial x_\beta}\frac{\partial}{\partial x_\gamma}+\frac{\partial C^\gamma}{\partial x_\alpha}\Gamma_{\beta\gamma}^\delta\frac{\partial}{\partial x_\delta}+C^\gamma\frac{\partial \Gamma_{\beta\gamma}^\delta}{\partial x_\alpha}\frac{\partial}{\partial x_\delta}+\frac{\partial C^\gamma}{\partial x_\beta}\Gamma_{\alpha\gamma}^\delta\frac{\partial}{\partial x_\delta}+C^\gamma\Gamma_{\beta\gamma}^\delta\Gamma_{\alpha\delta}^\eta\frac{\partial}{\partial x_\eta}\right)\\
        &\quad+A^\alpha\frac{\partial B^\beta}{\partial x_\alpha}\left(\frac{\partial C^\gamma}{\partial x_\beta}\frac{\partial}{\partial x_\gamma}+C^\gamma\Gamma_{\beta\gamma}^\delta\frac{\partial}{\partial x_\delta}\right).
    \end{split}
\end{equation*}
Since
\begin{equation*}
\begin{split}
     A^\alpha B^\beta C^\gamma\left(\frac{\partial\Gamma_{\beta\gamma}^\delta}{\partial x_\alpha}\frac{\partial}{\partial x_\delta}+\Gamma_{\beta\gamma}^\delta\Gamma_{\alpha\delta}^\eta\frac{\partial}{\partial x_\eta}\right)&= A^\alpha B^\beta C^\gamma\left(\frac{\partial\Gamma_{\beta\gamma}^\delta}{\partial x_\alpha}+\Gamma_{\beta\gamma}^\eta\Gamma_{\alpha\eta}^\delta\right)\frac{\partial}{\partial x_\delta}\\
     &\overset{\eqref{riemanincoordinate}}{=}\R(\A ,\B )\C +A^\alpha B^\beta C^\gamma\left(\frac{\partial\Gamma_{\alpha\gamma}^\delta}{\partial x_\beta}+\Gamma_{\alpha\gamma}^\eta\Gamma_{\beta\eta}^\delta\right)\frac{\partial}{\partial x_\delta},
\end{split}
\end{equation*}
then \eqref{nablatrecosenonnormale} follows. 
Moreover, \eqref{nablatrecose} follows by \eqref{nablatrecosenonnormale} and \eqref{normcoordcristo}. Finally, at $p$,
    \begin{equation*}
        \begin{split}
            \nabla_{\A}&\nabla_{\B}\nabla_{\C}\D=\nabla_{\A}\nabla_{\B}\left(\nabla_{\C}\D\right)\\
            &\overset{\eqref{nablatrecose}}{=}\R(\A ,\B )\nabla_{\C} \D+A^\alpha B^\beta \left(\nabla_{\C} \D\right)^\gamma\frac{\partial \Gamma_{\alpha\gamma}^\delta}{\partial x_\beta} \frac{\partial}{\partial x_\delta}+A^\alpha B^\beta\frac{\partial^2 \left(\nabla_{\C} \D\right)^\gamma}{\partial x_\alpha \partial x_\beta} \frac{\partial}{\partial x_\gamma}+A^\alpha\frac{\partial B^\beta}{\partial x_\alpha}\frac{\partial \left(\nabla_{\C} \D\right)^\gamma}{\partial x_\beta}\frac{\partial}{\partial x_\gamma}.
        \end{split}
    \end{equation*}
    Fix $\alpha,\beta,\gamma=1,\ldots,n+1$.
    Recall that
    \begin{equation*}
        \left(\nabla_{\C} \D\right)^\gamma=C^\delta\frac{\partial D^\gamma}{\partial x_\delta}+C^\delta D^\eta\Gamma^\gamma_{\delta\eta}.
    \end{equation*}
    Then, at $p$,
    \begin{equation*}
        \begin{split}
            \frac{\partial \left(\nabla_{\C} \D\right)^\gamma}{\partial x_\beta}&= \frac{\partial C^\delta}{\partial x_\beta}\frac{\partial D^\gamma}{\partial x_\delta}+C^\delta\frac{\partial^2 D^\gamma}{\partial x_\beta\partial x_\delta}+\frac{\partial C^\delta  }{\partial x_\beta}D^\eta\Gamma^\gamma_{\delta\eta}+C^\delta\frac{\partial D^\eta}{\partial x_\beta}\Gamma^\gamma_{\delta\eta}+C^\delta D^\eta\frac{\partial \Gamma^\gamma_{\delta \eta}}{\partial x_\beta}\\
            \overset{\eqref{normcoordcristo}}&{=}\frac{\partial C^\delta}{\partial x_\beta}\frac{\partial D^\gamma}{\partial x_\delta}+C^\delta\frac{\partial^2 D^\gamma}{\partial x_\beta\partial x_\delta}+C^\delta D^\eta\frac{\partial \Gamma^\gamma_{\delta \eta}}{\partial x_\beta},
        \end{split}
    \end{equation*}
    and moreover
    \begin{equation*}
        \begin{split}
              \frac{\partial^2 \left(\nabla_{\C} \D\right)^\gamma}{\partial x_\alpha\partial x_\beta}&\overset{\eqref{normcoordcristo}}{=}\frac{\partial ^2C^\delta}{\partial x_\alpha\partial x_\beta}\frac{\partial D^\gamma}{\partial x_\delta}+\frac{\partial C^\delta}{\partial x_\beta}\frac{\partial^2 D^\gamma}{\partial x_\alpha\partial x_\delta}+\frac{\partial C^\delta}{\partial x_\alpha}\frac{\partial^2 D^\gamma}{\partial x_\beta\partial x_\delta}+C^\delta\frac{\partial^3 D^\gamma}{\partial x_\alpha\partial x_\beta\partial x_\delta}\\
              &\quad+\frac{\partial C^\delta  }{\partial x_\beta}D^\eta\frac{\partial \Gamma^\gamma_{\delta\eta}}{\partial x_\alpha}+C^\delta\frac{\partial D^\eta}{\partial x_\beta}\frac{\partial \Gamma^\gamma_{\delta\eta}}{\partial x_\alpha}+\frac{\partial C^\delta }{\partial x_\alpha}D^\eta\frac{\partial \Gamma^\gamma_{\delta \eta}}{\partial x_\beta}+C^\delta \frac{\partial D^\eta}{\partial x_\alpha}\frac{\partial \Gamma^\gamma_{\delta \eta}}{\partial x_\beta}+C^\delta D^\eta\frac{\partial^2 \Gamma^\gamma_{\delta \eta}}{\partial x_\alpha\partial x_\beta}.
        \end{split}
    \end{equation*}
    Therefore
    \begin{equation*}
             \begin{split}          \nabla_{\A}&\nabla_{\B}\nabla_{\C}\D=\R(\A ,\B )\nabla_{\C} \D+A^\alpha B^\beta C^\eta\frac{\partial D^\gamma}{\partial x_\eta}\frac{\partial \Gamma_{\alpha\gamma}^\delta}{\partial x_\beta} \frac{\partial}{\partial x_\delta}\\
          &\quad+A^\alpha B^\beta\left(\frac{\partial ^2C^\delta}{\partial x_\alpha\partial x_\beta}\frac{\partial D^\gamma}{\partial x_\delta} +\frac{\partial C^\delta}{\partial x_\beta}\frac{\partial^2 D^\gamma}{\partial x_\alpha\partial x_\delta}+\frac{\partial C^\delta}{\partial x_\alpha}\frac{\partial^2 D^\gamma}{\partial x_\beta\partial x_\delta} +  C^\delta\frac{\partial^3 D^\gamma}{\partial x_\alpha\partial x_\beta\partial x_\delta} \right)\frac{\partial}{\partial x_\gamma}\\
          &\quad+\underbrace{A^\alpha B^\beta\left(\frac{\partial C^\delta  }{\partial x_\beta}D^\eta\frac{\partial \Gamma^\gamma_{\delta\eta}}{\partial x_\alpha} + C^\delta\frac{\partial D^\eta}{\partial x_\beta}\frac{\partial \Gamma^\gamma_{\delta\eta}}{\partial x_\alpha}\right)\frac{\partial}{\partial x_\gamma}}_{\mathrm{I}} +A^\alpha B^\beta\left(\frac{\partial C^\delta }{\partial x_\alpha}D^\eta\frac{\partial \Gamma^\gamma_{\delta \eta}}{\partial x_\beta} +C^\delta \frac{\partial D^\eta}{\partial x_\alpha}\frac{\partial \Gamma^\gamma_{\delta \eta}}{\partial x_\beta} \right)\frac{\partial}{\partial x_\gamma}\\
          &\quad+\underbrace{A^\alpha B^\beta C^\delta D^\eta\frac{\partial^2 \Gamma^\gamma_{\delta \eta}}{\partial x_\alpha\partial x_\beta} \frac{\partial}{\partial x_\gamma}}_{\mathrm{II}}+A^\alpha\frac{\partial B^\beta}{\partial x_\alpha}\left(\frac{\partial C^\delta}{\partial x_\beta}\frac{\partial D^\gamma}{\partial x_\delta}+ C^\delta\frac{\partial^2 D^\gamma}{\partial x_\beta\partial x_\delta}\right)\frac{\partial}{\partial x_\gamma}+ \underbrace{A^\alpha\frac{\partial B^\beta}{\partial x_\alpha}C^\delta D^\eta\frac{\partial \Gamma^\gamma_{\delta \eta}}{\partial x_\beta}\frac{\partial}{\partial x_\gamma}}_{\mathrm{III}}.
        \end{split}
    \end{equation*}
    Since 
    \begin{align*}
       \mathrm{I}\overset{\eqref{rieminnorm}}&{=}\R\left(\A ,\nabla_{\B} \C\right)\D+A^\alpha B^\beta\frac{\partial C^\delta}{\partial x_\beta}D^\eta\frac{\partial \Gamma^\gamma_{\alpha\eta}}{\partial x_\delta}\frac{\partial}{\partial x_\gamma}+\R\left(\A ,\C \right)\nabla_{\B} \D+A^\alpha B^\beta C^\delta\frac{\partial D^\eta}{\partial x_\beta}\frac{\partial\Gamma^\gamma_{\alpha\eta}}{\partial x_\delta}\frac{\partial}{\partial x_\gamma},\\
       \mathrm{II}&=A^\alpha B^\beta C^\delta D^\eta\frac{\partial^2 \Gamma^\gamma_{\delta \eta}}{\partial x_\beta\partial x_\alpha} \frac{\partial}{\partial x_\gamma}\overset{\eqref{deririeminnorm}}{=}\left(\nabla_{\B}\R\right)(\A,\C )\D+A^\alpha B^\beta C^\delta D^\eta\frac{\partial^2 \Gamma^\gamma_{\alpha \eta}}{\partial x_\beta\partial x_\delta} \frac{\partial}{\partial x_\gamma},\\
       \mathrm{III}\overset{\eqref{rieminnorm}}&{=}\R\left(\nabla_\A \B,\C\right)\D+A^\alpha\frac{\partial B^\beta}{\partial x_\alpha}C^\delta D^\eta\frac{\partial \Gamma^\gamma_{\beta \eta}}{\partial x_\delta}\frac{\partial}{\partial x_\gamma}
    \end{align*}
    then \eqref{nablaquattrocose} follows.
    \end{proof}    \subsection{Pointwise variations} In this section we deduce the pointwise evolution of the relevant geometric quantities. 
    The first lemma (cf. \cite[Lemma 1.13]{MR4676392} and \cite[Lemma 1.22]{MR4676392}) describes the behavior of $E$.
    \begin{lemma}\label{lemmaevolutiondie}
    Let $p\in S.$ Let $e\in T_p S$. Then 
    \begin{align}
        \left.\frac{D}{dt}\right|_{t=0}E(t)&=\nabla_{e}\X;\label{prounoaux}\\
        \left.\frac{D^2}{dt^2}\right|_{t=0}E(t)&=\nabla_{e}\Z+\R(\X,e)\X.\label{produeaux}
    \end{align}
\end{lemma}
By \Cref{lemmaevolutiondie}, we can compute the evolution of the Jacobian of $\Phi$. \Cref{lemmaderivatejacobiano} is surely well-known (cf. \cite[Theorem 1.11]{MR4676392} and \cite[Theorem 1.21]{MR4676392}). We include its proof for the sake of completeness.
\begin{lemma}\label{lemmaderivatejacobiano}
    Let $p\in S$. Then
    \begin{align}
        \left.\frac{d}{dt}\right|_{t=0}\jac\left(\Phi_t|_{S}\right)\left(p\right)&=\varphi H+\divv^S\X^T\label{derijacuno},\\
        \left.\frac{d^2}{dt^2}\right|_{t=0}\jac\left(\Phi_t|_{S}\right)\left(p\right)&=\left(\divv^S\X\right)^2-\left\langle\nabla^S\X,\left(\nabla^S\X\right)^t\right\rangle-\ric(\X,\X)-\R(\X^T,\n,\X^T,\n)\label{derijacdue}\\
        &\quad+\left|\nabla^S\varphi\right|^2-2h\left(\nabla^S\varphi,\X^T\right)+h^2\left(\X^T,\X^T\right)+\divv^S\Z\nonumber.
    \end{align}
    In particular, if $\Phi$ is a normal variation, 
    \begin{equation}\label{derjduenormal}
    \begin{split}
         \left.\frac{d^2}{dt^2}\right|_{t=0}\jac\left(\Phi_t|_{S}\right)\left(p\right)&=\varphi^2H^2-\varphi^2|h|^2-\varphi^2\ric(\n,\n)+\left|\nabla^S\varphi\right|^2+\divv^S\Z.
        \end{split}
    \end{equation}
\end{lemma}
\begin{proof}
    Since $\G(t)$ has full rank for every $t$ sufficiently small, Jacobi's formula grants that, for any such $t$,
    \begin{equation*}
        \frac{d \det \G(t)}{d t}=\det\G(t)\trace\left(\G(t)^{-1}\frac{d \G(t)}{dt}\right), 
    \end{equation*}
    whence
    \begin{equation*}
          \frac{d}{dt}\jac\left(\Phi_t|_{S}\right)\left(p\right)=\frac{1}{2}\jac\left(\Phi_t|_{S}\right)\left(p\right)\trace\left(\G(t)^{-1}\frac{d \G(t)}{dt}\right).
    \end{equation*}
   Recalling that $\G(0)_{ij}=\delta_{ij}$ for $i,j=1,\ldots,n$ by construction,
    \begin{align*}
         \left. \frac{d}{dt}\right|_{t=0}\jac\left(\Phi_t|_{S}\right)\left(p\right)=\frac{1}{2}\trace\left(\left.\frac{d }{dt}\right|_{t=0}\G(t)\right)\overset{\eqref{prounoaux}}{=}\sum_{i=1}^n\left\langle \nabla_{e_i}\X,e_i\right\rangle=\divv^S\X.
    \end{align*}
    Since 
    \begin{equation*}
        \divv^S\X=\divv^S\left(\varphi\n\right)+\divv^S\X^T=\varphi H+\divv^S\X^T,
    \end{equation*}
    \eqref{derijacuno} follows. Moreover, since \begin{equation}\label{derivatadimatriceinversa}
        \frac{d\G(t)^{-1}}{dt}=-\G(t)^{-1}\frac{d\G(t)}{dt}\G(t)^{-1},
    \end{equation}
    then
    \begin{equation*}
        \begin{split}
             \left. \frac{d^2}{dt^2}\right|_{t=0}\jac\left(\Phi_t|_{S}\right)\left(p\right)&\overset{\eqref{derijacuno}}{=}\left(\divv^S\X\right)^2+\underbrace{\frac{1}{2}\trace\left(-\left(\left.\frac{d }{dt}\right|_{t=0}\G(t)\right)^2\right)}_{\mathrm{I}}+\underbrace{\frac{1}{2}\trace\left(\left.\frac{d^2 }{dt^2}\right|_{t=0}\G(t)\right)}_{\mathrm{II}}.
        \end{split}
    \end{equation*}
    First,
    \begin{equation*}
        \begin{split}
            \mathrm{I}=
            \overset{\eqref{prounoaux}}{=}-\frac{1}{2}\sum_{i,j=1}^n\left(\left\langle\nabla_{e_i}\X,e_j\right\rangle+\left\langle \nabla_{e_j}\X,e_i\right\rangle\right)^2=-\sum_{i,j=1}^n\left\langle\nabla_{e_i}\X,e_j\right\rangle^2-\sum_{i,j=1}^n\left\langle\nabla_{e_i}\X,e_j\right\rangle\left\langle \nabla_{e_j}\X,e_i\right\rangle.
        \end{split}
    \end{equation*}
    Moreover,
    \begin{equation*}
        \begin{split}
            \mathrm{II}&\overset{\eqref{prounoaux},\eqref{produeaux}}{=}\sum_{i=1}^n\left\langle \nabla_{e_i}\X,\nabla_{e_i}\X\right\rangle+\sum_{i=1}^n\left\langle\nabla_{e_i}\Z+\R(\X,e_i)\X,e_i\right\rangle\\
            &=\sum_{i,j=1}^n\left\langle \nabla_{e_i}\X,e_j\right\rangle^2+\sum_{i=1}^n\left\langle \nabla_{e_i}\X,\n\right\rangle^2+\divv^S\Z-\ric(\X,\X)-\R(\X,\n,\X,\n)\\
            &=\sum_{i,j=1}^n\left\langle \nabla_{e_i}\X,e_j\right\rangle^2+\sum_{i=1}^n\left(e_i(\varphi)-\left\langle A\left(\X^T\right),e_i\right\rangle\right)^2+\divv^S\Z-\ric(\X,\X)-\R(\X^T,\n,\X^T,\n)\\
            &=\sum_{i,j=1}^n\left\langle \nabla_{e_i}\X,e_j\right\rangle^2+\left|\nabla^S\varphi\right|^2-2h\left(\nabla^S\varphi,\X^T\right)+h^2\left(\X^T,\X^T\right)+\divv^S\Z-\ric(\X,\X)-\R(\X^T,\n,\X^T,\n).\\
        \end{split}
    \end{equation*}
    In this way, \eqref{derijacdue} follows. Finally, if $\Phi$ is a normal variation, then $\X=\varphi\n$ and $\X^T=0$. In particular, $\nabla^S\X=\left(\nabla^S\X\right)^t=\varphi h$, and \eqref{derjduenormal} by \eqref{derijacdue}.
\end{proof}
Define $V(t)\coloneqq \N(t,\beta(t))$. By definition, $V$ is a vector field along $\beta$, and in particular $V(t)$ is normal to $S_t$ at $\beta(t)$ for every $t\in I$. $V$ evolves as follows.
\begin{lemma}
    Let $p\in S$. 
    Let $e_1,\ldots,e_n$ be an orthonormal basis of $T_p S$. Then
    \begin{align}
        &\,V'(p)\coloneqq \left.\frac{D}{dt}\right|_{t=0}V(t)=-\nabla^S\varphi+ A\left( \X^T\right) ; \label{lemma125ritore}\\
          &V''(p)\coloneqq \left.\frac{D^2}{dt^2}\right|_{t=0}V(t)=\left(-\left|\nabla^S\varphi\right|^2-\left|A\left(\X ^T\right)\right|^2+2 h\left( \X^T,\nabla^S\varphi\right)\right)\n\label{secondordernormalbehave}\\
         &\quad+\sum_{j=1}^n\left(-\left \langle \n ,\nabla_{e_j}\Z+\R(\X,e_j)\ \X^T\right\rangle+2\left\langle \nabla^S\varphi- A\left( \X^T\right) ,\varphi A(e_j)+\nabla_{e_j}\ \X^T\right\rangle\right)e_j.\nonumber
    \end{align}
    Moreover, fix local coordinates $x_1,\ldots,x_{n+1}$ in $M$ near $p$, and set 
    \begin{equation*}
        \n'(q)=\frac{\partial c^k}{\partial t}(0,q)\left.\frac{\partial}{\partial x_k}\right|_q,\qquad \n''(q)=\frac{\partial^2 c^k}{\partial t^2}(0,q)\left.\frac{\partial}{\partial x_k}\right|_q\qquad\text{locally around $p$.}
    \end{equation*}
 Then, when $q\in S$ is close to $p$,
 \begin{align}
        \n'(q)&=V'(q)-\nabla_\X \n, \label{piccolopezzodercovnorm}\\
        \n''(q)&    =V''(q)-2\nabla_\X\n'-\nabla_{\X'}\n-\nabla_\X\nabla_\X \n\label{pezzocondtquadrockappa}
 \end{align}
\end{lemma}
\begin{proof}
   First,  \eqref{lemma125ritore} follows by \cite[Lemma 1.25]{MR4676392}. We prove \eqref{secondordernormalbehave}. As $|V(t)|\equiv 1$,
then
\begin{equation}\label{solitacosa}
      \left\langle\frac{D}{dt}V(t),V(t)\right\rangle\equiv 0.
\end{equation}
In particular, by \eqref{lemma125ritore} and \eqref{solitacosa},
\begin{equation*}
    \begin{split}
         \left\langle\left.\frac{D^2}{dt^2}\right|_{t=0}V(t),\left .\n\right|_p\right\rangle&=\left.\frac{d}{dt}\right|_{t=0}\left\langle\frac{D}{dt}V(t),V(t)\right\rangle-\left|V'(p)\right|^2
         =-\left|\nabla^S\varphi\right|^2-\left|A\left(\X ^T\right)\right|^2+2 h\left( \X^T,\nabla^S\varphi\right).
    \end{split}
\end{equation*}
Let $f\in T_p S$. Set $F(t)\coloneqq \left(d\Phi_t\right)|_p(f)$. Recall that $F(t)$ is a vector field along $\beta$, and $F(t)\in T_{\beta(t)}S_t$. Therefore
\begin{equation*}
    \begin{split}
         \left\langle\left.\frac{D^2}{dt^2}\right|_{t=0}V(t),f\right\rangle&=\left.\frac{d}{dt}\right|_{t=0}\left\langle\frac{D}{dt}V(t),F(t)\right\rangle-\left\langle\left.\frac{D}{dt}\right|_{t=0}V(t),\left.\frac{D}{dt}\right|_{t=0}F(t)\right\rangle\\   
         &=-\left.\frac{d}{dt}\right|_{t=0}\left\langle V(t),\frac{D}{dt}F(t)\right\rangle-\left\langle\left.\frac{D}{dt}\right|_{t=0}V(t),\left.\frac{D}{dt}\right|_{t=0}F(t)\right\rangle\\  
         &=-\left\langle \n,\left.\frac{D^2}{dt^2}\right|_{t=0}F(t)\right\rangle-2\left\langle\left.\frac{D}{dt}\right|_{t=0}V(t),\left.\frac{D}{dt}\right|_{t=0}F(t)\right\rangle\\  
         \overset{\eqref{prounoaux},\eqref{produeaux},\eqref{lemma125ritore}}&{=}-\left \langle \n ,\nabla_{f}\left(\nabla_\X\X+\X'\right)+\R(\X,f)\X\right\rangle+2\left\langle \nabla^S\varphi- A\left( \X^T\right) ,\nabla_f \X\right\rangle\\
         &=-\left \langle \n,\nabla_{f}\Z+\R(\X,f)\ \X^T\right\rangle+2\left\langle \nabla^S\varphi- A\left(\ \X^T\right) ,\varphi A(f)+\nabla_f\ \X^T\right\rangle.
    \end{split}
\end{equation*}
In this way, \eqref{secondordernormalbehave} follows. Next we prove \eqref{piccolopezzodercovnorm}. Indeed, if $q\in S$ is close to $p$,
    \begin{equation*}
    \begin{split}
          \left.\frac{D}{dt}\right|_{t=0}\N(t,\beta_q(t))&\overset{\eqref{comefarecovdev}}{=}\frac{\partial c^k}{\partial t}(0,q)\left.\frac{\partial}{\partial x_k}\right|_q+X^i(q)\frac{\partial N^k}{\partial x_i}(q)\left.\frac{\partial}{\partial x_k}\right|_q+X^i(q)N^k(q)\Gamma_{ik}^l(q)\left.\frac{\partial}{\partial x_l}\right|_q=\n'+\nabla_\X \n.
    \end{split}
    \end{equation*}
 Finally, for $q$ close to $p$,
\begin{equation*}
\begin{split}
    & \left.\frac{D^2}{dt^2}\right|_{t=0}\mathcal N(t,\beta_q(t))=\left.\frac{D^2}{dt^2}\right|_{t=0}\left(c^k(t,\beta_q(t))\left.\frac{\partial}{\partial x_k}\right|_{\beta_q(t)}\right)\\
    &\quad=\left.\frac{D}{dt}\right|_{t=0}\left(\frac{\partial c^k}{\partial t}(t,\beta_q(t))\left.\frac{\partial}{\partial x_k}\right|_{\beta_q(t)}+a^i(t,\beta_q(t))\frac{\partial c^k}{\partial x_i}(t,\beta_q(t))\left.\frac{\partial}{\partial x_k}\right|_{\beta_q(t)}\right)\\
    &\qquad+\left.\frac{D}{dt}\right|_{t=0}\left(a^i(t,\beta_q(t))c^k(t,\beta_q(t))\Gamma_{ik}^l(\beta_q(t))\left.\frac{\partial}{\partial x_l}\right|_{\beta_q(t)}\right)\\
    &\quad=\n''(q)+2X^i(q)\frac{\partial^2 c^k}{\partial x_i\partial t}(0,q)\left.\frac{\partial}{\partial x_k}\right|_{q}+2X^i(q)\frac{\partial c^k}{\partial t}(0,q)\Gamma_{ik}^l(q)\left.\frac{\partial}{\partial x_l}\right|_{q}+(X')^i(q)\frac{\partial N^k}{\partial x_i}(q)\left.\frac{\partial}{\partial x_k}\right|_{q}\\
    &\qquad+X^l(q)\frac{\partial X^i}{\partial x_l}(q)\frac{\partial N^k}{\partial x_i}(q)\left.\frac{\partial}{\partial x_k}\right|_{q}+X^i(q)X^l(q)\frac{\partial ^2N^k}{\partial x_l\partial x_i}(q)\left.\frac{\partial}{\partial x_k}\right|_{q}+2X^i(q)X^l(q)\frac{\partial N^k}{\partial x_i}(q)\Gamma_{lk}^m(q)\left.\frac{\partial}{\partial x_m}\right|_{q}\\
    &\qquad+\left(X'\right)^i(q)N^k(q)\Gamma_{ik}^l(q)\left.\frac{\partial}{\partial x_l}\right|_{q}+X^m(q)\frac{\partial X^i}{\partial x_m}(q)N^k(q)\Gamma_{ik}^l(q)\left.\frac{\partial}{\partial x_l}\right|_{q}\\
    &\qquad+X^i(q)X^m(q)N^k(q)\frac{\partial \Gamma_{ik}^l}{\partial x_m}(q)\left.\frac{\partial}{\partial x_l}\right|_{q}+X^i(q)X^m(q)N^k(q)\Gamma^l_{ik}(q)\Gamma^s_{ml}\left.\frac{\partial}{\partial x_s}\right|_q,
\end{split}
\end{equation*}
whence \eqref{pezzocondtquadrockappa} follows by \eqref{nablatrecosenonnormale}.
\end{proof}
Next, we describe the evolution of the shape operator. The following proof relies crucially on \Cref{fondamentalecomplemmavariazioniriemannianejhrnfi4rfjnfin}.
\begin{lemma}\label{lemmaevoluzioneshapeoperatoririem3858943}
    Let $p\in S.$ Let $e\in T_p S$. Then
    \begin{align}
          \left.\frac{D}{dt}\right|_{t=0}\left(\left.\nabla_{E(t)}\N(t,\cdot)\right|_{\beta(t)}\right)&=\nabla_eV'+\R(\X,e)\n ,\label{protreaux}\\
            \left.\frac{D^2}{dt^2}\right|_{t=0}\left(\left.\nabla_{E(t)}\N(t,\cdot)\right|_{\beta(t)}\right)&=\nabla_e V''+\R\left(\X,\nabla_e\X\right)\n+2\R(\X,e)V'+\R\left(\Z,e\right)\n+\left(\nabla_\X \R\right)(\X,e)\n\label{proquattroaux},
    \end{align}
    where $V'$ and $V''$ are given respectively by \eqref{lemma125ritore} and \eqref{secondordernormalbehave}.
\end{lemma}
\begin{proof}

Consider local normal coordinates $x_1,\ldots, x_{n+1}$ centered at $p$.  Since $E(t)\in T_{\beta(t)}S_t,$ then
    \begin{equation*}
        \begin{split}
            \left.\nabla_{E(t)}\N(t,\cdot)\right|_{\beta(t)}&=b^j(t)\frac{\partial c^k}{\partial x_j}(t,\beta(t))\left.\frac{\partial}{\partial x_k}\right|_{\beta(t)}+b^j(t)c^k(t,\beta(t))\Gamma_{jk}^l(\beta(t))\left.\frac{\partial}{\partial x_l}\right|_{\beta(t)}.
        \end{split}
    \end{equation*}
    Therefore
    \begin{equation}\label{derivataprimashape}
        \begin{split}
           &\frac{D}{dt}\Big(\left.\nabla_{E(t)}\N(t,\cdot)\right|_{\beta(t)}\Big)=\dot b^j(t)\frac{\partial c^k}{\partial x_j}(t,\beta(t))\left.\frac{\partial}{\partial x_k}\right|_{\beta(t)}+b^j(t)\frac{\partial^2 c^k}{\partial t\partial x_j}(t,\beta(t))\left.\frac{\partial}{\partial x_k}\right|_{\beta(t)}\\
           &\quad+a^i(t,\beta(t))b^j(t)\frac{\partial ^2c^k}{\partial x_i\partial x_j}(t,\beta(t))\left.\frac{\partial}{\partial x_k}\right|_{\beta(t)}+a^i(t,\beta(t)) b^j(t)\frac{\partial c^k}{\partial x_j}(t,\beta(t))\Gamma_{ik}^l(\beta(t))\left.\frac{\partial}{\partial x_l}\right|_{\beta(t)}\\
           &\quad+\dot b^j(t)c^k(t,\beta(t))\Gamma_{jk}^l(\beta(t))\left.\frac{\partial}{\partial x_l}\right|_{\beta(t)}+b^j(t)\frac{\partial c^k}{\partial t}(t,\beta(t))\Gamma_{jk}^l(\beta(t))\left.\frac{\partial}{\partial x_l}\right|_{\beta(t)}\\
           &\quad+a^i(t,\beta(t))b^j(t)\frac{\partial c^k}{\partial x_i}(t,\beta(t))\Gamma_{jk}^l(\beta(t))\left.\frac{\partial}{\partial x_l}\right|_{\beta(t)}+a^i(t,\beta(t))b^j(t)c^k(t,\beta(t))\frac{\partial\Gamma_{jk}^l}{\partial x_i}(\beta(t))\left.\frac{\partial}{\partial x_l}\right|_{\beta(t)}\\
           &\quad+a^i(t,\beta(t))b^j(t)c^k(t,\beta(t))\Gamma_{jk}^l(\beta(t))\Gamma_{i l}^m(\beta(t))\left.\frac{\partial}{\partial x_m}\right|_{\beta(t)}.
        \end{split}
    \end{equation}
    In particular, by \eqref{normcoordcristo},
      \begin{equation*}
        \begin{split}
           \left.\frac{D}{dt}\right|_{t=0}\Big(\left.\nabla_{E(t)}\N(t,\cdot)\right|_{\beta(t)}\Big)&=\dot b^j(0)\frac{\partial N^k}{\partial x_j}(p)\left.\frac{\partial}{\partial x_k}\right|_{p}+e^j\frac{\partial^2 c^k}{\partial t\partial x_j}(0,p)\left.\frac{\partial}{\partial x_k}\right|_{p}        \\
&\quad+X^i(p)e^j\frac{\partial ^2N^k}{\partial x_i\partial x_j}(p)\left.\frac{\partial}{\partial x_k}\right|_{p}+X^i(p)e^jN^k(p)\frac{\partial\Gamma_{jk}^l}{\partial x_i}(p)\left.\frac{\partial}{\partial x_l}\right|_{p}.
        \end{split}
    \end{equation*}
    Observe that 
    \begin{equation}\label{comeebdotzero}
       \dot b^i(0)\left.\frac{\partial}{\partial x_i}\right|_{p}\overset{\eqref{normcoordcristo}}{=} \left.\frac{D}{dt}\right|_{t=0}E(t)\overset{\eqref{prounoaux}}{=}\nabla_e \X\overset{\eqref{normcoordcristo}}{=}\frac{\partial X^i}{\partial x_j}(p)e^j\left.\frac{\partial}{\partial x_i}\right|_p,
    \end{equation}
   whence   \eqref{protreaux} follows by
    \begin{equation*}
        \begin{split}
           \left.\frac{D}{dt}\right|_{t=0}\Big(\left.\nabla_{E(t)}\N(t,\cdot)\right|_{\beta(t)}\Big)&=\frac{\partial X^i}{\partial x_j}(p)e^j\frac{\partial N^k}{\partial x_i}(p)\left.\frac{\partial}{\partial x_k}\right|_{p}+e^j\frac{\partial^2 c^k}{\partial x_j\partial t}(0,p)\left.\frac{\partial}{\partial x_k}\right|_{p}        \\
&\quad+X^i(p)e^j\frac{\partial ^2N^k}{\partial x_i\partial x_j}(p)\left.\frac{\partial}{\partial x_k}\right|_{p}+X^i(p)e^jN^k(p)\frac{\partial\Gamma_{jk}^l}{\partial x_i}(p)\left.\frac{\partial}{\partial x_l}\right|_{p}\\
\overset{\eqref{nablatrecose}}&{=}e^j\frac{\partial^2 c^k}{\partial x_j\partial t}(0,p)\left.\frac{\partial}{\partial x_k}\right|_{p} +\nabla_e\nabla_\X \n-\R(e,\X)\n\\
\overset{\eqref{normcoordcristo}}&{=}\nabla_e\n'+\nabla_e\nabla_\X \n-\R(e,\X)\n\\
\overset{\eqref{piccolopezzodercovnorm}}&{=}\nabla_eV'-\R(e,\X)\n.
        \end{split}
    \end{equation*}  
 We prove \eqref{proquattroaux}. By \eqref{derivataprimashape} and \eqref{normcoordcristo},
    \begin{equation*}
        \begin{split}
             &\left.\frac{D^2}{dt^2}\right|_{t=0}\Big(\left.\nabla_{E(t)}\N(t,\cdot)\right|_{\beta(t)}\Big)=\ddot b^j(0)\frac{\partial N^k}{\partial x_j}(p)\left.\frac{\partial}{\partial x_k}\right|_{p}+2\dot b^j(0)\frac{\partial^2 c^k}{\partial x_j\partial t}(0,p)\left.\frac{\partial}{\partial x_k}\right|_{p}\\
             &\quad+2X^i(p)\dot b^j(0)\frac{\partial^2 N^k}{\partial x_i\partial x_j}(p)\left.\frac{\partial}{\partial x_k}\right|_{p}+e^j\frac{\partial^3 c^k}{\partial x_j\partial t^2}(0,p)\left.\frac{\partial}{\partial x_k}\right|_{p}+2X^i(p)e^j\frac{\partial^3 c^k}{\partial x_i\partial x_j\partial t}(0,p)\left.\frac{\partial}{\partial x_k}\right|_{p}\\
             &\quad +(X')^i(p)e^j\frac{\partial ^2N^k}{\partial x_i\partial x_j}(p)\left.\frac{\partial}{\partial x_k}\right|_{p}+X^l(p)\frac{\partial X^i}{\partial x_l}(p)e^j\frac{\partial ^2N^k}{\partial x_i\partial x_j}(p)\left.\frac{\partial}{\partial x_k}\right|_{p}+X^i(p)X^l(p)e^j\frac{\partial ^3N^k}{\partial x_l\partial x_i\partial x_j}(p)\left.\frac{\partial}{\partial x_k}\right|_{p}\\
             &\quad+X^i(p)X^m(p) e^j\frac{\partial N^k}{\partial x_j}(p)\frac{\partial \Gamma_{ik}^l}{\partial x_m}(p)\left.\frac{\partial}{\partial x_l}\right|_{p}+2X^i(p)\dot b^j(0)N^k(p)\frac{\partial \Gamma_{jk}^l}{\partial x_i}(p)\left.\frac{\partial}{\partial x_l}\right|_{p}\\
             &\quad +2X^i(p) e^j\frac{\partial c^k}{\partial t}(0,p)\frac{\partial \Gamma_{jk}^l}{\partial x_i}(p)\left.\frac{\partial}{\partial x_l}\right|_{p}+2X^i(p)X^m(p)e^j\frac{\partial N^k}{\partial x_i}(p)\frac{\partial \Gamma_{jk}^l}{\partial x_m}(p)\left.\frac{\partial}{\partial x_l}\right|_{p}\\
            &\quad+(X')^i(p)e^jN^k(p)\frac{\partial\Gamma_{jk}^l}{\partial x_i}(p)\left.\frac{\partial}{\partial x_l}\right|_{p}+X^m(p)\frac{\partial X^i}{\partial x_m}(p)e^jN^k(p)\frac{\partial\Gamma_{jk}^l}{\partial x_i}(p)\left.\frac{\partial}{\partial x_l}\right|_{p}\\
            &\quad+X^i(p)X^m(p)e^jN^k(p)\frac{\partial^2\Gamma_{jk}^l}{\partial x_m\partial x_i}(p)\left.\frac{\partial}{\partial x_l}\right|_{p}.
        \end{split}
    \end{equation*}
       Notice that 
       \begin{equation*}
       \begin{split}
            \left.\frac{D^2}{dt^2}\right|_{t=0}E(t)\overset{\eqref{comefarecovdev}}&{=} \left.\frac{D}{dt}\right|_{t=0}\left(\dot b^j(t)\left.\frac{\partial}{\partial x_j}\right|_{\beta(t)}+a^i(t,\beta(t))b^j(t)\Gamma_{ij}^l(\beta(t))\left.\frac{\partial}{\partial x_l}\right|_{\beta(t)}\right)\\
            \overset{\eqref{normcoordcristo}}&{=}\ddot b^j(0)\left.\frac{\partial}{\partial x_j}\right|_{p}+X^i(p)X^m(p)e^j\frac{\partial \Gamma_{ij}^l}{\partial x_m}(p)\left.\frac{\partial}{\partial x_l}\right|_{p}\\
            &=\left(\ddot b^j(0)+X^i(p)X^m(p)e^l\frac{\partial \Gamma_{il}^j}{\partial x_m}(p)\right) \left.\frac{\partial}{\partial x_j}\right|_{p}.
       \end{split}
       \end{equation*}
       Therefore, by \eqref{normcoordcristo}, \eqref{nablatrecose}, \eqref{produeaux} and for any $j=1,\ldots,n+1$,      \begin{equation}\label{dersecetcond}
       \begin{split}
              &\ddot b^j(0)=\left(\nabla_{e}\nabla_\X\X\right)^j(p)+\left(\nabla_e \X'\right)^j(p)+\left(\R(\X,e)\X\right)^j(p)-X^i(p)X^m(p)e^l\frac{\partial \Gamma_{il}^j}{\partial x_m}(p)\\
              &\quad=e^lX^i(p)X^m(p)\frac{\partial\Gamma_{lm}^j}{\partial x_i}+ e^l X^i(p)\frac{\partial^2 X^j}{\partial x_l \partial x_i}(p)+e^l\frac{\partial X^i}{\partial x_l}(p)\frac{\partial X^j}{\partial x_i}(p)+e^l\frac{\partial (X')^j}{\partial x_l}(p)-X^i(p)X^m(p)e^l\frac{\partial\Gamma_{il}^j}{\partial x_m}(p)\\
              &\quad= e^l X^i(p)\frac{\partial^2 X^j}{\partial x_l \partial x_i}(p)+e^l\frac{\partial X^i}{\partial x_l}(p)\frac{\partial X^j}{\partial x_i}(p)+e^l\frac{\partial (X')^j}{\partial x_l}(p),
       \end{split}
       \end{equation}
       where the last equality follows by the symmetry of $\nabla$. 
 Then, by  \eqref{comeebdotzero} and \eqref{dersecetcond},
    \begin{equation*}
        \begin{split}
             &\left.\frac{D^2}{dt^2}\right|_{t=0}\Big(\left.\nabla_{E(t)}\N(t,\cdot)\right|_{\beta(t)}\Big)=e^l X^i(p)\frac{\partial^2 X^j}{\partial x_l \partial x_i}(p)\frac{\partial N^k}{\partial x_j}(p)\left.\frac{\partial}{\partial x_k}\right|_{p}+e^l\frac{\partial X^i}{\partial x_l}(p)\frac{\partial X^j}{\partial x_i}(p)\frac{\partial N^k}{\partial x_j}(p)\left.\frac{\partial}{\partial x_k}\right|_{p}\\
             &\quad+e^l\frac{\partial (X')^j}{\partial x_l}(p)\frac{\partial N^k}{\partial x_j}(p)\left.\frac{\partial}{\partial x_k}\right|_{p}+2\frac{\partial X^i}{\partial x_j}(p)e^j\frac{\partial^2 c^k}{\partial x_i\partial t}(0,p)\left.\frac{\partial}{\partial x_k}\right|_{p}\\
             &\quad+2X^i(p)\frac{\partial X^m}{\partial x_j}(p)e^j\frac{\partial^2 N^k}{\partial x_i\partial x_m}(p)\left.\frac{\partial}{\partial x_k}\right|_{p}+e^j\frac{\partial^3 c^k}{\partial x_j\partial t^2}(0,p)\left.\frac{\partial}{\partial x_k}\right|_{p}+2X^i(p)e^j\frac{\partial^3 c^k}{\partial x_i\partial x_j\partial t}(0,p)\left.\frac{\partial}{\partial x_k}\right|_{p}\\
             &\quad +(X')^i(p)e^j\frac{\partial ^2N^k}{\partial x_i\partial x_j}(p)\left.\frac{\partial}{\partial x_k}\right|_{p}+X^l(p)\frac{\partial X^i}{\partial x_l}(p)e^j\frac{\partial ^2N^k}{\partial x_i\partial x_j}(p)\left.\frac{\partial}{\partial x_k}\right|_{p}+X^i(p)X^l(p)e^j\frac{\partial ^3N^k}{\partial x_l\partial x_i\partial x_j}(p)\left.\frac{\partial}{\partial x_k}\right|_{p}\\
             &\quad+X^i(p)X^m(p) e^j\frac{\partial N^k}{\partial x_j}(p)\frac{\partial \Gamma_{ik}^l}{\partial x_m}(p)\left.\frac{\partial}{\partial x_l}\right|_{p}+2X^i(p)\frac{\partial X^m}{\partial x_j}(p)e^jN^k(p)\frac{\partial \Gamma_{mk}^l}{\partial x_i}(p)\left.\frac{\partial}{\partial x_l}\right|_{p}\\
             &\quad +2X^i(p) e^j\frac{\partial c^k}{\partial t}(0,p)\frac{\partial \Gamma_{jk}^l}{\partial x_i}(p)\left.\frac{\partial}{\partial x_l}\right|_{p}+2X^i(p)X^m(p)e^j\frac{\partial N^k}{\partial x_i}(p)\frac{\partial \Gamma_{jk}^l}{\partial x_m}(p)\left.\frac{\partial}{\partial x_l}\right|_{p}\\
            &\quad+(X')^i(p)e^jN^k(p)\frac{\partial\Gamma_{jk}^l}{\partial x_i}(p)\left.\frac{\partial}{\partial x_l}\right|_{p}+X^m(p)\frac{\partial X^i}{\partial x_m}(p)e^jN^k(p)\frac{\partial\Gamma_{jk}^l}{\partial x_i}(p)\left.\frac{\partial}{\partial x_l}\right|_{p}\\
            &\quad+X^i(p)X^m(p)e^jN^k(p)\frac{\partial^2\Gamma_{jk}^l}{\partial x_m\partial x_i}(p)\left.\frac{\partial}{\partial x_l}\right|_{p}.
        \end{split}
    \end{equation*}
    A careful comparison between the above expression and \eqref{nablatrecose} and \eqref{nablaquattrocose} grants that
    \begin{equation*}
        \begin{split}
            \left.\frac{D^2}{dt^2}\right|_{t=0}&\Big(\left.\nabla_{E(t)}\N(t,\cdot)\right|_{\beta(t)}\Big)=\nabla _e\nabla_\X\nabla_\X \n-\left(\nabla_\X \R\right)(e,\X)\n-2 \R(e,\X)\nabla_\X \n-\R\left(e,\nabla_\X \X\right)\n\\
            &\quad-\R\left(\nabla_e\X,\X\right)\n+2\nabla_e\nabla_\X\n'-2\R(e,\X)\n'+\nabla_e \n''+\nabla_e\nabla_{\X'}\n-\R(e,\X')\n.
        \end{split}
    \end{equation*}
    In addition,
    \begin{equation*}
        -2 \R(e,\X)\n'-2\R(e,\X)\nabla_\X \n\overset{\eqref{piccolopezzodercovnorm}}{=}-2\R(e,\X)V'
    \end{equation*}
    and 
    \begin{equation*}
        \nabla_e \n''+2\nabla_e\nabla_\X \n'+\nabla_e \nabla_{\X'}\n+\nabla_e\nabla_\X\nabla_\X \n\overset{\eqref{pezzocondtquadrockappa}}{=}\nabla_e V'',
    \end{equation*}
    whence
  \begin{equation*}
        \begin{split}
            \left.\frac{D^2}{dt^2}\right|_{t=0}\Big(\left.\nabla_{E(t)}\N(t,\cdot)\right|_{\beta(t)}\Big)&=-\left(\nabla_\X \R\right)(e,\X)\n-\R\left(e,\Z\right)\n+\nabla_e V''-2\R(e,\X)V'-\R\left(\nabla_e\X,\X\right)\n,
        \end{split}
    \end{equation*}
    and \eqref{proquattroaux} follows.
\end{proof}
\Cref{lemmaevoluzioneshapeoperatoririem3858943} allows to compute the pointwise evolution of the mean curvature at first and second-order.
\begin{lemma}\label{lemmavariazioneprimacurvaturaeseconda}
    Let $p\in S$. Then 
    \begin{align}
        \left.\frac{d}{dt}\right|_{t=0}H_t\left(\Phi_t(p)\right)&=-\Delta^S\varphi-\varphi\left(|h|^2+\ric(\n,\n)\right)+ \X^T H\label{derprimhinstatement}\\   
         \left.\frac{d^2}{dt^2}\right|_{t=0}H_t\left(\Phi_t(p)\right)&=2\left\langle h,\left(\nabla^S X\right)^2\right\rangle-2\left\langle \nabla^S X,\left(\nabla^SV'\right)^t\right\rangle-\left\langle h,\j(\X,\X)\right\rangle-\left\langle\nabla^S \X,\j(\n,\X)\right\rangle\label{dersechinstatement}\\
         &\quad+\left(-\left|\nabla^S\varphi\right|^2+2h\left(\nabla^S\varphi, \X^T\right)-h^2\left( \X^T, \X^T\right)\right)H+\divv^S\left( \tilde V''\right)^T\nonumber\\
         &\quad-2\ric(\X,V')-\R\left(\X^T,\n,V',\n\right)-\left(\nabla_\X \ric\right)(\X,\n)\nonumber\\
         &\quad-\Delta^S\left \langle\Z,\n\right\rangle-\left\langle\Z,\n\right\rangle\left(|h|^2+\ric(\n,\n)\right)+\Z^T H,\nonumber
    \end{align}
    where 
    \begin{equation*}
        \left(\tilde V''\right)^T\coloneqq \sum_{j=1}^n\left(\R(\X,e_j,\n,\X^T)+2\left\langle \nabla^S\varphi- A\left( \X^T\right) ,\varphi A(e_j)+\nabla_{e_j}\ \X^T\right\rangle\right)e_j
    \end{equation*}
    In particular, if $\Phi$ is a normal variation,
    \begin{equation}\label{secvarhnorm}
        \begin{split}
           \left.\frac{d^2}{dt^2}\right|_{t=0}H_t\left(\Phi_t(p)\right)&=\divv^S \left( A\left(\nabla^S\varphi^2\right)\right)+2\varphi\divv^S A\left(\nabla^S\varphi\right)-2\varphi\left\langle \nabla^S H,\nabla^S\varphi\right\rangle-\left|\nabla^S\varphi\right|^2H\\
                &\quad+2\varphi^2\trace\left(h^3\right)-2\varphi^2\left\langle h,\j(\n,\n)\right\rangle-\varphi^2\left(\nabla_\n \ric\right)(\n,\n)\\
         &\quad-\Delta^S\left \langle\Z,\n\right\rangle-\left\langle\Z,\n\right\rangle\left(|h|^2+\ric(\n,\n)\right)+\Z^T H.
        \end{split}
    \end{equation}
\end{lemma}
\begin{proof}
Formula \eqref{derprimhinstatement} is well-known (cf. \cite[Lemma 1.26]{MR4676392}). We prove it for the sake of completeness. 
    Let $e_1,\ldots,e_n$ be any orthonormal basis of $T_p S$. 
    For $j=1,\ldots,n$, let $E_j(t)$ be as in \eqref{estensionevettore}. Define $\S(t)$ by
    \begin{equation*}
       \S(t)_{ij}=\left\langle\left.\nabla_{E_i(t)}\N(t,\cdot)\right|_{\beta(t)},E_j(t)\right\rangle,\qquad i,j=1,\ldots,n.
    \end{equation*}
    Then $\S(t)$ is symmetric, and $       H(\Phi_t(p))=\trace\left(\G(t)^{-1}\S(t)\right),$
   where $\G$ is defined in \eqref{grammatrix}. 
   Recalling \eqref{derivatadimatriceinversa},
    \begin{equation*}
        \frac{d}{dt}H_t(\Phi_t(p))=\trace\left(-\G(t)^{-1}\frac{d\G(t)}{dt}\G(t)^{-1}\S(t)+\G(t)^{-1}\frac{d\S(t)}{dt}\right).
    \end{equation*}
Since $e_1,\ldots,e_n$ is orthonormal, then $\G(0)_{ij}=\delta_{ij}$ for $i,j=1,\ldots,n$, whence
\begin{align}
     \left.\frac{d}{dt}\right|_{t=0}H_t(\Phi_t(p))&=\trace\left(-\left(\left.\frac{d}{dt}\right|_{t=0}\G(t)\right)\S(0)+\left.\frac{d}{dt}\right|_{t=0}\S(t)\right),\label{unatracciaderiprimh}\\
      \left.\frac{d^2}{dt^2}\right|_{t=0}H_t(\Phi_t(p))&=\underbrace{\trace\left(2\left(\left.\frac{d}{dt}\right|_{t=0}\G(t)\right)^2\S(0)\right)}_{\mathrm{I}}+\underbrace{\trace\left(-\left(\left.\frac{d^2}{dt^2}\right|_{t=0}\G(t)\right)\S(0)\right)}_{\mathrm{II}}\label{quattrotraccedersech}\\
      &\quad+\underbrace{\trace\left(-2\left(\left.\frac{d}{dt}\right|_{t=0}\G(t)\right)\left(\left.\frac{d}{dt}\right|_{t=0}\S(t)\right)\right)}_{\mathrm{III}}+\underbrace{\trace\left(\left.\frac{d^2}{dt^2}\right|_{t=0}\S(t)\right)}_{\mathrm{IV}}\nonumber.
\end{align}
First, by \eqref{unatracciaderiprimh},
\begin{equation*}
    \begin{split}
         &\left.\frac{d}{dt}\right|_{t=0}H_t(\Phi_t(p))=-\sum_{i,j=1}^n\left(\left.\frac{d}{dt}\right|_{t=0}\G(t)_{ij}\right)\S(0)_{ij}+\sum_{i=1}^n\left.\frac{d}{dt}\right|_{t=0}\S(t)_{ii}\\
         &\qquad\overset{\eqref{prounoaux},\eqref{protreaux}}{=}-\sum_{i,j=1}^nh_{ij}\left(\left\langle\nabla_{e_i} \X,e_j\right\rangle+\left\langle \nabla_{e_j} \X,e_i\right\rangle\right)+\sum_{i=1}^n\left\langle \nabla_{e_i}V',e_i\right\rangle-\ric(\X,\n)+\sum_{i=1}^n\left\langle\nabla_{e_i}\n ,\nabla_{e_i}\X\right\rangle.
    \end{split}
\end{equation*}
Notice that, for $i,j=1,\ldots,n,$
\begin{equation}\label{pezzoconsyminproof}
    \begin{split}
        \left\langle\nabla_{e_i} \X,e_j\right\rangle+\left\langle \nabla_{e_j} \X,e_i\right\rangle&=2\varphi h_{ij}+ \left\langle\nabla_{e_i}  \X^T,e_j\right\rangle+\left\langle \nabla_{e_j}  \X^T,e_i\right\rangle=2\varphi h_{ij}+2\sym\left(\nabla^S  \X^T\right)_{ij}
    \end{split}
\end{equation}
and
\begin{equation}\label{auxprimadiprg}
    \left\langle\nabla_{e_i}V',e_j\right\rangle\overset{\eqref{lemma125ritore}}{=}-\left(\hess^S\varphi\right)_{ij}+\left\langle \nabla_{e_i}A\left( \X^T\right),e_j\right\rangle.
\end{equation}
Moreover,
\begin{equation}\label{solopernecessitàdisbatti}
    \begin{split}
        \sum_{i=1}^n\left\langle\nabla_{e_i}\n ,\nabla_{e_i}\X\right\rangle=\varphi\sum_{i=1}^n\left\langle\nabla_{e_i}\n ,\nabla_{e_i}\n \right\rangle+\sum_{i=1}^n\left\langle\nabla_{e_i}\n ,\nabla_{e_i} \X^T\right\rangle=\varphi|h|^2+\left\langle h,\sym\left(\nabla^S  \X^T\right)\right\rangle
    \end{split}
\end{equation}
Therefore, by \eqref{pezzoconsyminproof}, \eqref{auxprimadiprg} and \eqref{solopernecessitàdisbatti},
\begin{equation*}
     \left.\frac{d}{dt}\right|_{t=0}H_t(\Phi_t(p))=-\Delta^S\varphi-\varphi\left(|h|^2+\ric(\n,\n)\right)-\left\langle h,\sym\left(\nabla^S  \X^T\right)\right\rangle+\divv^SA\left( \X^T\right)-\ric( \X^T,N).
\end{equation*}
Next, extend $ \X^T$ to a smooth vector field with compact support in $M$. Denote by $\Psi_t:M\to M$ its flow. Then $\Psi$ is a variation of $S$ with velocity $ \X^T$. By the above formula,
\begin{equation*}
     \X^T H(p)=  \left.\frac{d}{dt}\right|_{t=0}H_t(\Psi_t(p))=-\left\langle h,\sym\left(\nabla^S  \X^T\right)\right\rangle+\divv^S A\left( \X^T\right)-\ric( \X^T,\n).
\end{equation*}
Therefore, \eqref{derprimhinstatement} follows.
Next, we prove \eqref{dersechinstatement}. First we compute $\mathrm{I}$. Indeed,
\begin{equation*}
    \begin{split}
        \mathrm{I}&=2\sum_{i,j,k=1}^n\left(\left.\frac{d}{dt}\right|_{t=0}\G(t)_{ik}\right)\left(\left.\frac{d}{dt}\right|_{t=0}\G(t)_{jk}\right)\S(0)_{ij}\\
        \overset{\eqref{prounoaux}}&{=}2\sum_{i,j,k=1}^nh_{ij}\left(\left\langle\nabla_{e_i}\X,e_k\right\rangle+\left\langle\nabla_{e_k}\X,e_i\right\rangle\right)\left( \left\langle\nabla_{e_j}\X,e_k\right\rangle+\left\langle\nabla_{e_k}\X,e_j\right\rangle\right).
    \end{split}
\end{equation*}
Next we compute $\mathrm{II}$. Noticing that
\begin{equation}\label{componentidivprimoinproof}
    \left\langle\nabla_{e_i}\X,\n\right\rangle=e_i(\varphi)+\left\langle\nabla_{e_i}\X^T,\n\right\rangle=e_i(\varphi)-\left\langle\nabla_{e_i}\n,\X^T\right\rangle=-\left\langle V',e_i\right\rangle\qquad i=1,\ldots,n,
\end{equation}
we deduce that
\begin{equation*}
    \begin{split}
        \mathrm{II}&=-\sum_{i,j=1}^n\left(\left.\frac{d^2}{dt^2}\right|_{t=0}\G(t)_{ij}\right)\S(0)_{ij}\\
        &=-\sum_{i,j=1}^n h_{ij}\left.\frac{d^2}{dt^2}\right|_{t=0}\langle E_i(t),E_j(t)\rangle\\
        &=-2\sum_{i,j=1}^n h_{ij}\left.\frac{d}{dt}\right|_{t=0}\left\langle \frac{D}{dt}E_i(t),E_j(t)\right\rangle\\
        &=-2\sum_{i,j=1}^n h_{ij}\left\langle \left.\frac{D^2}{dt^2}\right|_{t=0}E_i(t),E_j(t)\right\rangle -2\sum_{i,j=1}^n h_{ij}\left\langle \left.\frac{D}{dt}\right|_{t=0}E_i(t),\left.\frac{D}{dt}\right|_{t=0}E_j(t)\right\rangle\\
        \overset{\eqref{prounoaux},\eqref{produeaux}}&{=}-2\sum_{i,j=1}^n h_{ij}\left\langle \nabla_{e_i}\Z,e_j\right\rangle-2\sum_{i,j=1}^n h_{ij}\R(\X,e_i,\X,e_j)-2\sum_{i,j=1}^n h_{ij}\left\langle\nabla_{e_i}\X,\nabla_{e_j}\X\right\rangle\\
        &=-2\left\langle h,\nabla^S\Z\right\rangle-2\left\langle h,\j(\X,\X)\right\rangle-2\sum_{i,j=1}^n h_{ij}\left\langle \nabla_{e_i}\X,e_k\right\rangle\left\langle \nabla_{e_j}\X,e_k\right\rangle-2\sum_{i,j=1}^n h_{ij}\left\langle \nabla_{e_i}\X,\n\right\rangle\left\langle \nabla_{e_j}\X,\n\right\rangle\\
        \overset{\eqref{componentidivprimoinproof}}&{=}-2\left\langle h,\nabla^S\Z\right\rangle-2\left\langle h,\j(\X,\X)\right\rangle-2\sum_{i,j=1}^n h_{ij}\left\langle \nabla_{e_i}\X,e_k\right\rangle\left\langle \nabla_{e_j}\X,e_k\right\rangle-2h\left(V',V'\right).\\
    \end{split}
\end{equation*}
We compute $\mathrm{III}$. To this aim,
\begin{equation*}
    \begin{split}
        \mathrm{III}&=-2\sum_{i,j=1}^n\left(\left.\frac{d}{dt}\right|_{t=0}\G(t)_{ij}\right)\left(\left.\frac{d}{dt}\right|_{t=0}\S(t)_{ij}\right)\\
        \overset{\eqref{prounoaux},\eqref{protreaux}}&{=}-2\sum_{i,j=1}^n\left(\left\langle\nabla_{e_i}\X,e_j\right\rangle+\left\langle \nabla_{e_j}\X,e_i\right\rangle\right)\left(\left\langle\nabla_{e_i}V'+\R(\X,e_i)\n,e_j\right\rangle+\left\langle\nabla_{e_i}\n ,\nabla_{e_j}\X\right\rangle\right)\\
        &=-2\sum_{i,j=1}^n\left(\left\langle\nabla_{e_i}\X,e_j\right\rangle+\left\langle \nabla_{e_j}\X,e_i\right\rangle\right)\left\langle\nabla_{e_i}V',e_j\right\rangle-2\sum_{i,j,k=1}^nh_{ij}\left(\left\langle\nabla_{e_i}\X,e_k\right\rangle+\left\langle \nabla_{e_k}\X,e_i\right\rangle\right)\left\langle\nabla_{e_k}\X,e_j\right\rangle\\
       &\quad-2\left\langle \nabla^S \X,\j(\X,\n)\right\rangle-2\left\langle \nabla^S \X,\j(\n,\X)\right\rangle.
    \end{split}
\end{equation*}
Finally, we compute $\mathrm{IV}$. Indeed, by \eqref{produeaux}, \eqref{protreaux} and \eqref{proquattroaux}, 
\begin{equation*}
    \begin{split}
        \mathrm{IV}&=\sum_{i=1}^n\left.\frac{d^2}{dt^2}\right|_{t=0}\S(t)_{ii}\\
        &=\underbrace{\divv^S V''}_{\mathrm{IV.1}}+\underbrace{\sum_{i=1}^n\left\langle \R\left(\X,\nabla_{e_i}\X\right)\n+2\R(\X,e_i)V'+\R\left(\Z,e_i\right)\n+\left(\nabla_\X \R\right)(\X,e_i)\n,e_i\right\rangle}_{\mathrm{IV.2}}\\
        &\quad+\underbrace{2\sum_{i=1}^n\left\langle  \nabla_{e_i}V'+\R(\X,e_i)\n, \nabla_{e_i}\X\right\rangle}_{\mathrm{IV.3}}+\underbrace{\sum_{i=1}^n\left\langle \nabla_{e_i}\n , \nabla_{e_i}\Z+\R(\X,e_i)\X\right\rangle}_{\mathrm{IV.4}}.
    \end{split}
\end{equation*}
First, 
\begin{equation*}
    \begin{split}
        \mathrm{IV.1}\overset{\eqref{secondordernormalbehave}}&{=}\left(-\left|\nabla^S\varphi\right|^2+2h\left(\nabla^S\varphi, \X^T\right)-h^2\left( \X^T, \X^T\right)\right)H+\divv^S\left( V''\right)^T.
    \end{split}
\end{equation*}
Moreover,
\begin{equation*}
    \begin{split}
        \mathrm{IV.2}&= \sum_{i,j=1}^n\left\langle\nabla_{e_i}\X,e_j\right\rangle \R(\X,e_j,\n,e_i)+\sum_{i=1}^n\left\langle\nabla_{e_i}\X,\n\right\rangle \R(\X,\n,\n,e_i)\\
        &\quad-2\ric(\X,V')-2 \R(\X,\n,V',\n)-\ric(\Z,\n)-\left(\nabla_\X \ric\right)(\X,\n)\\
        \overset{\eqref{componentidivprimoinproof}}&{=}\left \langle\nabla^S \X,\j(\n,\X)\right\rangle-2\ric(\X,V')- \R(\X^T,\n,V',\n)-\ric(\Z,\n)-\left(\nabla_\X \ric\right)(\X,\n).
    \end{split}
\end{equation*}
In addition,
\begin{equation*}
    \begin{split}
        \mathrm{IV.3}&=2\sum_{i,j=1}^n\left\langle\nabla_{e_i}V',e_j\right\rangle\left\langle\nabla_{e_i}\X,e_j\right\rangle+2\sum_{i=1}^n\left\langle\nabla_{e_i}V',\n\right\rangle\left\langle\nabla_{e_i}\X,\n\right\rangle+2\sum_{i,j=1}^n\R(\X,e_i,\n,e_j)\left\langle\nabla_{e_i}\X,e_j\right\rangle\\
        \overset{\eqref{componentidivprimoinproof}}&{=}2\sum_{i,j=1}^n\left\langle\nabla_{e_i}V',e_j\right\rangle\left\langle\nabla_{e_i}\X,e_j\right\rangle+2h\left(V',V'\right)+2\left \langle\nabla^S X,\j(\X,\n)\right\rangle
    \end{split}
\end{equation*}
Finally,
\begin{equation*}
    \begin{split}
        \mathrm{IV.4}&=\sum_{i,j=1}^nh_{ij}\left\langle \nabla_{e_i}\Z,e_j\right\rangle+\sum_{i,j=1}^nh_{ij}\R(\X,e_i,\X,e_j)=\left\langle h,\nabla^S\Z\right\rangle+\left\langle h,\j(\X,\X)\right\rangle.
    \end{split}
\end{equation*}
Recalling \eqref{quattrotraccedersech} and combining the above computations,
\begin{equation*}
    \begin{split}
        \left.\frac{d^2}{dt^2}\right|_{t=0}H_t\left(\Phi_t(p)\right)&=2\left\langle h,\left(\nabla^S X\right)^2\right\rangle-2\left\langle \nabla^S X,\left(\nabla^SV'\right)^t\right\rangle-\left\langle h,\j(\X,\X)\right\rangle-\left\langle\nabla^S \X,\j(\n,\X)\right\rangle\\
         &\quad+\left(-\left|\nabla^S\varphi\right|^2+2h\left(\nabla^S\varphi, \X^T\right)-h^2\left( \X^T, \X^T\right)\right)H+\divv^S\left( V''\right)^T\\
         &\quad-2\ric(\X,V')-\R\left(\X^T,\n,V',\n\right)-\left(\nabla_\X \ric\right)(\X,\n)-\left\langle h,\nabla^S \Z\right\rangle-\ric(\Z,\n).
    \end{split}
\end{equation*}
 Next, fix $p\in S$ and a geodesic frame $\E_1,\ldots,\E_n$ at $p$. Recall that 
    \begin{equation*}
        \left(V''\right)^T=-\sum_{i=1}^n\left \langle \n ,\nabla_{\E_i}\Z\right\rangle\E_i+\left(\tilde V''\right)^T.
    \end{equation*}
    Therefore
    \begin{equation*}
        \begin{split}
            \divv^S\left(V''\right)^T&=-\sum_{i=1}^n\E_i\left \langle \n ,\nabla_{\E_i}\Z\right\rangle+\divv^S \left(\tilde V''\right)^T\\
            &=-\sum_{i=1}^n\E_i\left \langle \n ,\nabla_{\E_i}\left(\left\langle\Z,\n\right\rangle\n\right)\right\rangle-\sum_{i=1}^n\E_i\left \langle \n ,\nabla_{\E_i}\Z^T\right\rangle+\divv^S \left(\tilde V''\right)^T\\
            &=-\sum_{i=1}^n\E_i\left(\E_i\left \langle\Z,\n\right\rangle\right)+\sum_{i=1}^n\E_i\left \langle A\left(\Z^T\right) ,\E_i\right\rangle+\divv^S \left(\tilde V''\right)^T\\
            &=-\Delta^S\left \langle\Z,\n\right\rangle+\divv^S A\left(\Z^T\right)+\divv^S \left(\tilde V''\right)^T.
        \end{split}
    \end{equation*}
Moreover, 
      \begin{equation*}
        \begin{split}
            -\left\langle h,\nabla^S \Z\right\rangle-\ric(\Z,\n)\overset{\eqref{corollarioditracedcodazzi}}&{=}-\divv^S A\left(\Z^T\right)+\Z^T H-\left\langle\Z,\n\right\rangle\left(|h|^2+\ric(\n,\n)\right).
        \end{split}
    \end{equation*}
    Therefore
\eqref{dersechinstatement} follows. Finally, assume that $\Phi$ is normal. Then $\X=\varphi\n$, $\X^T=0$ and $V'=-\nabla^S\varphi$, so that
\begin{equation*}
    \begin{split}
         \left.\frac{d^2}{dt^2}\right|_{t=0}&H_t\left(\Phi_t(p)\right)=2\left\langle h,\left(\nabla^S \left(\varphi\n\right)\right)^2\right\rangle+2\left\langle \nabla^S \left(\varphi\n\right),\left(\nabla^S\left(\nabla^S\varphi\right)\right)^t\right\rangle-\varphi^2\left\langle h,\j(\n,\n)\right\rangle\\
         &\quad-\varphi\left\langle\nabla^S \left(\varphi\n\right),\j(\n,\n)\right\rangle-\left|\nabla^S\varphi\right|^2H+\divv^S\left( \tilde V''\right)^T+2\varphi\ric(\n,\nabla^S\varphi)-\varphi^2\left(\nabla_\n \ric\right)(\n,\n)\\
         &\quad -\Delta^S\left \langle\Z,\n\right\rangle-\left\langle\Z,\n\right\rangle\left(|h|^2+\ric(\n,\n)\right)+\Z^T H.
    \end{split}
\end{equation*}
Noticing that $\nabla^S\left(\varphi\n\right)=\varphi h$ and $\left(\nabla^S\left(\nabla^S\varphi\right)\right)^t=\hess^S\varphi$, we deduce that
\begin{equation*}
        \begin{split}
                \left.\frac{d^2}{dt^2}\right|_{t=0}&H_t\left(\Phi_t(p)\right)=2\varphi^2\trace\left(h^3\right)+2\varphi\left\langle h,\hess^S\varphi\right\rangle-2\varphi^2\left\langle h,\j(\n,\n)\right\rangle-\left|\nabla^S\varphi\right|^2H+\divv^S\left( \tilde V''\right)^T\\
         &\quad+2\varphi\ric\left(\n,\nabla^S\varphi\right)-\varphi^2\left(\nabla_\n \ric\right)(\n,\n)-\Delta^S\left \langle\Z,\n\right\rangle-\left\langle\Z,\n\right\rangle\left(|h|^2+\ric(\n,\n)\right)+\Z^T H.
        \end{split}
    \end{equation*}
    Moreover, 
    \begin{equation*}
    \begin{split}
          2\varphi\left\langle h,\hess^S\varphi\right\rangle+2\varphi\ric\left(\n,\nabla^S\varphi\right)\overset{\eqref{corollarioditracedcodazzi}}&{=}2\varphi\divv^S A\left(\nabla^S\varphi\right)-2\varphi\left\langle \nabla^S H,\nabla^S\varphi\right\rangle.
    \end{split}
    \end{equation*}
    Finally, since
    \begin{equation*}
        \left(\tilde V''\right)^T=2\sum_{i=1}^n\left\langle \nabla^S\varphi,\varphi A(e_i)\right\rangle e_i=\sum_{i=1}^n\left\langle \nabla^S\left(\varphi^2\right), A(e_i)\right\rangle e_i=\sum_{i=1}^n\left\langle  A\left(\nabla^S\left(\varphi^2\right)\right), e_i\right\rangle e_i=A\left(\nabla^S\left(\varphi^2\right)\right),
    \end{equation*}
    \eqref{secvarhnorm} follows.
\end{proof}
\subsection{Variation formulas}
 If 
$f:\rr\to\rr$ is smooth in a neighborhood of $\left\{H(p)\,:\,p\in S\right\}$, set
\begin{equation}\label{genfunriemappend}
    \tmc_f(S)=\int_S f(H)\,dS.
\end{equation}
Let $\Phi$ be a variation. The area formula implies that 
\begin{equation}\label{areaformulaapplicata}
     \tmc_f\left(\Phi_t(S)\right)=\int_S f\left(H_t\left(\Phi_t(p)\right)\right)\jac\left(\Phi_t|_{S}\right)\left(p\right)\,dS.
\end{equation}
Denote by $\jacobi$ the \emph{Jacobi operator} associated to $S$, 
\begin{equation}\label{def_jacobi_operator}
    \jacobi \varphi=-\Delta^S\varphi-\varphi\left(|h|^2+\ric(\n,\n)\right),\qquad\varphi\in C^\infty(S).
\end{equation}
Recall that $\jacobi$ is self-adjoint, namely
\begin{equation}\label{jacobiselfadjoint}
    \int_S\varphi\jacobi\psi\,dS=\int_S\psi\jacobi\varphi\,dS,\qquad\varphi,\psi\in C^\infty(S).
\end{equation}
The first and second variation formulas for \eqref{genfunriemappend} read as follows.
\begin{theorem}\label{maintheoremappendixA}
    Let $S\subseteq M$ be a smooth, embedded, closed, two-sided hypersurface. Fix a function $f:\rr\to\rr$ which is smooth in a neighborhood of $\left\{H(p)\,:\,p\in S\right\}$. Let $\Phi$ be a variation. Denote by $\X$ and $\Z$ its velocity and acceleration respectively. On $S$, decompose $\X$ as $\X|_S=\varphi \n+\X^T$. Set 
        \begin{equation*}
 \delta\tmc_f(S)[\Phi]\coloneqq \left.\frac{d}{d t}\right|_{t=0}\tmc_f\left(\Phi_t(S)\right),\qquad   \delta^2\tmc_f(S)[\Phi]\coloneqq \left.\frac{d^2}{dt^2}\right|_{t=0}\tmc_f\left(\Phi_t(S)\right).
\end{equation*}     
    Then 
    \begin{align}
        \delta\tmc_f(S)[\varphi]\coloneqq  \delta\tmc_f(S)[\Phi]&=\int_S\varphi\Big(-\Delta^S f'(H)-f'(H)\left(|h|^2+\ric(\n,\n)\right)+ f(H) H\Big)\,dS,\label{primavarfunzgen}\\
        \delta^2\tmc_f(S)[\Phi]&=\delta\tmc_f(S)\left[\left\langle\Z|_S,\n\right\rangle-2\X^T\varphi+h\left(\X^T,\X^T\right)\right]+\int_S\varphi\,\mathcal L\varphi\,dS\label{variazsecondanormalemaarbitrariaperilresto},
    \end{align}
    where
    \begin{equation*}
        \begin{split}
            \mathcal L\varphi&=\jacobi\left(f''(H)\jacobi\varphi\right)+2f'(H)\divv^S A\left(\nabla^S\varphi\right)-\left(f'(H)H+f(H)\right)\Delta^S\varphi\\
            &\quad-2 f''(H)\left\langle\nabla^S H,A\left(\nabla^S\varphi\right)\right\rangle+\left(f''(H)H-2f'(H)\right)\left\langle \nabla^S H,\nabla^S\varphi\right\rangle\\
            &\quad+\varphi f'(H)\left(2\trace\left(h^3\right)-2\left\langle h,\j(\n,\n)\right\rangle-\left(\nabla_\n \ric\right)(\n,\n)\right)\\
            &\quad+\varphi \left ( f(H)H^2-\left(2f'(H) H+f(H)\right)\left(|h|^2+ \ric(\n,\n)\right)\right).
        \end{split}
    \end{equation*}
\end{theorem}
\begin{proof}
    First,
    \begin{equation*}
        \begin{split}
              &\frac{d\tmc_f\left(\Phi_t(S)\right)}{dt}\overset{\eqref{areaformulaapplicata}}{=}\int_S\frac{d}{dt}\Big(f\left(H_t\left(\Phi_t(p)\right)\right)\jac\left(\Phi_t|_{S}\right)\left(p\right)\Big)\,dS\\
              &\qquad=\int_S \left(f'\left(H_t\left(\Phi_t(p)\right)\right)\frac{d H_t\left(\Phi_t(p)\right)}{dt}\jac\left(\Phi_t|_{S}\right)\left(p\right)+f\left(H_t\left(\Phi_t(p)\right)\right)\frac{d \jac\left(\Phi_t|_{S}\right)\left(p\right)}{dt}\right)\,dS.\\
        \end{split}
    \end{equation*}
    In particular, by \eqref{derijacuno} and \eqref{derprimhinstatement},
    \begin{equation*}
        \delta\tmc_f(S)[\Phi]=\int_S\Big(f'(H)\left(-\Delta^S\varphi-\varphi\left(|h|^2+\ric(\n,\n)\right)+ \X^T H\right)+f(H)\left(\varphi H+\divv^S\X^T\right)\Big)\,dS.
    \end{equation*}
    Notice that, being $S$ closed,
    \begin{equation}\label{teoremadivergenzainproofvariazioneseconda}
        \int_S\Big( f'(H)\X^T H+f(H)\divv^S\X^T\Big)\,dS=\int_S\divv^S\left(f(H)\X^T\right)\,dS=0. 
    \end{equation}
The divergence theorem and \eqref{teoremadivergenzainproofvariazioneseconda} imply \eqref{primavarfunzgen}. We prove \eqref{variazsecondanormalemaarbitrariaperilresto}.
    We reduce to deal with the case of normal variations. Let $\W\in\Gamma(TM)$ be such that $\W|_S=\X^T$. Denote by $\Psi$ the flow of $-\W$. Since $\W|_S\in\Gamma(TS)$, then $\Psi(t,S)=S$ for any small $t$. Set $\tilde\Phi(t,p)\coloneqq\Phi(t,\Psi(t,p))$. Then $\tilde \Phi$ is a smooth variation, and moreover
    \begin{equation}\label{auxfromnormaltogeneral1}
        \tilde S_t\coloneqq\tilde\Phi(t,S)=\Phi(t,\Psi(t,S))=\Phi(t,S)=S_t
    \end{equation}
    for any $t$ sufficiently small. The area formula and \eqref{auxfromnormaltogeneral1} imply that 
    \begin{equation}\label{auxfromnormaltogeneraltre}
        \delta^2\tmc_f(S)[\Phi]=\delta^2\tmc_f(S)[\tilde \Phi].
    \end{equation} We compute the velocity $\tilde X$ and the acceleration $\tilde Z$ of $\tilde\Phi$. Fix $p\in M$. Fix local normal coordinates $x_1,\ldots,x_{n+1}$ centered at $p$. Set $q(t)=\Phi_t^{-1}(p)$ and $r(t)=\Psi^{-1}_t(q(t)).$ Then, since $\Psi$ is the flow of $-\W$, 
    \begin{equation}\label{auxfromnormaltogeneral2}
        \begin{split}
            \tilde\Y(t,p)&=\left.\frac{\partial}{\partial s}\right|_{s=0}\Phi(t+s,\Psi(t+s,r(t)))\\
            &=\sum_{i=1}^{n+1}\frac{\partial\Phi^i}{\partial t}(t,q(t))\frac{\partial}{\partial x_i}+\sum_{i,j=1}^{n+1}\frac{\partial\Phi^i}{\partial x_j}(t,q(t))\frac{\partial\Psi^j}{\partial t}(t,r(t))\frac{\partial}{\partial x_i}\\
            &=\Y(t,p)-\sum_{i,j=1}^{n+1}\frac{\partial\Phi^i}{\partial x_j}(t,q(t))\W(q(t))\frac{\partial}{\partial x_i}.
        \end{split}
    \end{equation}
    In particular, $\tilde X=\X-\W$, and $\tilde \X|_S=\varphi \n$, whence $\tilde \Phi$ is a normal variation. Moreover, as $   \Phi\left(t,q(t)\right)=p$,
\begin{equation*}
    \begin{split}
        0        
=\frac{\partial\Phi}{\partial t}(0,p)+\sum_{i,j=1}^{n+1}\frac{\partial\Phi_i}{\partial z_j}(0,p)\dot q^j(0)\frac{\partial}{\partial x_i}        =\frac{\partial\Phi}{\partial t}(0,p)+\dot q(0),
    \end{split}
\end{equation*}
whence $\dot q(0)=-\X(p)$. Therefore

    \begin{equation*}
        \begin{split}
            \tilde \X'(p)\overset{\eqref{auxfromnormaltogeneral2}}&{=}\X'(p)-\sum_{i,j,k=1}^{n+1}\left(\delta_{ik}\frac{\partial^2\Phi^i}{\partial t\partial x_j}(0,p)\W^j(p)+\frac{\partial^2\Phi^i}{\partial x_k\partial x_j}(0,p)\dot q^k(0)\W^j(p)+\frac{\partial\Phi^i}{\partial x_j}(0,p)\frac{\partial\W^j}{\partial x_k}(p)\dot q^k(0)\right)\frac{\partial}{\partial x_i}\\
            &=\X'(p)-\sum_{i,j=1}^{n+1}\frac{\partial\X^i}{\partial x_j}(p)\W^j(p)\frac{\partial}{\partial x_i}+\sum_{i,k=1}^{n+1}\frac{\partial\W^i}{\partial x_k}(p)\X^k(p)\frac{\partial}{\partial x_i}\\
            &=\X'(p)-\left(\nabla_\W\X\right)(p)+\left(\nabla_\X\W\right)(p).
        \end{split}
    \end{equation*}
    Finally,
    \begin{equation*}
        \tilde\Z=\tilde\X'+\nabla_{\tilde\X}\tilde\X=\X'-\nabla_\W\X+\nabla_\X\W+\nabla_{\X-\W}\left(\X-\W\right)=\Z-2\nabla_\W\X+\nabla_\W\W,
    \end{equation*}
    so that, recalling that $\W|_S=\X^T$ ,
    \begin{equation}\label{auxfromnormaltogeneral4}
        \left\langle\tilde\Z|_S,\n\right\rangle=\left\langle\Z|_S,\n\right\rangle-2\X^T\varphi+h\left(\X^T,\X^T\right).
    \end{equation}   
    Then, by \eqref{derijacuno}, \eqref{derjduenormal},  \eqref{derprimhinstatement}, \eqref{secvarhnorm} and \eqref{auxfromnormaltogeneraltre},
    \begin{equation*}
        \begin{split}
            \delta^2\tmc_f(S)[\Phi]&=\int_Sf''(H)\left(\Delta^S\varphi+\varphi\left(|h|^2+\ric(\n,\n)\right)\right)^2\,dS\\
            &\quad+\int_Sf'(H)\Big( \divv^S \left( A\left(\nabla^S\varphi^2\right)\right)+2\varphi\divv^S A\left(\nabla^S\varphi\right)-2\varphi\left\langle \nabla^S H,\nabla^S\varphi\right\rangle-\left|\nabla^S\varphi\right|^2H\Big)\,dS\\
            &\quad+\int_Sf'(H)\Big(2\varphi^2\trace\left(h^3\right)-2\varphi^2\left\langle h,\j(\n,\n)\right\rangle-\varphi^2\left(\nabla_\n \ric\right)(\n,\n)\Big)\,dS\\
            &\quad+\int_Sf'(H)\Big(-\Delta^S\left \langle\tilde\Z,\n\right\rangle-\left\langle\tilde\Z,\n\right\rangle\left(|h|^2+\ric(\n,\n)\right)+\tilde\Z^T H\Big)\,dS\\
            &\quad+\int_S 2\varphi Hf'(H)\Big(-\Delta^S\varphi-\varphi\left(|h|^2+\ric(\n,\n)\right)\Big)\\
            &\quad+\int_Sf(H)\Big (\varphi^2H^2-\varphi^2|h|^2-\varphi^2\ric(\n,\n)+\left|\nabla^S\varphi\right|^2+\divv^S\tilde\Z\Big)\,dS.
        \end{split}
    \end{equation*}
    First, combining \eqref{primavarfunzgen} and \eqref{teoremadivergenzainproofvariazioneseconda}, and recalling \eqref{def_jacobi_operator},
        \begin{equation*}
        \begin{split}
            \delta^2&\tmc_f(S)[\Phi]=\underbrace{\int_S\left(f''(H)\jacobi\varphi\right)\jacobi\varphi\,dS}_\mathrm{I}\\
            &\quad+\underbrace{\int_Sf'(H)\Big( \divv^S \left( A\left(\nabla^S\varphi^2\right)\right)-\left|\nabla^S\varphi\right|^2H\Big)\,dS}_\mathrm{II}+\int_Sf'(H)\Big(2\varphi\divv^S A\left(\nabla^S\varphi\right)-2\varphi\left\langle \nabla^S H,\nabla^S\varphi\right\rangle\Big)\,dS\\
            &\quad+\int_Sf'(H)\Big(2\varphi^2\trace\left(h^3\right)-2\varphi^2\left\langle h,\j(\n,\n)\right\rangle-\varphi^2\left(\nabla_\n \ric\right)(\n,\n)\Big)\,dS\\
            &\quad+\int_S 2\varphi Hf'(H)\Big(-\Delta^S\varphi-\varphi\left(|h|^2+\ric(\n,\n)\right)\Big)\\
            &\quad+\underbrace{\int_Sf(H)\left|\nabla^S\varphi\right|^2\,dS}_{\mathrm{III}}+\int_Sf(H)\Big (\varphi^2H^2-\varphi^2|h|^2-\varphi^2\ric(\n,\n)\Big)\,dS+\delta\tmc_f(S)\left[\left\langle\tilde\Z|_S,\n\right\rangle\right].
        \end{split}
    \end{equation*}
    First,
    \begin{equation*}
        \mathrm{I}\overset{\eqref{jacobiselfadjoint}}{=}\int_S\varphi\,\jacobi\left(f''(H)\jacobi\varphi\right)\,dS.
    \end{equation*}
    Moreover,
    \begin{equation*}
 \begin{split}
     \mathrm{II}&=\int_S\Big(-\left\langle\nabla^S f'(H),A\left(\nabla^S\varphi^2\right)\right\rangle-\left\langle f'(H) H\nabla^S\varphi,\nabla^S\varphi\right\rangle\Big)\,dS\\
     &=\int_S\varphi\Big(-2 f''(H)\left\langle\nabla^S H,A\left(\nabla^S\varphi\right)\right\rangle+\divv^S\left(f'(H)H\nabla^S\varphi\right)\Big)\,dS\\
     &=\int_S\varphi\Big(-2 f''(H)\left\langle\nabla^S H,A\left(\nabla^S\varphi\right)\right\rangle+\left(f'(H)+f''(H)H\right)\left\langle \nabla^S H,\nabla^S\varphi\right\rangle+f'(H)H\Delta^S\varphi\Big)\,dS.\\
 \end{split}       
    \end{equation*}
    Finally,
    \begin{equation*}
        \begin{split}
            \mathrm{III}&=-\int_S \varphi\divv^S\left( f(H)\nabla^S\varphi\right)\,dS=-\int_S\varphi \Big(f'(H)\left\langle\nabla^S H,\nabla^S\varphi\right\rangle+f(H)\Delta^S\varphi\Big)\,dS.
        \end{split}
    \end{equation*}
    In particular,
    \begin{equation*}
        \mathrm{II}+\mathrm{III}=\int_S\varphi\Big(-2 f''(H)\left\langle\nabla^S H,A\left(\nabla^S\varphi\right)\right\rangle+f''(H)H\left\langle \nabla^S H,\nabla^S\varphi\right\rangle+\left(f'(H)H-f(H)\right)\Delta^S\varphi\Big)\,dS.
    \end{equation*}
The thesis follows combining the above computations with \eqref{auxfromnormaltogeneral4}.
    \end{proof}
\section{Variation formulas for sub-Riemannian mean curvature functionals}\label{appendicesubriem}
   In this section, we apply the results of \Cref{appendiceriem} to establish variation formulas in the sub-Riemannian Heisenberg group $\hh^n$, for $n\geq 1$. Namely, we approximate it with a sequence of Riemannian manifolds $(\hh^n,\langle\cdot,\cdot\rangle_\eps)_{\eps>0}$, and we exploit \Cref{maintheoremappendixA} to deduce the analogous sub-Riemannian variation formulas. 
   
\subsection{Heisenberg groups}\label{subsec_hg} The Heisenberg group $(\hn,\cdot)$ is $\mathbb R^{2n+1}$ endowed with a group law
that realizes it as stratified Lie group. Its Lie algebra is generated by the (global frame of) left-invariant vector fields 
\begin{equation*}
    Z_i=X_i=\frac{\partial}{\partial x_i}+y_i\frac{\partial}{\partial t},\qquad Z_{n+i}=Y_i=\frac{\partial}{\partial y_i}-x_i\frac{\partial}{\partial t},\qquad Z_{2n+1}=T=\frac{\partial}{\partial t},\qquad i=1,\ldots,n.
\end{equation*}
The only nontrivial commutation relations among $Z_1,\ldots,Z_{2n+1}$ are
\begin{equation*}
    [X_i,Y_i]=-2T,\qquad i=1,\ldots,n.
\end{equation*}
Accordingly, the \emph{horizontal distribution} $\hhh$ generated by $Z_1,\ldots,Z_{2n}$ is \emph{bracket-generating}. A vector field which is tangent to $\hhh$ at every point is called \emph{horizontal}. 
The \emph{complex structure} $J:\Gamma(T\hh^n)\longrightarrow\Gamma(T\hh^n)$ is the unique $C^\infty(\hh^n)$-linear map which satisfies
\begin{equation*}
    J(X_i)=Y_i,\qquad J(Y_i)=-X_i,\qquad J(T)=0,\qquad i=1,\ldots,n.
\end{equation*}
 The triple $(\hh^n,\hhh,J)$ is a prototype of \emph{pseudohermitian manifold}  \cite[Appendix]{MR2165405}. 
The restriction to $\hhh$ of the unique Riemannian metric $\langle\cdot,\cdot\rangle$ making $Z_1,\ldots,Z_{2n+1}$ orthonormal endows $\hh^n$ with the sub-Riemannian structure $\left(\hh^n,\hhh,\langle\cdot,\cdot\rangle|_\hhh\right).$ 
 The \emph{pseudohermitian connection} $\nabla$ \cite{MR4193432} is the unique metric connection for which
\begin{equation}\label{pseudotorsion}
        \nabla_\A\B-\nabla_\B \A-[\A,\B]=2\langle J(\A),\B\rangle T,\qquad \A,\B\in\Gamma(T\hh^n).
\end{equation}
 The Riemannian volume induced by $\langle\cdot,\cdot\rangle$ is the Haar measure of the group, i.e. the standard Lebesgue measure.  Therefore, the Riemannian divergence induced by $\left\langle\cdot,\cdot\right\rangle$ is the Euclidean divergence.
\subsection{Hypersurfaces}\label{subsechypersurf}  Throughout this section,  $S\subseteq\hh^n$ is an embedded, closed, two-sided hypersurface of class $C^2$. A point $p\in S$ is called \emph{characteristic} if $\hhh_p=T_p S$. The set of characteristic points of $S$ is denoted by $S_0$. We will always assume that $S\setminus S_0$ is smooth. At non-characteristic points, the \emph{horizontal tangent space} $\hhh TS$ is the smooth, $(2n-1)$-dimensional distribution defined by 
\begin{equation*}
    \hhh T_p S=\hhh_p\cap T_p S,\qquad p\in S\setminus S_0.
\end{equation*}
Denote by $\n$ the Riemannian unit normal to $S$, and by $\n^\hhh$ its orthogonal projection onto $\hhh$. 
Then, the \emph{horizontal unit normal }  $$\vh=\frac{\n^\hhh}{|\n^\hhh|}$$ 
is well-defined on $S\setminus S_0$, and is the unique, up to sign, horizontal unit vector field orthogonal to $\hhh TS$. Notice that $p$ is a characteristic point if and only if $|\n^\hhh|=0$. Close to every non-characteristic point, it is always possible to extend $\vh$ to a full neighborhood in $\hh^n$ by setting 
   \begin{equation}\label{normcondist}
    \vh=\nabla^\hhh d,
\end{equation}
   where $d$ is the signed \emph{Carnot-Carathéodory distance} from $S$ (cf. \cite{simons,MR4193432}). 
   Henceforth, $\vh$ is always extended as in \eqref{normcondist}.
   The \emph{fundamental function} $\alpha$ is defined on $S\setminus S_0$ as the unique smooth function such that 
   \begin{equation*}\label{firstdefinitionofthevecotrs}
       \s\coloneqq T-\alpha\vh\in\Gamma(TS).
   \end{equation*}
   It is known (cf. \cite{simons,MR4923606}) that $\alpha=Td$, and moreover
  \begin{equation}\label{propvh5}
    \nabla_{\vh}\vh=-2\alpha J(\vh).
\end{equation}
   Denote by $\hhh' TS$ the distribution defined by
\begin{equation*}
    \hhh'T_p S=\hhh T_pS\cap J\left(\hhh T_p S\right),\qquad p\in S\setminus S_0.
\end{equation*}
Then, $\hhh'TS$ is a $(2n-2)$-dimensional sub-bundle of $\hhh TS$, and the latter can be orthogonally decomposed as $
    \hhh TS=\hhh' TS\oplus\spann J(\vh). $
Notice that, in the first Heisenberg group, $\hhh' TS=\{0\}$. 
It is easy to check that 
\begin{equation*}
    \n=\frac{1}{\sqrt{1+\alpha^2}}\vh+\frac{\alpha}{\sqrt{1+\alpha^2}}T,\qquad \alpha=\frac{\left\langle \n,T\right\rangle}{|\n^\hhh|}.
\end{equation*}
The \emph{horizontal second fundamental form} $h^\hhh$ and the \emph{symmetric horizontal second fundamental form} (cf. \cite{simons,MR4193432}) $\tilde h^\hhh$ are defined on $S\setminus S_0$ respectively by
\begin{equation*}
    h^\hhh(\A,\B)\coloneqq\left\langle A^\hhh(\A),\B\right\rangle,\qquad \tilde h^\hhh(\A,\B)\coloneqq\frac{h^\hhh(\A,\B)+h^\hhh(\B,\A)}{2},\qquad\A,\B\in\Gamma(\hhh TS),
\end{equation*}
 where $ A^\hhh(\A)\coloneqq \nabla_\A\vh$ is the \emph{horizontal shape operator}. The forms $h^\hhh$ and $\tilde h^\hhh$ are related (cf. \cite{simons}) by
\begin{equation}\label{rapptrahetildacca2026}
    \tilde h^\hhh(\A,\B)=h^\hhh(\A,\B)+\alpha\left\langle J(\A),\B\right\rangle,\qquad\A,\B\in\Gamma(\hhh TS).
\end{equation}
The \emph{horizontal mean curvature} is then defined on $S\setminus S_0$ by
\begin{equation*}
    H^\hhh=\trace h^\hhh=\trace \tilde h^\hhh.
\end{equation*}
Finally, the relevant sub-Riemannian surface measure $\sigma^\hhh$ is defined (cf. \cite{MR2354992, MR2262196}) by
\begin{equation}\label{areaelementorizz}
    \sigma^\hhh\coloneqq \frac{1}{\sqrt{1+\alpha^2}}\sigma,
\end{equation}
where $\sigma$ is the Riemannian surface measure induced by $\langle\cdot,\cdot\rangle$. Since $S$ is of class $C^2$, then $\sigma^\hhh(S_0)=\sigma(S_0)=0$ \cite{MR2021034}: accordingly, $S_0$ will be omitted in the forthcoming integrals.
The \emph{tangent pseudohermitian connection} $\nabla^S$ is the affine connection defined on $S$ by
\begin{equation*}
    \nabla^S _\A \B=\nabla_\A\B-\langle\nabla_\A\B,\vh\rangle\vh,\qquad \A,\B\in\Gamma(\hhh TS).
\end{equation*}
Let $\E_1,\ldots,\E_{2n-1}$ be any local orthonormal frame of $\hhh TS$. If $\varphi\in C^\infty(S\setminus S_0)$ and $\A\in\Gamma(\hhh TS)$ is supported in $S\setminus S_0$, the \emph{horizontal tangential gradient} of $\varphi$ and the  \emph{horizontal tangential divergence} of $\A$ are defined by 
\begin{equation*}
    \nabla ^{\hhh,S}\varphi=\sum_{i=1}^{2n-1}\left(\E_i\varphi\right)\E_i,\qquad \divv^{\hhh,S}\A\coloneqq\sum_{i=1}^{2n-1}\left\langle \nabla^S_{\E_i}\A,\E_i\right\rangle.
\end{equation*}
The \emph{horizontal tangential Laplacian} and the \emph{modified horizontal tangential Laplacian} are then defined by
\begin{equation*}
    \Delta^{\hhh,S} \varphi=\divv^{\hhh,S}\nabla^{\hhh,S}\varphi,\qquad \hat\Delta^{\hhh,S}\varphi\coloneqq \Delta^{\hhh,S}\varphi +2\alpha J(\vh)\varphi.
\end{equation*}
   Finally, we denote by $\jacobi^\hhh$ the \emph{horizontal Jacobi operator}
    \begin{equation*}\label{horizontaljacopdefmain}
    \jacobi^\hhh \varphi=-\hat\Delta^{\hhh,S}\varphi-\varphi\left(|\tilde h^\hhh|^2+4J(\vh)\alpha+(2n+2)\alpha^2\right),\qquad\varphi\in C^\infty(S\setminus S_0).
    \end{equation*}
    Unlike $\Delta^{\hhh,S}$  (cf. \cite{MR2354992}), both $\hat\Delta^{\hhh,S}$ and $\jacobi^\hhh$ are self-adjoint on $C^\infty_c(S\setminus S_0)$ (cf. \Cref{proposizioneibpsubriemsurf}).   
   \subsection{Riemannian approximation}
   For any $\eps>0$, denote by $\left\langle\cdot,\cdot\right\rangle_\eps$ the unique Riemannian metric on $\hh^n$ such that $(Z_1^\eps,\ldots,Z_{2n+1}^\eps)=(X_1,\ldots,X_n,Y_1,\ldots,Y_n,\eps T)$ is a global orthonormal frame. If $\A\in\Gamma(T\hh^n)$, denote by $A^{1,\eps},\ldots,A^{2n,\eps},A^{2n+1,\eps}$ its components with respect to the above frame. Clearly $A^j\coloneqq A^{j,\eps}$ is independent of $\eps$ for $j=1,\ldots,2n$. We may write $A^{2n+1}\coloneqq A^{2n+1,\eps}$ when there is no ambiguity.
   Since $\left\langle\A,\B\right\rangle_\eps$ does not depend on $\eps$ when $\A,\B\in\Gamma(\hhh)$, we set $\left\langle\A,\B\right\rangle\coloneqq \left\langle\A,\B\right\rangle_\eps.$    Denote by $\nabla^\eps$ and $\ \R^\eps $ the induced Levi-Civita connection and curvature tensor respectively. We recall the following relations (cf. \cite{MR5029815}):
   \begin{alignat}{2}
&\notag \nabla_{X_i}^\eps X_j=0, \qquad \qquad  \ \, \nabla^\eps_{Y_i}Y_j=0, \qquad \ \,\nabla^\eps_{T}T=0, \\
\label{levi_civita}&\nabla^\eps_{X_i}Y_j=-\delta_{i,j}T, \qquad \, \nabla^\eps_{X_i}\eps T=\frac{Y_i}{\eps}, \quad \ \ \nabla^\eps_{Y_i}\eps T=-\frac{X_i}{\eps}, \\
&\notag \nabla^\eps_{Y_i}X_j=\delta_{i,j}T, \qquad \ \ \,\,\nabla^\eps_{\eps T}X_i=\frac{Y_i}{\eps}, \  \ \quad \nabla^\eps_{\eps T}Y_i=-\frac{X_i}{\eps}.
\end{alignat}
The Levi-Civita connection $\nabla^\eps$ relates to the pseudohermitian connection $\nabla$ as follows.
\begin{lemma}
    Let $\A,\B\in\Gamma(T\hh^n)$. Then
    \begin{equation}\label{levicivitavspseudohermitian}
        \nabla^\eps_\A\B=\nabla_\A\B+\frac{\langle\A,\eps T\rangle_\eps}{\eps}J(\B)+\frac{\langle\B,\eps T\rangle_\eps}{\eps}J(\A)+\langle \A,J(\B)\rangle T.
    \end{equation}
    In particular,
    \begin{equation}\label{connessioneerotazione}
    \nabla^\eps_\A \eps T=\frac{1}{\eps}J(\A).
\end{equation}
Moreover, if $\A,\B,\C\in\Gamma(\hhh)$, then
\begin{equation}\label{relazionetraleconnessioni}
    \left\langle\nabla^\eps_\A\B,\C\right\rangle_\eps=\left\langle\nabla_\A\B,\C\right\rangle.
\end{equation}
\end{lemma}
\begin{proof}
    By a direct computation,
    \begin{equation*}
        \begin{split}
            \nabla^\eps_\A\B&=\sum_{j,k=1}^{2n+1}A^{j,\eps}Z_j^\eps B^{k,\eps}Z_k^\eps+\sum_{j,k=1}^{2n+1}A^{j,\eps}B^{k,\eps}\nabla^\eps_{Z_j^\eps}Z_k^\eps\\
            \overset{\eqref{levi_civita}}&{=}\nabla_\A\B+\sum_{j=1}^nA^jB^{n+j}\nabla^\eps_{X_j}Y_j+\langle \B,\eps T\rangle_\eps\sum_{j=1}^nA^j\nabla^\eps_{X_j}\eps T+\sum_{j=1}^nA^{n+j}B^{j}\nabla^\eps_{Y_j}X_j\\
            &\quad+\langle \B,\eps T\rangle_\eps\sum_{j=1}^nA^{n+j}\nabla^\eps_{Y_j}\eps T+\langle \A,\eps T\rangle_\eps\sum_{j=1}^nB^{j}\nabla^\eps_{\eps T}X_j+\langle \A,\eps T\rangle_\eps\sum_{j=1}^nB^{n+j}\nabla^\eps_{\eps T}Y_j\\
            \overset{\eqref{levi_civita}}&{=}\nabla_\A\B+\frac{\langle\A,\eps T\rangle_\eps}{\eps}J(\B)+\frac{\langle\B,\eps T\rangle_\eps}{\eps}J(\A)+\langle \A,J(\B)\rangle T.
        \end{split}
    \end{equation*}
\end{proof}
\subsection{Ambient curvatures} First, we compute the ambient curvatures of $(\hh^n,\langle\cdot,\cdot \rangle_\eps)$, starting from $\R^\eps$.
\begin{proposition}\label{riemepsexplicitprop}
    Let $\eps>0$. Let $\A,\B,\C\in\Gamma(T\hh^n)$. Then
    \begin{equation*}
        \begin{split}
             \R^\eps (\A,\B)\C&=\frac{2}{\eps^2}\left\langle J(\A),\B\right\rangle _\eps J(\C)+\frac{1}{\eps^2}\left\langle J(\A),\C\right\rangle_\eps J(\B)-\frac{1}{\eps^2}\left\langle J(\B),\C\right\rangle _\eps J(\A)\\
            &\quad+\frac{A^{2n+1,\eps}}{\eps^2}\langle\B,\C\rangle _\eps\eps T-\frac{B^{2n+1,\eps}}{\eps^2}\langle\A,\C\rangle_\eps \eps T+\frac{B^{2n+1,\eps}C^{2n+1,\eps}}{\eps^2}\A-\frac{A^{2n+1,\eps}C^{2n+1,\eps}}{\eps^2}\B.
        \end{split}
    \end{equation*}
    In particular, if $\D\in\Gamma(T\hh^n)$,
    \begin{equation}\label{riemannepsilonquattrozero}
        \begin{split}
            & \R^\eps (\A,\B,\C,\D)=\frac{2}{\eps^2}\left\langle J(\A),\B\right\rangle _\eps \left\langle J(\C),\D\right\rangle_\eps+\frac{1}{\eps^2}\left\langle J(\A),\C\right\rangle_\eps \left\langle J(\B),\D\right\rangle_\eps-\frac{1}{\eps^2}\left\langle J(\B),\C\right\rangle _\eps \left\langle J(\A),\D\right\rangle_\eps\\
            &\quad+\frac{A^{2n+1,\eps}D^{2n+1,\eps}}{\eps^2}\langle\B,\C\rangle _\eps-\frac{B^{2n+1,\eps}D^{2n+1,\eps}}{\eps^2}\langle\A,\C\rangle_\eps +\frac{B^{2n+1,\eps}C^{2n+1,\eps}}{\eps^2}\left\langle \A,\D\right\rangle_\eps-\frac{A^{2n+1,\eps}C^{2n+1,\eps}}{\eps^2}\langle\B,\D\rangle_\eps.
        \end{split}
    \end{equation}
\end{proposition}
\begin{proof}
Fix $i,j,k=1,\ldots,n$. Then, by \eqref{levi_civita},
\begin{align*}
     \R^\eps (X_i,X_j)X_k&=0,\\
     \R^\eps (X_i,X_j)Y_k&=\nabla^{\eps}_{X_i}\nabla^\eps_{X_j}Y_k-\nabla^{\eps}_{X_j}\nabla^\eps_{X_i}Y_k-\nabla^\eps_{[X_i,X_j]}Y_k=-\frac{\delta_{jk}}{\eps}\nabla^{\eps}_{X_i}\eps T+\frac{\delta_{ik}}{\eps}\nabla^{\eps}_{X_j}\eps T=\frac{1}{\eps^2}\left(\delta_{ik}Y_j-\delta_{jk}Y_i\right),\\
     \R^\eps (X_i,X_j)\eps T&=\nabla^{\eps}_{X_i}\nabla^\eps_{X_j}\eps T-\nabla^{\eps}_{X_j}\nabla^\eps_{X_i}\eps T-\nabla^\eps_{[X_i,X_j]}\eps T=\frac{1}{\eps}\nabla^{\eps}_{X_i}Y_j-\frac{1}{\eps}\nabla^{\eps}_{X_j}Y_i=0.
\end{align*}
Moreover, 
\begin{align*}
     \R^\eps (X_i,Y_j)X_k&=\nabla^{\eps}_{X_i}\nabla^\eps_{Y_j}X_k-\nabla^{\eps}_{Y_j}\nabla^\eps_{X_i}X_k-\nabla^\eps_{[X_i,Y_j]}X_k=\frac{\delta_{jk}}{\eps}\nabla^{\eps}_{X_i}\eps T+\frac{2\delta_{ij}}{\eps}\nabla^\eps_{\eps T}X_k=\frac{1}{\eps^2}\left(\delta_{jk}Y_i+2\delta_{ij}Y_k\right),\\
     \R^\eps (X_i,Y_j)Y_k&=\nabla^{\eps}_{X_i}\nabla^\eps_{Y_j}Y_k-\nabla^{\eps}_{Y_j}\nabla^\eps_{X_i}Y_k-\nabla^\eps_{[X_i,Y_j]}Y_k=\frac{\delta_{ik}}{\eps}\nabla^\eps_{Y_j}\eps T+\frac{2\delta_{ij}}{\eps}\nabla^\eps_{\eps T}Y_k=-\frac{1}{\eps^2}\left(\delta_{ik}X_j+2\delta_{ij}X_k\right),\\
     \R^\eps (X_i,Y_j)\eps T&=\nabla^{\eps}_{X_i}\nabla^\eps_{Y_j}\eps T-\nabla^{\eps}_{Y_j}\nabla^\eps_{X_i}\eps T-\nabla^\eps_{[X_i,Y_j]}\eps T=0.
\end{align*}
In addition,
\begin{align*}
     \R^\eps (X_i,\eps T)X_k&=\nabla^{\eps}_{X_i}\nabla^\eps_{\eps T}X_k-\nabla^{\eps}_{\eps T}\nabla^\eps_{X_i}X_k-\nabla^\eps_{[X_i,\eps T]}X_k=\frac{1}{\eps}\nabla^{\eps}_{X_i}Y_k=-\frac{\delta_{ik}}{\eps^2}\eps T,\\
      \R^\eps (X_i,\eps T)Y_k&=\nabla^{\eps}_{X_i}\nabla^\eps_{\eps T}Y_k-\nabla^{\eps}_{\eps T}\nabla^\eps_{X_i}Y_k-\nabla^\eps_{[X_i,\eps T]}Y_k=0,\\
       \R^\eps (X_i,\eps T)\eps T&=\nabla^{\eps}_{X_i}\nabla^\eps_{\eps T}\eps T-\nabla^{\eps}_{\eps T}\nabla^\eps_{X_i}\eps T-\nabla^\eps_{[X_i,\eps T]}\eps T=-\frac{1}{\eps}\nabla^{\eps}_{\eps T}Y_i=\frac{1}{\eps^2}X_i.
\end{align*}
Furthermore,
\begin{align*}
     \R^\eps (Y_i,Y_j)X_k&=\nabla^{\eps}_{Y_i}\nabla^\eps_{Y_j}X_k-\nabla^{\eps}_{Y_j}\nabla^\eps_{Y_i}X_k-\nabla^\eps_{[Y_i,Y_j]}X_k=\frac{\delta_{jk}}{\eps}\nabla^{\eps}_{Y_i}\eps T-\frac{\delta_{ik}}{\eps}\nabla^{\eps}_{Y_j}\eps T=\frac{1}{\eps^2}\left(\delta_{ik}X_j-\delta_{jk}X_i\right),\\
     \R^\eps (Y_i,Y_j)Y_k&=0,\\
     \R^\eps (Y_i,Y_j)\eps T&=\nabla^{\eps}_{Y_i}\nabla^\eps_{Y_j}\eps T-\nabla^{\eps}_{Y_j}\nabla^\eps_{Y_i}\eps T-\nabla^\eps_{[Y_i,Y_j]}\eps T=-\frac{1}{\eps}\nabla^{\eps}_{Y_i}X_j+\frac{1}{\eps}\nabla^{\eps}_{Y_j}X_i=0.
\end{align*}
Finally,
\begin{align*}
      \R^\eps (Y_i,\eps T)X_k&=\nabla^{\eps}_{Y_i}\nabla^\eps_{\eps T}X_k-\nabla^{\eps}_{\eps T}\nabla^\eps_{Y_i}X_k-\nabla^\eps_{[Y_i,\eps T]}X_k=0,\\
        \R^\eps (Y_i,\eps T)Y_k&=\nabla^{\eps}_{Y_i}\nabla^\eps_{\eps T}Y_k-\nabla^{\eps}_{\eps T}\nabla^\eps_{Y_i}Y_k-\nabla^\eps_{[Y_i,\eps T]}Y_k=-\frac{1}{\eps}\nabla^{\eps}_{Y_i}X_k=-\frac{\delta_{ik}}{\eps^2}\eps T,\\
         \R^\eps (Y_i,\eps T)\eps T&=\nabla^{\eps}_{Y_i}\nabla^\eps_{\eps T}\eps T-\nabla^{\eps}_{\eps T}\nabla^\eps_{Y_i}\eps T-\nabla^\eps_{[Y_i,\eps T]}\eps T=\frac{1}{\eps}\nabla^\eps_{\eps T}X_i=\frac{1}{\eps^2}Y_i.
\end{align*}
    Therefore 
    \begin{equation*}
        \begin{split}
            & \R^\eps (\A,\B)\C=\sum_{i,j,k=1}^{2n+1}A^iB^jC^k \R^\eps (Z_i,Z_j)Z_k\\
            &=\sum_{i=1}^n\sum_{j,k=1}^{2n+1}A^iB^jC^k \R^\eps (X_i,Z_j)Z_k+\sum_{i=1}^n\sum_{j,k=1}^{2n+1}A^{n+i}B^jC^k \R^\eps (Y_i,Z_j)Z_k+A^{2n+1}\sum_{j,k=1}^{2n+1}B^jC^k \R^\eps (\eps T,Z_j)Z_k\\
            &=\sum_{i,j=1}^n\sum_{k=1}^{2n+1}A^iB^jC^k \R^\eps (X_i,X_j)Z_k+\sum_{i,j=1}^n\sum_{k=1}^{2n+1}A^iB^{n+j}C^k \R^\eps (X_i,Y_j)Z_k+B^{2n+1}\sum_{i=1}^n\sum_{k=1}^{2n+1}A^iC^k \R^\eps (X_i,\eps T)Z_k\\
            &\quad+\sum_{i,j=1}^n\sum_{k=1}^{2n+1}A^{n+i}B^jC^k \R^\eps (Y_i,X_j)Z_k+\sum_{i,j=1}^n\sum_{k=1}^{2n+1}A^{n+i}B^{n+j}C^k \R^\eps (Y_i,Y_j)Z_k\\
            &\quad+B^{2n+1}\sum_{i=1}^n\sum_{k=1}^{2n+1}A^{n+i}C^k \R^\eps (Y_i,\eps T)Z_k+A^{2n+1}\sum_{j=1}^n\sum_{k=1}^{2n+1}B^jC^k \R^\eps (\eps T,X_j)Z_k\\
            &\quad +A^{2n+1}\sum_{j=1}^n\sum_{k=1}^{2n+1}B^{n+j}C^k \R^\eps (\eps T,Y_j)Z_k.           
        \end{split}
    \end{equation*}
    By the symmetry of $\ \R^\eps ,$ 
    \begin{equation*}
        \begin{split}
             \R^\eps (\A,\B)\C&=\underbrace{\sum_{i,j=1}^n\sum_{k=1}^{2n+1}A^iB^jC^k \R^\eps (X_i,X_j)Z_k}_\mathrm{I}+\underbrace{\sum_{i,j=1}^n\sum_{k=1}^{2n+1}\left(A^iB^{n+j}-A^{n+j}B^i\right)C^k \R^\eps (X_i,Y_j)Z_k}_\mathrm{II}\\
            &\quad+\underbrace{\sum_{i=1}^n\sum_{k=1}^{2n+1}\left(A^iB^{2n+1}-A^{2n+1}B^i\right)C^k \R^\eps (X_i,\eps T)Z_k}_\mathrm{III}+\underbrace{\sum_{i,j=1}^n\sum_{k=1}^{2n+1}A^{n+i}B^{n+j}C^k \R^\eps (Y_i,Y_j)Z_k}_\mathrm{IV}\\
            &\quad+\underbrace{\sum_{i=1}^n\sum_{k=1}^{2n+1}\left(A^{n+i}B^{2n+1}-A^{2n+1}B^{n+i}\right)C^k \R^\eps (Y_i,\eps T)Z_k}_\mathrm{V}.
        \end{split}
    \end{equation*}
    First,
    \begin{equation*}
        \begin{split}
            \mathrm{I}=\frac{1}{\eps^2}\sum_{i,j=1}^nA^iB^jC^{n+i}Y_j-\frac{1}{\eps^2}\sum_{i,j=1}^nA^iB^jC^{n+j}Y_i=\frac{1}{\eps^2}\sum_{i,j=1}^nA^jB^iC^{n+j}Y_i-\frac{1}{\eps^2}\sum_{i,j=1}^nA^iB^jC^{n+j}Y_i.
        \end{split}
    \end{equation*}
    Moreover,
    \begin{equation*}
        \begin{split}
            \mathrm{II}&=\sum_{i,j,k=1}^n\left(A^iB^{n+j}-A^{n+j}B^i\right)C^k \R^\eps (X_i,Y_j)X_k+\sum_{i,j,k=1}^n\left(A^iB^{n+j}-A^{n+j}B^i\right)C^{n+k} \R^\eps (X_i,Y_j)Y_k\\
            &=\frac{1}{\eps^2}\sum_{i,j=1}^n\left(A^iB^{n+j}-A^{n+j}B^i\right)C^jY_i+\frac{2}{\eps^2}\sum_{i,k=1}^n\left(A^iB^{n+i}-A^{n+i}B^i\right)C^kY_k\\
            &\quad-\frac{1}{\eps^2}\sum_{i,j=1}^n\left(A^iB^{n+j}-A^{n+j}B^i\right)C^{n+i}X_j-\frac{2}{\eps^2}\sum_{i,k=1}^n\left(A^iB^{n+i}-A^{n+i}B^i\right)C^{n+k}X_i\\
            &=\frac{1}{\eps^2}\sum_{i,j=1}^n\left(A^iB^{n+j}-A^{n+j}B^i\right)C^jY_i-\frac{1}{\eps^2}\sum_{i,j=1}^n\left(A^jB^{n+i}-A^{n+i}B^j\right)C^{n+j}X_i+\frac{2}{\eps^2}\left\langle J(\A),\B\right\rangle J(\C).
        \end{split}
    \end{equation*}
    In addition,
    \begin{equation*}
        \begin{split}
            \mathrm{III}&=\sum_{i,k=1}^n\left(A^iB^{2n+1}-A^{2n+1}B^i\right)C^k \R^\eps (X_i,\eps T)X_k+C^{2n+1}\sum_{i=1}^n\left(A^iB^{2n+1}-A^{2n+1}B^i\right) \R^\eps (X_i,\eps T)\eps T\\
            &=-\frac{1}{\eps^2}\sum_{i=1}^n\left(A^iB^{2n+1}-A^{2n+1}B^i\right)C^i\eps T+\frac{C^{2n+1}}{\eps^2}\sum_{i=1}^n\left(A^iB^{2n+1}-A^{2n+1}B^i\right)X_i\\
            &=-\frac{B^{2n+1}}{\eps^2}\sum_{j=1}^nA^jC^j\eps T+\frac{A^{2n+1}}{\eps^2}\sum_{j=1}^nB^jC^j\eps T+\frac{B^{2n+1}C^{2n+1}}{\eps^2}\sum_{i=1}^nA^iX_i-\frac{A^{2n+1}C^{2n+1}}{\eps^2}\sum_{i=1}^nB^iX_i.
        \end{split}
    \end{equation*}
    Furthermore,
    \begin{equation*}
        \mathrm{IV}=\frac{1}{\eps^2}\sum_{i,j=1}^nA^{n+i}B^{n+j}C^iX_j-\frac{1}{\eps^2}\sum_{i,j=1}^nA^{n+i}B^{n+j}C^jX_i=\frac{1}{\eps^2}\sum_{i,j=1}^nA^{n+j}B^{n+i}C^jX_i-\frac{1}{\eps^2}\sum_{i,j=1}^nA^{n+i}B^{n+j}C^jX_i,
    \end{equation*}
    Finally,
    \begin{equation*}
        \begin{split}
            \mathrm{V}&=\sum_{i,k=1}^n\left(A^{n+i}B^{2n+1}-A^{2n+1}B^{n+i}\right)C^{n+k} \R^\eps (Y_i,\eps T)Y_k\\
            &\quad+C^{2n+1}\sum_{i=1}^n\left(A^{n+i}B^{2n+1}-A^{2n+1}B^{n+i}\right) \R^\eps (Y_i,\eps T)\eps T\\
            &=-\frac{1}{\eps^2}\sum_{i=1}^n\left(A^{n+i}B^{2n+1}-A^{2n+1}B^{n+i}\right)C^{n+i}\eps T+\frac{C^{2n+1}}{\eps^2}\sum_{i=1}^n\left(A^{n+i}B^{2n+1}-A^{2n+1}B^{n+i}\right)Y_i\\
            &=-\frac{B^{2n+1}}{\eps^2}\sum_{i=1}^nA^{n+i}C^{n+i}\eps T+\frac{A^{2n+1}}{\eps^2}\sum_{i=1}^nB^{n+i}C^{n+i}\eps T\\
            &\quad+\frac{B^{2n+1}C^{2n+1}}{\eps^2}\sum_{i=1}^nA^{n+i}Y_i-\frac{A^{2n+1}C^{2n+1}}{\eps^2}\sum_{i=1}^nB^{n+i}Y_i.
        \end{split}
    \end{equation*}
    By the above computations,
    \begin{equation*}
\begin{split}
    \mathrm{I}+\mathrm{II}+\mathrm{IV}&=\frac{2}{\eps^2}\left\langle J(\A),\B\right\rangle J(\C)+\frac{1}{\eps^2}\left\langle J(\A),\C\right\rangle \sum_{i=1}^nB^i Y_i-\frac{1}{\eps^2}\left\langle J(\A),\C\right\rangle \sum_{i=1}^nB^{n+i} X_i\\
    &\quad+\frac{1}{\eps^2}\left\langle J(\B),\C\right\rangle \sum_{i=1}^nA^{n+i} X_i-\frac{1}{\eps^2}\left\langle J(\B),\C\right\rangle \sum_{i=1}^nA^{i} Y_i\\
    &=\frac{2}{\eps^2}\left\langle J(\A),\B\right\rangle J(\C)+\frac{1}{\eps^2}\left\langle J(\A),\C\right\rangle J(\B)-\frac{1}{\eps^2}\left\langle J(\B),\C\right\rangle J(\A)
\end{split}        
    \end{equation*}
    and 
    \begin{equation*}
        \begin{split}
            \mathrm{III}+\mathrm{V}&=\frac{A^{2n+1}}{\eps^2}\langle\B,\C\rangle \eps T-\frac{B^{2n+1}}{\eps^2}\langle\A,\C\rangle \eps T+\frac{B^{2n+1}C^{2n+1}}{\eps^2}\A-\frac{A^{2n+1}C^{2n+1}}{\eps^2}\B.
        \end{split}
    \end{equation*}
    The thesis follows combining the above computations.
\end{proof}
The explicit expression of the Ricci curvature $\ric^\eps$ is a simple consequence of \Cref{riemepsexplicitprop}.
\begin{proposition}
    Let $\eps>0$. Let $\A,\B,\C\in\Gamma(T\hh^n)$. Then
    \begin{equation}\label{ricciepsilon}
        \begin{split}
            \ric^\eps(\B,\C)&=-\frac{2}{\eps^2}\left\langle \B,\C\right\rangle_\eps +\frac{(2n+2)}{\eps^2}\left\langle \B,\eps T\right\rangle_\eps\left\langle \C,\eps T\right\rangle_\eps
        \end{split}
    \end{equation}
    and 
    \begin{equation*}
        \left(\nabla^\eps_\A\ric^\eps\right)(\B,\C)=\frac{2n+2}{\eps^3}\left\langle J(\A),\B\right\rangle_\eps\left\langle \C,\eps T\right\rangle_\eps+\frac{2n+2}{\eps^3}\left\langle J(\A),\C\right\rangle_\eps\left\langle \B,\eps T\right\rangle_\eps
    \end{equation*}
    In particular, 
    \begin{equation}\label{nabaepsriccizerosudiagonale}
        \left(\nabla^\eps_\A\ric^\eps\right)(\A,\A)=0.
    \end{equation}
\end{proposition}
\begin{proof}
    By definition,
    \begin{equation*}
        \begin{split}
            \eps^2\ric(\B,\C)&
    =\sum_{i=1}^{2n}\Big(2\left\langle J(Z_i),\B\right\rangle _\eps \left\langle J(\C),Z_i\right\rangle_\eps+\left\langle J(Z_i),\C\right\rangle_\eps \left\langle J(\B),Z_i\right\rangle_\eps-\left\langle J(\B),\C\right\rangle _\eps \left\langle J(Z_i),Z_i\right\rangle_\eps\Big)\\
            &\quad+\langle\B,\C\rangle _\eps-B^{2n+1,\eps}C^{2n+1,\eps}+(2n+1)B^{2n+1,\eps}C^{2n+1,\eps}-B^{2n+1,\eps}C^{2n+1,\eps}.\\
            &=-3\left\langle J(\B),J(\C)\right\rangle +\langle\B,\C\rangle _\eps+(2n-1)B^{2n+1,\eps}C^{2n+1,\eps}\\
            &=-2\left\langle \B,\C\right\rangle +(2n+2)B^{2n+1,\eps}C^{2n+1,\eps}.
        \end{split}
    \end{equation*}
    Therefore, since $\nabla^\eps\left\langle\cdot,\cdot\right\rangle_\eps\equiv 0$,
    \begin{equation*}
        \begin{split}
            \eps^2\left(\nabla_\A\ric^\eps\right)(\B,\C)&=(2n+2)\big(\A\left(\left\langle \B,\eps T\right\rangle_\eps\left\langle \C,\eps T\right\rangle_\eps\right)-\left\langle \nabla^\eps_\A\B,\eps T\right\rangle_\eps\left\langle \C,\eps T\right\rangle_\eps-\left\langle \B,\eps T\right\rangle_\eps\left\langle \nabla^\eps_\A\C,\eps T\right\rangle_\eps\big)\\
            &=(2n+2)\big(\left\langle \B,\nabla^\eps_\A\eps T\right\rangle_\eps\left\langle \C,\eps T\right\rangle_\eps+\left\langle \B,\eps T\right\rangle_\eps\left\langle \C,\nabla^\eps_\A\eps T\right\rangle_\eps\big)\\
            \overset{\eqref{connessioneerotazione}}&{=}\frac{2n+2}{\eps}\left\langle J(\A),\B\right\rangle_\eps\left\langle \C,\eps T\right\rangle_\eps+\frac{2n+2}{\eps}\left\langle J(\A),\C\right\rangle_\eps\left\langle \B,\eps T\right\rangle_\eps.
        \end{split}
    \end{equation*}
    Finally, \eqref{nabaepsriccizerosudiagonale} is straightforward.
       \end{proof}
   \subsection{Extrinsic curvatures}\label{sec_extrinsiccurvaturesriemapp} Fix $\eps>0$. We adapt some aspects of \Cref{subsechypersurf} to the Riemannian structure $(\hh^n,\langle\cdot,\cdot\rangle_\eps).$
We assign an upper index $\eps$ to the geometric quantities associated with $S$ and induced by $\left\langle\cdot,\cdot\right\rangle_\eps$. Accordingly,
\begin{equation}\label{normaleriemannianaepsilon}
    \n^\eps=\frac{1}{\sqrt{1+\eps^2\alpha^2}}\vh+\frac{\eps\alpha}{\sqrt{1+\eps^2\alpha^2}}\eps T.
\end{equation}
Therefore, $TS$ admits the $\left\langle\cdot,\cdot\right\rangle_\eps$-orthonormal  decomposition $TS=\hhh' TS\oplus\spann J(\vh)\oplus\spann \s^\eps,$ 
where 
\begin{equation}\label{formadiesseepsilon}
    \s^\eps=-\frac{\eps\alpha}{\sqrt{1+\eps^2\alpha^2}}\vh+\frac{1}{\sqrt{1+\eps^2\alpha^2}}\eps T.
\end{equation}
If $\sigma^\eps$ is the Riemannian surface measure induced by $\langle\cdot,\cdot\rangle_\eps$, it is easy to check that 
\begin{equation}\label{areaelementeps}
    \sigma^\eps=\frac{\sqrt{1+\eps^2\alpha^2}}{\eps\sqrt{1+\alpha^2}}\sigma^1.
\end{equation}
In the rest of this section, we express some Riemannian extrinsic quantities with respect to the relevant sub-Riemannian geometric objects. We begin with the second fundamental form.
\begin{lemma}\label{lemmasecondaformafondriem}
    Let $\eps>0$. Let $\A,\B\in\Gamma(\hhh T S)$. Let $\C\in\Gamma(\hhh'TS)$. Then
    \begin{equation}\label{formaepsenonepstuttoorizz}
        \begin{alignedat}{2}                  h^\eps(\A,\B)&=\frac{1}{\sqrt{1+\eps^2\alpha^2}}\tilde h^\hhh(\A,\B),&\qquad
        h^\eps(\C,\s^\eps)&=\frac{\eps\C\alpha}{1+\eps^2\alpha^2},\\
        h^\eps(J(\vh),\s^\eps)&=\frac{1}{\eps}+\frac{\eps J(\vh)\alpha}{1+\eps^2\alpha^2},&\qquad h^\eps(\s^\eps,\s^\eps)&=\frac{\eps^2\s\alpha}{\left(1+\eps^2\alpha^2\right)^{\frac{3}{2}}}.
        \end{alignedat}
    \end{equation}
\end{lemma}
\begin{proof}
    First, since
    \begin{equation}\label{utilenellaproofmadatogliere}
        \left\langle \vh,\B\right\rangle=\left\langle\eps T,\B\right\rangle_\eps=0,
    \end{equation}
    then
    \begin{equation*}
        \begin{split}
            h^\eps(\A,\B)
            \overset{\eqref{normaleriemannianaepsilon}}&{=}\left\langle\nabla^\eps_\A\left(\frac{1}{\sqrt{1+\eps^2\alpha^2}}\vh\right),\B\right\rangle_\eps+\left\langle\nabla^\eps_\A\left(\frac{\eps\alpha}{\sqrt{1+\eps^2\alpha^2}}\eps T\right),\B\right\rangle_\eps\\
            \overset{\eqref{utilenellaproofmadatogliere}}&{=}\frac{1}{\sqrt{1+\eps^2\alpha^2}}\Big(\left\langle\nabla^\eps_\A\vh,\B\right\rangle_\eps+\eps\alpha\left\langle\nabla^\eps_\A\eps T,\B\right\rangle_\eps\Big)\\
            \overset{\eqref{connessioneerotazione},\eqref{relazionetraleconnessioni}}&{=}\frac{1}{\sqrt{1+\eps^2\alpha^2}}\Big(h^\hhh(\A,\B)+\alpha\left\langle J(\A),\B\right\rangle\Big)\\
            \overset{\eqref{rapptrahetildacca2026}}&{=}\frac{1}{\sqrt{1+\eps^2\alpha^2}}\tilde h^\hhh(\A,\B).
        \end{split}
    \end{equation*}
    Let $\D\in\Gamma(TS)$. Then
    \begin{equation*}
        \begin{split}
            h^\eps&(\D,\s^\eps)=\left\langle\nabla^\eps_\D\n^\eps,\s^\eps\right\rangle_\eps\\
            \overset{\eqref{normaleriemannianaepsilon}}&{=}\frac{1}{\sqrt{1+\eps^2\alpha^2}}\left\langle\nabla^\eps_\D\left(\frac{1}{\sqrt{1+\eps^2\alpha^2}}\vh\right),-\eps\alpha \vh+\eps T\right\rangle_\eps+\frac{1}{\sqrt{1+\eps^2\alpha^2}}\left\langle\nabla^\eps_\D\left(\frac{\eps\alpha}{\sqrt{1+\eps^2\alpha^2}}\eps T\right),-\eps\alpha \vh+\eps T\right\rangle_\eps\\
            &=-\frac{\eps\alpha}{\sqrt{1+\eps^2\alpha^2}}\D\left(\frac{1}{\sqrt{1+\eps^2\alpha^2}}\right)+\frac{1}{1+\eps^2\alpha^2}\Big\langle\nabla^\eps_\D\vh,-\eps\alpha \vh+\eps T\Big\rangle_\eps\\
            &+\frac{\eps\alpha}{\sqrt{1+\eps^2\alpha^2}}\D\left(\frac{1}{\sqrt{1+\eps^2\alpha^2}}\right)+\frac{\eps\D\alpha}{1+\eps^2\alpha^2}+\frac{\eps\alpha}{1+\eps^2\alpha^2}\Big\langle\nabla^\eps_\D\eps T,-\eps\alpha \vh+\eps T\Big\rangle_\eps\\
            &=\frac{1}{1+\eps^2\alpha^2}\Big\langle\nabla^\eps_\D\vh,\eps T\Big\rangle_\eps-\frac{\eps^2\alpha^2}{1+\eps^2\alpha^2}\Big\langle\nabla^\eps_\D\eps T, \vh\Big\rangle_\eps+\frac{\eps\D\alpha}{1+\eps^2\alpha^2}\\
            \overset{\eqref{connessioneerotazione}}&{=}-\frac{1}{\eps(1+\eps^2\alpha^2)}\left\langle J(\D),\vh\right\rangle_\eps-\frac{\eps^2\alpha^2}{\eps(1+\eps^2\alpha^2)}\left\langle J(\D), \vh\right\rangle_\eps+\frac{\eps\D\alpha}{1+\eps^2\alpha^2}\\
            &=\frac{1}{\eps}\left\langle \D,J(\vh)\right\rangle+\frac{\eps\D\alpha}{1+\eps^2\alpha^2}.
        \end{split}
    \end{equation*}
\end{proof}
A first trivial consequence of \Cref{lemmasecondaformafondriem} is the behavior of the mean curvature.
\begin{corollary}
    Let $p\in S\setminus S_0$. Then
    \begin{equation}\label{curvepsecurvh}
    \begin{split}
         H^\eps&=\frac{1}{\sqrt{1+\eps^2\alpha^2}}H^\hhh+\frac{\eps^2\s\alpha}{\left(1+\eps^2\alpha^2\right)^{\frac{3}{2}}}.
    \end{split}
    \end{equation}
    In particular,
    \begin{equation}\label{convhepstoh}
        \A_1\cdots\A_k H^\eps\xrightarrow[\eps\to 0]{}\A_1\cdots\A_k H^\hhh\text{ locally uniformly on $S\setminus S_0$ for any $k\in\mathbb N$, $\A_1,\ldots,\A_k\in\Gamma(TS)$.}
    \end{equation}
   \end{corollary}
   In the next result, the approximation occurs both in the second fundamental form and in its entries.
\begin{lemma}\label{lemmasecofondformcondueparametri}
    Let $p\in S\setminus S_0$. Let $f,g\in C^\infty(S\setminus S_0)$. Then 
    \begin{equation}\label{hepsgradepsprimadilimieteffbgyhygtfr}
        \begin{split}
            h^\eps\left(\nabla^{\eps,S}f,\nabla^{\eps,S}g\right)&=\frac{\tilde h^\hhh\left(\nabla^{\hhh,S}f,\nabla^{\hhh,S}g\right)}{\sqrt{1+\eps^2\alpha^2}}+\frac{J(\vh) f\s g+J(\vh)g\s f}{\sqrt{1+\eps^2\alpha^2}}\\
            &\quad+\frac{\eps^2\left\langle \nabla^{\hhh,S}\alpha,\s g\nabla^{\hhh,S}f+\s f\nabla^{\hhh,S}g\right\rangle}{\left(1+\eps^2\alpha^2\right)^\frac{3}{2}}+\frac{\eps^4\s\alpha\s f\s g}{\left(1+\eps^2\alpha^2\right)^\frac{5}{2}}.
        \end{split}
    \end{equation}
    In particular, if $(f^\eps)_\eps,\,(g^\eps)_\eps$ converge smoothly to $f$ and $g$ as in \eqref{convhepstoh}, then    \begin{equation}\label{convhepsgradeps}
        \begin{split}
        h^\eps\left(\nabla^{\eps,S}f^\eps,\nabla^{\eps,S}g^\eps\right)&\xrightarrow[\eps\to 0]{}\tilde h^\hhh\left(\nabla^{\hhh,S}f,\nabla^{\hhh,S}g\right)+ J(\vh)f\,\s g+J(\vh)g\,\s f \text{ locally uniformly on $S\setminus S_0$.}
        \end{split}
    \end{equation}
    
\end{lemma}
\begin{proof}
       Fix $p\in S\setminus S_0$. Let $\E_1,\ldots,\E_{n-1},\E_{n+1},\ldots,\E_{2n-1}$ be a local orthonormal frame of $\hhh'TS$. Set $\E_n=J(\vh)$ and $\E_{2n}=\s^\eps$. In this way, $\E_1,\ldots,\E_{2n}$ is a local orthonormal frame of $TS$. Then
       \begin{equation*}
           \begin{split}
               h^\eps&\left(\nabla^{\eps,S}f,\nabla^{\eps,S}g\right)=\sum_{i,j=1}^{2n}h^\eps(\E_i,\E_j)\E_i f\E_i g\\
               &=\sum_{i,j=1}^{2n-1}h^\eps(\E_i,\E_j)\E_i f\E_i g+\sum_{i=1}^{2n-1}h^\eps(\E_i,\s^\eps)\left(\E_i f\s^\eps g+\E_i g\s^\eps f\right)+h^\eps(\s^\eps,\s^\eps)\s^\eps f\s^\eps g\\
               \overset{\eqref{formaepsenonepstuttoorizz}}&{=}\frac{\tilde h^\hhh\left(\nabla^{\hhh,S}f,\nabla^{\hhh,S}g\right)}{\sqrt{1+\eps^2\alpha^2}}+\frac{\eps^2}{\left(1+\eps^2\alpha^2\right)^\frac{3}{2}}\sum_{\substack{i=1\\ i\neq n}}^{2n-1}\E_i\alpha\left(\E_i f\s g+\E_i g\s f\right)\\
               &\quad+\frac{1}{\sqrt{1+\eps^2\alpha^2}}\left(1+\frac{\eps^2J(\vh)\alpha}{1+\eps^2\alpha^2}\right)\left(J(\vh) f\s g+J(\vh)g\s f\right)+\frac{\eps^4\s\alpha\s f\s g}{\left(1+\eps^2\alpha^2\right)^\frac{5}{2}}\\
               &=\frac{\tilde h^\hhh\left(\nabla^{\hhh,S}f,\nabla^{\hhh,S}g\right)}{\sqrt{1+\eps^2\alpha^2}}+\frac{J(\vh) f\s g+J(\vh)g\s f}{\sqrt{1+\eps^2\alpha^2}}+\frac{\eps^2\left\langle \nabla^{\hhh,S}\alpha,\s g\nabla^{\hhh,S}f+\s f\nabla^{\hhh,S}g\right\rangle}{\left(1+\eps^2\alpha^2\right)^\frac{3}{2}}+\frac{\eps^4\s\alpha\s f\s g}{\left(1+\eps^2\alpha^2\right)^\frac{5}{2}}.
           \end{split}
       \end{equation*}
       In particular, \eqref{convhepsgradeps} follows by \eqref{hepsgradepsprimadilimieteffbgyhygtfr}.
\end{proof}
The next convergence result has been achieved, through a different approach, in \cite{MR5029815}. 
\begin{lemma}
    Let $p\in S\setminus S_0$. Then
    \begin{equation*}
        \left|h^\eps\right|^2+\ric^\eps(\n^\eps,\n^\eps)=\frac{|\tilde h^\hhh|^2+4 J(\vh)\alpha+(2n+2)\alpha^2}{1+\eps^2\alpha^2}+\frac{2\eps^2\left|\nabla^{\hhh,S}\alpha\right|^2}{\left(1+\eps^2\alpha^2\right)^2}+\frac{\eps^4\left(\s\alpha\right)^2}{\left(1+\eps^2\alpha^2\right)^3}
    \end{equation*}
    In particular,
    \begin{equation}\label{convergenzadelpotenzialejacobi}
        \left|h^\eps\right|^2+\ric^\eps(\n^\eps,\n^\eps)\xrightarrow[\eps \to 0]{}|\tilde h^\hhh|^2+4 J(\vh)\alpha+(2n+2)\alpha^2\qquad\text{  locally uniformly on $S\setminus S_0$.}
    \end{equation}
\end{lemma}
\begin{proof}
     Let $\E_1,\ldots,\E_{2n-1}$ be as in the proof of \Cref{lemmasecofondformcondueparametri}. First, by \eqref{ricciepsilon} and \eqref{normaleriemannianaepsilon}, 
     \begin{equation*}
         \ric^\eps(\n^\eps,\n^\eps)=-\frac{2}{\eps^2}+\frac{(2n+2)\alpha^2}{1+\eps^2\alpha^2}.
     \end{equation*}
     Moreover,
     \begin{equation*}
         \begin{split}
             \left|h^\eps\right|^2&=\sum_{i,j=1}^{2n-1}h^\eps(\E_i,\E_j)^2+2h^\eps(J(\vh),\s^\eps)^2+2\sum_{\substack{i=1 \\ i\neq n}}^{2n-1}h^\eps(\E_i,\s^\eps)^2+h^\eps(\s^\eps,\s^\eps)^2\\
         \overset{\eqref{formaepsenonepstuttoorizz}}&{=}\frac{|\tilde h^\hhh|^2}{1+\eps^2\alpha^2}+\frac{2}{\eps^2}+\frac{4 J(\vh)\alpha}{1+\eps^2\alpha^2}+\frac{2\eps^2\left(J(\vh)\alpha\right)^2}{\left(1+\eps^2\alpha^2\right)^2}+\frac{2\eps^2}{\left(1+\eps^2\alpha^2\right)^2}\sum_{\substack{i=1 \\ i\neq n}}^{2n-1}\left(\E_i\alpha\right)^2+\frac{\eps^4\left(\s\alpha\right)^2}{\left(1+\eps^2\alpha^2\right)^3}\\
         &=\frac{2}{\eps^2}+\frac{|\tilde h^\hhh|^2+4 J(\vh)\alpha}{1+\eps^2\alpha^2}+\frac{2\eps^2\left|\nabla^{\hhh,S}\alpha\right|^2}{\left(1+\eps^2\alpha^2\right)^2}+\frac{\eps^4\left(\s\alpha\right)^2}{\left(1+\eps^2\alpha^2\right)^3}.
         \end{split}
     \end{equation*}
     The thesis follows combining the above computations.
\end{proof}
\Cref{lemmasecondaformafondriem} allows to compare the Laplace-Beltrami operator with $\hat \Delta^{\hhh,S}$. We need the following lemma.
\begin{lemma}
    Let $p\in S\setminus S_0$. Then 
    \begin{equation}\label{nabepssepsseps}
        \nabla^\eps_{\s^\eps}\s^\eps=-2\alpha J(\vh)-\frac{\eps^2\alpha}{1+\eps^2\alpha^2}\nabla^{\hhh,S}\alpha-\frac{\eps^2\s\alpha}{\left(1+\eps^2\alpha^2\right)^{\frac{3}{2}}}\n^\eps.
    \end{equation}
    \end{lemma}
    \begin{proof}
       First,
        \begin{equation*}
            \left\langle \nabla^\eps_{\s^\eps}\s^\eps,\n^\eps\right\rangle_\eps=-h^\eps(\s^\eps,\s^\eps)\overset{\eqref{formaepsenonepstuttoorizz}}{=}-\frac{\eps^2\s\alpha}{\left(1+\eps^2\alpha^2\right)^{\frac{3}{2}}}.
        \end{equation*}
        Moreover, $\left\langle \nabla^\eps_{\s^\eps}\s^\eps,\s^\eps\right\rangle_\eps=0$. 
        Finally, fix $\E\in\Gamma(\hhh T S)$. Then
        \begin{equation*}
            \begin{split}
                   \Big\langle &\nabla^\eps_{\s^\eps}\s^\eps,\E\Big\rangle_\eps=\frac{1}{\sqrt{1+\eps^2\alpha^2}}\left\langle \nabla^\eps_{-\eps\alpha\vh+\eps T}\left(\frac{-\eps\alpha\vh+\eps T}{\sqrt{1+\eps^2\alpha^2}}\right),\E\right\rangle_\eps\\
                   &=\frac{1}{1+\eps^2\alpha^2}\left\langle \nabla^\eps_{-\eps\alpha\vh+\eps T}\left(-\eps\alpha\vh\right),\E\right\rangle_\eps+\frac{1}{1+\eps^2\alpha^2}\left\langle \nabla^\eps_{-\eps\alpha\vh+\eps T}\eps T,\E\right\rangle_\eps\\
                   \overset{\eqref{connessioneerotazione}}&{=}\frac{\eps^2\alpha^2}{1+\eps^2\alpha^2}\left\langle \nabla^\eps_{\vh}\vh,\E\right\rangle_\eps-\frac{\eps\alpha}{1+\eps^2\alpha^2}\left\langle \nabla^\eps_{\eps T}\vh,\E\right\rangle_\eps-\frac{\alpha}{1+\eps^2\alpha^2}\left\langle \E,J(\vh)\right\rangle\\
                \overset{\eqref{levi_civita},\eqref{propvh5}}&{=}-\frac{2\eps^2\alpha^3}{1+\eps^2\alpha^2}\left\langle \E,J(\vh)\right\rangle-\frac{\eps^2\alpha}{1+\eps^2\alpha^2}\sum_{i=1}^{2n}T(\vh^i)E^i-\frac{\alpha}{1+\eps^2\alpha^2}\sum_{i=1}^n\left(\vh^i Y_i-\vh^{n+i}X_i\right)-\frac{\alpha}{1+\eps^2\alpha^2}\left\langle \E,J(\vh)\right\rangle\\
                \overset{\eqref{normcondist}}&{=}-2\alpha\left\langle \E,J(\vh)\right\rangle-\frac{\eps^2\alpha\E\alpha}{1+\eps^2\alpha^2}.
            \end{split}
        \end{equation*}
        The thesis follows by the above computations.
    \end{proof}
    \begin{proposition}
        Let $ \A\in\Gamma(\hhh T S)$. Let $f\in C^\infty(S\setminus S_0)$. Let $p\in S\setminus S_0$. Then
        \begin{equation}\label{diverepsilonespressa}
            \divv^{\eps,S}\left(\A+f\s\right)=\divv^{\hhh,S}\A+2\alpha \left\langle \A,J(\vh)\right\rangle+\s f-f H^\hhh\alpha+\frac{\eps^2\alpha}{1+\eps^2\alpha^2}\left(\A\alpha+f\s\alpha\right).
        \end{equation}
        In particular, if $\varphi\in\C^\infty(S\setminus S_0)$ and $p\in S\setminus S_0$,
        \begin{equation}\label{laplaepsilonespresso}
            \Delta^{\eps,S}\varphi=\hat \Delta^{\hhh,S}\varphi+\frac{\eps^2\s\s\varphi}{1+\eps^2\alpha^2}-\frac{\eps^4\alpha\s\alpha\s\varphi}{\left(1+\eps^2\alpha^2\right)^2}-\frac{\eps^2H^\hhh\alpha\s\varphi}{1+\eps^2\alpha^2}+\frac{\eps^2\alpha}{1+\eps^2\alpha^2}\nabla^{\hhh,S}\alpha.
        \end{equation}
        Therefore, if $(\varphi^\eps)_\eps$ converges to $\varphi$ smoothly as in \eqref{convhepstoh}, then
        \begin{equation}\label{jacepstojach}
        \jacobi^\eps\varphi^\eps\xrightarrow[\eps\to 0]{}\jacobi^\hhh\varphi,\qquad\text{ locally uniformly on $S\setminus S_0$.}
    \end{equation}
    \end{proposition}
    \begin{proof}
        Let $\E_1,\ldots,\E_{2n-1}$ be as in the proof of \Cref{lemmasecofondformcondueparametri}. First,
        \begin{equation*}
            \begin{split}
                \divv^{\eps,S}\A
                \overset{\eqref{relazionetraleconnessioni}}{=}\sum_{i=1}^{2n-1}\left\langle\nabla_{\E_i}\A,\E_i\right\rangle-\left\langle\nabla^\eps_{\s^\eps}\s^\eps,\A\right\rangle_\eps
                \overset{\eqref{nabepssepsseps}}{=}\divv^{\hhh,S}\A+2\alpha \left\langle \A,J(\vh)\right\rangle+\frac{\eps^2\alpha}{1+\eps^2\alpha^2}\A\alpha.
            \end{split}
        \end{equation*}
        Moreover,
        \begin{equation*}
            \begin{split}
                \divv^{\eps,S}\left(f\s\right)
                &=\s f+f\sum_{i=1}^{2n-1}\left\langle\nabla^\eps_{\E_i}\s,\E_i\right\rangle_\eps+f\left\langle\nabla^\eps_{\s^\eps}\s,\s^\eps\right\rangle_\eps\\
                &=\s f-f\alpha\sum_{i=1}^{2n-1}\left\langle\nabla^\eps_{\E_i}\vh,\E_i\right\rangle_\eps+f \sum_{i=1}^{2n-1}\left\langle\nabla^\eps_{\E_i} T,\E_i\right\rangle_\eps+f\left\langle\nabla^\eps_{\s^\eps}\left(\frac{\sqrt{1+\eps^2\alpha^2}}{\eps}\s^\eps\right),\s^\eps\right\rangle_\eps\\
                \overset{\eqref{connessioneerotazione},\eqref{relazionetraleconnessioni}}&{=}\s f-f\alpha\sum_{i=1}^{2n-1}\left\langle\nabla_{\E_i}\vh,\E_i\right\rangle+\frac{f}{\eps^2} \sum_{i=1}^{2n-1}\left\langle J(\E_i),\E_i\right\rangle+\frac{f}{\sqrt{1+\eps^2\alpha^2}}\s\left(\sqrt{1+\eps^2\alpha^2}\right)\\
                &=\s f-f H^\hhh\alpha+\frac{\eps^2f\alpha\s\alpha}{1+\eps^2\alpha^2}.
            \end{split}
        \end{equation*}
        To prove \eqref{laplaepsilonespresso} it suffices to apply \eqref{diverepsilonespressa}, noticing that 
        \begin{equation*}
            \nabla^{\eps,S}\varphi=\nabla^{\hhh,S}\varphi+\frac{\eps^2\s\varphi}{1+\eps^2\alpha^2}\s.
        \end{equation*}
        Finally, \eqref{jacepstojach} follows by \eqref{convergenzadelpotenzialejacobi} and \eqref{laplaepsilonespresso}.
    \end{proof}
   The above results and the divergence theorem yield the following integration-by-parts formula (cf. \cite{MR2354992}).
    \begin{proposition}\label{proposizioneibpsubriemsurf}
        Let $\varphi\in C^2(S)\cap C^\infty(S\setminus S_0)$.  Let $\A\in\Gamma(\hhh TS)$. If $\supp \left(\varphi\A\right)\subseteq S\setminus S_0,$ then 
        \begin{align}\label{ibpsubriemequation}
            \int_S\varphi\divv^{\hhh,S}\A\,d\sigma^\hhh+2\int_S\varphi\alpha\left\langle \A,J(\vh)\right\rangle\,d\sigma^\hhh=-\int_S\A\varphi\,d\sigma^\hhh.
        \end{align}
        In addition, if $\psi\in C^2(S)\cap C^\infty(S\setminus S_0)$ and $\supp(\varphi\psi)\subseteq S \setminus S_0$, then 
        \begin{equation}\label{ibpverticalformula}
            \int_S\varphi\,\s\psi\,d\sigma^\hhh=\int_S\varphi\,\psi H^\hhh\alpha\,d\sigma^\hhh-\int_S\s\varphi\,\psi\,d\sigma^\hhh.
        \end{equation}
    \end{proposition}
Finally, we show the convergence of the cubic curvature term appearing in \eqref{variazsecondanormalemaarbitrariaperilresto}.
\begin{lemma}\label{lemmatracehepsterza}
    Let $p\in S\setminus S_0$. Then
    \begin{equation*}
        \begin{split}
            \trace&\left(\left(h^\eps\right)^3\right)=\frac{3\tilde h^\hhh(J(\vh),J(\vh))}{\eps^2\left(1+\eps^2\alpha^2\right)^\frac{1}{2}}+\frac{1}{\left(1+\eps^2\alpha^2\right)^\frac{3}{2}}\left(\trace\left(\left(\tilde h^\hhh\right)^3\right)+6\tilde h^\hhh\left(\nabla^{\hhh,S}\alpha,J(\vh)\right)+3\s\alpha\right)\\
            &\quad+\frac{3\eps^2}{\left(1+\eps^2\alpha^2\right)^\frac{5}{2}}\left(\tilde h^\hhh\left(\nabla^{\hhh,S}\alpha,\nabla^{\hhh,S}\alpha\right)+2J(\vh)\alpha\s\alpha\right)+\frac{3\eps^4}{\left(1+\eps^2\alpha^2\right)^\frac{7}{2}}\s\alpha\left|\nabla^{\hhh,S}\alpha\right|^2+\frac{\eps^6}{\left(1+\eps^2\alpha^2\right)^\frac{9}{2}}\left(\s\alpha\right)^3
        \end{split}
    \end{equation*}
\end{lemma}
\begin{proof}
     Let $\E_1,\ldots,\E_{2n-1}$ be as in the proof of \Cref{lemmasecofondformcondueparametri}. Set $h^\eps_{ij}=h^\eps(\E_i,\E_j)$ for $i,j=1,\ldots,2n$. Then
      \begin{equation*}
          \begin{split}
              \trace\left(\left(h^\eps\right)^3\right)
              =\underbrace{\sum_{i,j,k=1 }^{2n-1}h^\eps_{ij}h^\eps_{jk}h^\eps_{ki}}_\mathrm{I}+\underbrace{3\sum_{i,j=1 }^{2n-1}h^\eps_{ij}h^\eps_{i,2n}h^\eps_{j,2n}}_\mathrm{II}+\underbrace{3h^\eps_{2n,2n}\sum_{i=1}^{2n-1}\left(h^\eps_{i,2n}\right)^2}_\mathrm{III}+\underbrace{\left(h^\eps_{2n,2n}\right)^3}_\mathrm{IV}.
          \end{split}
      \end{equation*}
      First,
      \begin{equation*}
          \mathrm{I}\overset{\eqref{formaepsenonepstuttoorizz}}{=}\frac{1}{\left(1+\eps^2\alpha^2\right)^\frac{3}{2}}\trace\left(\left(\tilde h^\hhh\right)^3\right).
      \end{equation*}
      Moreover,
      \begin{equation*}
          \begin{split}
              \mathrm{II}
              &=3\sum_{\substack{i,j=1 \\ i,j\neq n} }^{2n-1}h^\eps_{ij}h^\eps_{i,2n}h^\eps_{j,2n}+6h^\eps_{n,2n}\sum_{\substack{i=1  \\ i\neq n}}^{2n-1}h^\eps_{i,n}h^\eps_{i,2n}+3h^\eps_{n,n}\left(h^\eps_{n,2n}\right)^2\\
              \overset{\eqref{formaepsenonepstuttoorizz}}&{=}\frac{3\eps^2}{\left(1+\eps^2\alpha^2\right)^\frac{5}{2}}\sum_{\substack{i,j=1 \\ i,j\neq n} }^{2n-1}\tilde h^\hhh(\E_i,\E_j)\E_i\alpha\E_j\alpha+6\left(1+\frac{\eps^2J(\vh)\alpha}{1+\eps^2\alpha^2}\right)\frac{1}{\left(1+\eps^2\alpha^2\right)^\frac{3}{2}}\sum_{\substack{i=1  \\ i\neq n}}^{2n-1}\tilde h^\hhh (\E_i,J(\vh))\E_i\alpha\\
              &\quad+\frac{3}{\eps^2\left(1+\eps^2\alpha^2\right)^\frac{1}{2}}\tilde h^\hhh(J(\vh),J(\vh))\left(1+\frac{\eps^2J(\vh)\alpha}{1+\eps^2\alpha^2}\right)^2\\
              &=\frac{3\eps^2}{\left(1+\eps^2\alpha^2\right)^\frac{5}{2}}\tilde h^\hhh\left(\nabla^{\hhh,S}\alpha,\nabla^{\hhh,S}\alpha\right)+\frac{6}{\left(1+\eps^2\alpha^2\right)^\frac{3}{2}}\tilde h^\hhh\left(\nabla^{\hhh,S}\alpha,J(\vh)\right)+\frac{3}{\eps^2\left(1+\eps^2\alpha^2\right)^\frac{1}{2}}\tilde h^\hhh(J(\vh),J(\vh)).
          \end{split}
      \end{equation*}
      In addition,
      \begin{equation*}
          \begin{split}
              \mathrm{III}&=3h^\eps_{2n,2n}\sum_{\substack{i=1 \\ i\neq n}}^{2n-1}\left(h^\eps_{i,2n}\right)^2+3h^\eps_{2n,2n}\left(h^\eps_{n,2n}\right)^2\\
              \overset{\eqref{formaepsenonepstuttoorizz}}&{=}\frac{3\eps^4}{\left(1+\eps^2\alpha^2\right)^\frac{7}{2}}\s\alpha\sum_{\substack{i=1 \\ i\neq n}}^{2n-1}\left(\E_i\alpha\right)^2+\frac{3}{\left(1+\eps^2\alpha^2\right)^\frac{3}{2}}\s\alpha\left(1+\frac{\eps^2J(\vh)\alpha}{1+\eps^2\alpha^2}\right)^2\\
              &=\frac{3\eps^4}{\left(1+\eps^2\alpha^2\right)^\frac{7}{2}}\s\alpha\left|\nabla^{\hhh,S}\alpha\right|^2+\frac{6\eps^2}{\left(1+\eps^2\alpha^2\right)^\frac{5}{2}}\s\alpha J(\vh)\alpha+\frac{3}{\left(1+\eps^2\alpha^2\right)^\frac{3}{2}}\s\alpha.
          \end{split}
      \end{equation*}
      Finally,
      \begin{equation*}
          \mathrm{IV}\overset{\eqref{formaepsenonepstuttoorizz}}{=}\frac{\eps^6}{\left(1+\eps^2\alpha^2\right)^\frac{9}{2}}\left(\s\alpha\right)^3.
      \end{equation*}
      The thesis follows combining the above computations.
\end{proof}
\begin{lemma}\label{lemmadiscalprodhconcnn}
    Let $p\in S\setminus S_0$. Then
    \begin{equation*}
        \left\langle h^\eps,\j^\eps(\n^\eps,\n^\eps)\right\rangle_\eps=\frac{3\tilde h^\hhh(J(\vh),J(\vh))}{\eps^2\left(1+\eps^2\alpha^2\right)^\frac{3}{2}}-\frac{1}{\left(1+\eps^2\alpha^2\right)^\frac{3}{2}}\Big(\s\alpha+H^\hhh\alpha^2\Big).
    \end{equation*}
\end{lemma}
\begin{proof}
        Let $\E_1,\ldots,\E_{2n-1}$ be as in the proof of \Cref{lemmasecofondformcondueparametri}. By \eqref{riemannepsilonquattrozero} and \eqref{normaleriemannianaepsilon}, if $\B,\D\in\Gamma(TS)$,
         \begin{equation}\label{formadicnnepsilon}
        \begin{split}
            \j^\eps(\n^\eps,\n^\eps)(\B,\D)&=\frac{3}{\eps^2}\left\langle \n^\eps,J(\B)\right\rangle _\eps \left\langle \n^\eps,J(\D)\right\rangle_\eps-\frac{B^{2n+1,\eps}D^{2n+1,\eps}}{\eps^2}-\frac{\alpha^2}{1+\eps^2\alpha^2}\langle\B,\D\rangle_\eps.
        \end{split}
    \end{equation}
    Therefore, 
    \begin{equation*}
        \begin{split}
              \left\langle h^\eps,\j^\eps(\n^\eps,\n^\eps)\right\rangle_\eps&=\sum_{i,j=1}^{2n}h^\eps(\E_i,\E_j)\j^\eps(\n^\eps,\n^\eps)(\E_i,\E_j)\\
              \overset{\eqref{formadicnnepsilon}}&{=}\frac{3}{\eps^2}h^\eps (J(\vh),J(\vh))\left\langle \n^\eps,\vh\right\rangle_\eps^2-\frac{1}{\eps^2}h^\eps(\s^\eps,\s^\eps)\left\langle\s^\eps,\eps T\right\rangle_\eps^2-\frac{\alpha^2}{1+\eps^2\alpha^2}H^\eps\\
              \overset{\eqref{normaleriemannianaepsilon},\eqref{formadiesseepsilon},\eqref{formaepsenonepstuttoorizz},\eqref{curvepsecurvh}}&{=}\frac{3\tilde h^\hhh(J(\vh),J(\vh))}{\eps^2\left(1+\eps^2\alpha^2\right)^\frac{3}{2}}-\frac{\s\alpha}{\left(1+\eps^2\alpha^2\right)^\frac{5}{2}}-\frac{H^\hhh\alpha^2}{\left(1+\eps^2\alpha^2\right)^\frac{3}{2}}--\frac{\eps^2\alpha^2\s\alpha}{\left(1+\eps^2\alpha^2\right)^\frac{5}{2}}\\
              &=\frac{3\tilde h^\hhh(J(\vh),J(\vh))}{\eps^2\left(1+\eps^2\alpha^2\right)^\frac{3}{2}}-\frac{\s\alpha}{\left(1+\eps^2\alpha^2\right)^\frac{3}{2}}-\frac{H^\hhh\alpha^2}{\left(1+\eps^2\alpha^2\right)^\frac{3}{2}}.
        \end{split}
    \end{equation*}
\end{proof}
    \begin{theorem}
        Let $\eps>0$. Let $p\in S\setminus S_0$. 
        Then
        \begin{equation*}
            \begin{split}
                2&\trace\left(\left(h^\eps\right)^3\right)-2 \left\langle h^\eps,\j^\eps(\n^\eps,\n^\eps)\right\rangle_\eps-\left(\nabla^\eps_{\n^\eps}\ric^\eps\right)(\n^\eps,\n^\eps)\\
                &=\frac{1}{\left(1+\eps^2\alpha^2\right)^\frac{3}{2}}\left(2\trace\left(\left(\tilde h^\hhh\right)^3\right)+12\tilde h^\hhh\left(\nabla^{\hhh,S}\alpha,J(\vh)\right)+8\s\alpha+6\tilde h^\hhh(J(\vh),J(\vh))\alpha^2+2 H^\hhh\alpha^2\right)\\
                &\quad +\frac{6\eps^2}{\left(1+\eps^2\alpha^2\right)^\frac{5}{2}}\left(\tilde h^\hhh\left(\nabla^{\hhh,S}\alpha,\nabla^{\hhh,S}\alpha\right)+2J(\vh)\alpha\s\alpha\right)+\frac{6\eps^4}{\left(1+\eps^2\alpha^2\right)^\frac{7}{2}}\s\alpha\left|\nabla^{\hhh,S}\alpha\right|^2+\frac{2\eps^6}{\left(1+\eps^2\alpha^2\right)^\frac{9}{2}}\left(\s\alpha\right)^3.
            \end{split}
        \end{equation*}
        In particular,
        \begin{equation}\label{qbqequation}
            \begin{split}
              2&\trace\left(\left(h^\eps\right)^3\right)-2 \left\langle h^\eps,\j^\eps(\n^\eps,\n^\eps)\right\rangle_\eps-\left(\nabla^\eps_{\n^\eps}\ric^\eps\right)(\n^\eps,\n^\eps)\\
                &\xrightarrow[\eps\to 0]{}  2\trace\left(\left(\tilde h^\hhh\right)^3\right)+12\tilde h^\hhh\left(\nabla^{\hhh,S}\alpha,J(\vh)\right)+8\s\alpha+6\tilde h^\hhh(J(\vh),J(\vh))\alpha^2+2 H^\hhh\alpha^2
            \end{split}
        \end{equation}
        locally uniformly on $S\setminus S_0$.
    \end{theorem}
    \begin{proof}
    Notice that
    \begin{equation*}
        \begin{split}
            \frac{6\tilde h^\hhh(J(\vh),J(\vh))}{\eps^2\left(1+\eps^2\alpha^2\right)^\frac{1}{2}}-\frac{6\tilde h^\hhh(J(\vh),J(\vh))}{\eps^2\left(1+\eps^2\alpha^2\right)^\frac{3}{2}}=\frac{6\tilde h^\hhh(J(\vh),J(\vh))\alpha^2}{\left(1+\eps^2\alpha^2\right)^\frac{3}{2}}.
        \end{split}
    \end{equation*}
        The thesis follows by \eqref{nabaepsriccizerosudiagonale}, \Cref{lemmatracehepsterza} and \Cref{lemmadiscalprodhconcnn}.
    \end{proof}
    \subsection{Variation formulas}
      We establish the relevant variation formulas for \eqref{tmcfunsubriemderfgt5ryhythyh}. Accordingly, we adapt some notation of \Cref{subsec_variations} to the sub-Riemannian setting. If $\Phi$ is a variation and $\Y$ is as in \eqref{def_xechis}, we set 
      \begin{equation*}
    \X(q)=\Y(0,q),\qquad \Z(q)=\left.\frac{\partial }{\partial \tau}\right|_{\tau=0}\Y(\tau,q)+\left(\nabla_\X\X\right)(q)\qquad\text{for any $q\in \hh^n$,}
\end{equation*}
where in the above definition $\nabla$ is the pseudohermitian connection \eqref{pseudotorsion}.
According to the Riemannian terminology, we call $\X$ and $\Z$ respectively  \emph{variational velocity field} and \emph{variational acceleration field} of $\Phi$. If a hypersurface $S$ is fixed, we say that a variation is \emph{non-characteristic} whenever $S\cap K(\Phi)\Subset S\setminus S_0$. In the next result, we restrict ourselves to consider non-characteristic variations 
of closed hypersurfaces. If instead one computes variations on non-compact hypersurfaces, it suffices to restrict the relevant functionals to $K(\Phi)$. 
\begin{theorem}\label{teoremavariazionigeneralimainsubriem}
        Let $S\subseteq\hh^n$ be an embedded, closed, two-sided hypersurface of class $C^2$. Let $S\setminus S_0$ be smooth. Fix a function $f:\rr\to\rr$ which is smooth in a neighborhood of $\left\{H^\hhh(p)\,:\,p\in S\setminus S_0\right\}$. Let $\Phi$ be a non-characteristic variation.  Denote by $\X$ and $\Z$ its velocity and acceleration respectively. Define $\varphi,\psi\in C^\infty_c(S\setminus S_0)$ by
        \begin{equation*}\label{lefunzioniimportanti}
        \begin{split}        \varphi&\coloneqq\left\langle \X,\vh+\alpha T\right\rangle,\\
            \psi&\coloneqq \left\langle\Z,\vh+\alpha T\right\rangle-2\left(\X-\varphi\vh\right)\varphi-\Big\langle\nabla_{\X-\varphi\vh}\left(\X-\varphi\vh\right),\vh+\alpha T\Big\rangle-2\alpha\varphi\left\langle \X,J(\vh)\right\rangle.
        \end{split}
        \end{equation*}
               Set 
        \begin{equation*}
 \delta\tmc^\hhh_f(S)[\Phi]\coloneqq \left.\frac{d}{d\tau}\right|_{\tau=0}\tmc^\hhh_f\left(\Phi_\tau(S)\right),\qquad   \delta^2\tmc^\hhh_f(S)[\Phi]\coloneqq \left.\frac{d^2}{d\tau^2}\right|_{\tau=0}\tmc^\hhh_f\left(\Phi_\tau(S)\right).
\end{equation*} 
        Then 
            \begin{align}            \delta\tmc^\hhh_f(S)[\varphi]\coloneqq\delta\tmc^\hhh_f(S)[\Phi]&=  \int_S\varphi\Big(\jacobi^\hhh f'(H^\hhh)+f(H^\hhh) H^\hhh\Big)\,d\sigma^\hhh,\label{variazprimasubriemmain}\\
            \delta^2\tmc_f^\hhh(S)[\Phi]&=\delta\tmc^\hhh_f(S)\left[\psi\right]+\int_S\varphi\mathcal L^\hhh\varphi\,d\sigma^\hhh, \label{variazionesecondasubriemannianamain}
    \end{align}
    where $ \mathcal L^\hhh$ is the self-adjoint operator defined on $S\setminus S_0$ by
    \begin{equation*}
        \begin{split}
            \mathcal L^\hhh\varphi&=\jacobi^\hhh\left(f''(H^\hhh)\jacobi^\hhh\varphi\right)\\
            &\quad+2\divv^{\hhh,S} \left\langle f'(H^\hhh)A^\hhh\left(\nabla^{\hhh,S}\varphi\right)\right)+4 f'(H^\hhh)\s J(\vh)\varphi-\left(f'(H^\hhh)H^\hhh+f(H^\hhh)\right)\hat\Delta^{\hhh,S}\varphi\\
            &\quad+4\alpha f'(H^\hhh)\Big(\tilde h^\hhh\left(\nabla^{\hhh,S}\varphi,J(\vh)\right)-H^\hhh J(\vh)\varphi\Big)- 4 f''(H^\hhh)\Big(\tilde h^\hhh\left(\nabla^{\hhh,S}H^\hhh,\nabla^{\hhh,S}\varphi\right)+ \s \varphi J(\vh)H^\hhh\Big)\\
            &\quad+\left(f''(H^\hhh)H^\hhh-2f'(H^\hhh)\right)\left\langle \nabla^{\hhh,S} H^\hhh,\nabla^{\hhh,S}\varphi\right\rangle\\
            &\quad+\varphi f'(H^\hhh)\Big( 2\trace\left(\left(\tilde h^\hhh\right)^3\right)+12\tilde h^\hhh\left(\nabla^{\hhh,S}\alpha,J(\vh)\right)+8\s\alpha+6\tilde h^\hhh(J(\vh),J(\vh))\alpha^2+2 H^\hhh\alpha^2\Big)\\
            &\quad+\varphi \Big ( f(H^\hhh)\left( H^\hhh\right)^2-\left(2f'(H^\hhh) H^\hhh+f(H^\hhh)\right)\left(|\tilde h^\hhh|^2+4J(\vh)\alpha+(2n+2)\alpha^2\right)\Big).
        \end{split}
    \end{equation*}
   Here $\jacobi^\hhh$ and $A^\hhh$ are the horizontal Jacobi operator and the horizontal shape operator (cf. \Cref{subsechypersurf}).
    \end{theorem} 
    \begin{proof}
        Fix $\eps>0$. 
 Set $   \tmc^\eps_f(S)\coloneqq\int_S f\left(H^\eps\right)\,d\sigma^\eps. $
Then \begin{equation}\label{confrontoduefunzionali}
       \begin{split}
           \tmc^\eps_f(S)\overset{\eqref{areaelementeps},\eqref{curvepsecurvh}}&{=}\int_S f\left(\frac{H^\hhh}{\sqrt{1+\eps^2\alpha^2}}+\frac{\eps^2\s\alpha}{\left(1+\eps^2\alpha^2\right)^{\frac{3}{2}}}\right)\frac{\sqrt{1+\eps^2\alpha^2}}{\eps\sqrt{1+\alpha^2}}\,d\sigma^1,\\
       \tmc^\hhh_f(S)\overset{\eqref{areaelementorizz}}&{=}\int_S f\left(H^\hhh\right)\frac{1}{\sqrt{1+\alpha^2}}\,d\sigma^1.\\
       \end{split}
\end{equation}
Therefore, by \eqref{confrontoduefunzionali} and since $\Phi$ is supported on $S\setminus S_0$, we infer that 
\begin{equation}\label{convergenzadellevariazioniprime}
    \delta\tmc^\hhh_f(S)[\Phi]=\lim_{\eps\to 0}\eps  \delta\tmc^\eps_f(S)[\Phi],\qquad \delta^2\tmc^\hhh_f(S)[\Phi]=\lim_{\eps\to 0}\eps  \delta^2\tmc^\eps_f(S)[\Phi].
\end{equation}
      Set $\varphi^\eps=\left\langle\X,\n^\eps\right\rangle_\eps$. Then, by \eqref{primavarfunzgen},
\begin{equation*}
      \eps\delta\tmc_f^\eps(S)[\Phi]=  \int_S\Big(f'(H^\eps)\left(-\Delta^{\eps,S}\varphi^\eps-\varphi^\eps\left(|h^\eps|^2+\ric^\eps(\n^\eps,\n^\eps)\right)\right)+\varphi^\eps f(H^\eps) H^\eps\Big)\frac{\sqrt{1+\eps^2\alpha^2}}{\sqrt{1+\alpha^2}}\,d\sigma^1.
    \end{equation*}
    Notice that 
    \begin{equation*}
        \varphi^\eps=\left\langle\X,\n^\eps\right\rangle=\frac{\left\langle \X,\vh\right\rangle}{\sqrt{1+\eps^2\alpha^2}}+\frac{\eps\alpha\left\langle \X,\eps T\right\rangle_\eps}{\sqrt{1+\eps^2\alpha^2}}=\frac{\left\langle \X,\vh\right\rangle}{\sqrt{1+\eps^2\alpha^2}}+\frac{\alpha\left\langle \X, T\right\rangle_1}{\sqrt{1+\eps^2\alpha^2}}=\frac{\varphi}{\sqrt{1+\eps^2\alpha^2}}.
    \end{equation*}
    Therefore, \eqref{laplaepsilonespresso} implies that 
    \begin{equation*}
       \varphi^\eps\xrightarrow[\eps \to 0]{}\varphi,\qquad \Delta^{\eps,S}\varphi^\eps\xrightarrow[\eps\to 0]{}\hat\Delta^{\hhh,S}\varphi\qquad\text{uniformly on $S$}.
    \end{equation*}
    Then, by \eqref{convergenzadellevariazioniprime} and recalling \eqref{convhepstoh} and \eqref{convergenzadelpotenzialejacobi}, \eqref{variazprimasubriemmain} follows. Next, we prove \eqref{variazionesecondasubriemannianamain}. Set $$\X^{T,\eps}=\X-\varphi^\eps\n^\eps=\X-\frac{\varphi}{1+\eps^2\alpha^2}\left(\vh+\eps^2\alpha T\right),$$ and denote by $\Z^\eps$ the acceleration of $\Phi$ with respect to $\nabla^\eps$.
 By \eqref{variazsecondanormalemaarbitrariaperilresto}, 
   \begin{equation*}
        \begin{split}
            \delta^2\tmc^\eps_f(S)&[\Phi]=\delta\tmc^\eps_f(S)\left[\left\langle\Z|_S,\n^\eps\right\rangle_\eps-2\X^{T,\eps}\varphi^\eps +h^\eps\left(\X^{T,\eps},\X^{T,\eps}\right)\right]+\int_S\varphi^\eps\,\jacobi^\eps\left(f''(H^\eps)\jacobi^\eps\varphi^\eps\right)\,d\sigma^\eps\\
            &\quad+\int_S2\varphi^\eps \divv^{\eps,S} \left(f'(H^\eps)A^\eps\left(\nabla^{\eps,S}\varphi^\eps\right)\right)\,d\sigma^\eps-\int_S\varphi^\eps\left(f'(H^\eps)H^\eps+f(H^\eps)\right)\Delta^{\eps,S}\varphi^\eps\,d\sigma^\eps\\
            &\quad-\int_S4\varphi ^\eps f''(H^\eps)h^\eps\left(\nabla^{\eps,S} H^\eps,\nabla^{\eps,S}\varphi^\eps\right)\,d\sigma^\eps+\int_S\varphi^\eps\left(f''(H^\eps)H^\eps-2f'(H^\eps)\right)\left\langle \nabla^{\eps,S} H^\eps,\nabla^{\eps,S}\varphi^\eps\right\rangle_\eps\,d\sigma^\eps\\
            &\quad+\int_S\left(\varphi ^\eps\right)^2f'(H^\eps)\Big(2\trace\left(\left(h^\eps\right)^3\right)-2\left\langle h^\eps,\j^\eps(\n^\eps,\n^\eps)\right\rangle-\left(\nabla^\eps_{\n^\eps} \ric^\eps\right)(\n^\eps,\n^\eps)\Big)\,d\sigma^\eps\\
            &\quad+\int_S\left(\varphi^\eps\right)^2 \Big ( f(H^\eps)\left( H^\eps\right)^2-\left(2f'(H^\eps) H^\eps+f(H^\eps)\right)\left(|h^\eps|^2+ \ric^\eps(\n^\eps,\n^\eps)\right)\Big)\,d\sigma ^\eps.
        \end{split}
    \end{equation*}
    Denote the terms on the right-hand side by $\mathrm{I^\eps},\ldots,\mathrm{VIII^\eps}$. First, 
    \begin{equation*}
        \left\langle\Z^\eps,\n^\eps\right\rangle_\eps=\frac{\left\langle\X',\vh+\alpha T\right\rangle_1}{\sqrt{1+\eps^2\alpha^2}}+\left\langle \nabla^\eps_\X\X,\n^\eps\right\rangle_\eps\overset{\eqref{levicivitavspseudohermitian}}{=}\frac{\left\langle\Z,\vh+\alpha T\right\rangle_1}{\sqrt{1+\eps^2\alpha^2}}+\frac{2\left\langle\X,\eps T\right\rangle_\eps}{\eps}\left\langle J(\X),\n^\eps\right\rangle_\eps.
    \end{equation*}
    Moreover, $\X-\varphi\vh\in\Gamma(TS)$, and 
    \begin{equation*}
        -2\X^{T,\eps}\varphi^\eps=-2\left(\X-\frac{\varphi}{1+\eps^2\alpha^2}\left(\vh+\eps^2\alpha T\right)\right)\left(\frac{\varphi}{\sqrt{1+\eps^2\alpha^2}}\right)\xrightarrow[\eps\to 0]{}-2\left(\X-\varphi\vh\right)\varphi\qquad\text{uniformly on $S$}.
    \end{equation*}
    In addition,
    \begin{equation*}
        h^\eps\left(\X^{T,\eps},\X^{T,\eps}\right)=-\left\langle\nabla^\eps_{\X^{T,\eps}}\X^{T,\eps},\n^\eps\right\rangle_\eps\overset{\eqref{levicivitavspseudohermitian}}{=}-\left\langle\nabla_{\X^{T,\eps}}\X^{T,\eps},\n^\eps\right\rangle_\eps-\frac{2\left\langle\X^{T,\eps},\eps T\right\rangle_\eps}{\eps}\left\langle J(\X),\n^\eps\right\rangle_\eps.
    \end{equation*}
Therefore
\begin{equation*}
    \begin{split}
        \left\langle\Z^\eps,\n^\eps\right\rangle_\eps+h^\eps\left(\X^{T,\eps},\X^{T,\eps}\right)&=\frac{\left\langle\Z,\vh+\alpha T\right\rangle_1}{\sqrt{1+\eps^2\alpha^2}}-\left\langle\nabla_{\X^{T,\eps}}\X^{T,\eps},\n^\eps\right\rangle_\eps+\frac{2\left\langle J(\X),\n^\eps\right\rangle_\eps}{\eps}\left\langle\X-\X^{T,\eps},\eps T\right\rangle_\eps\\
        &=\frac{\left\langle\Z,\vh+\alpha T\right\rangle_1}{\sqrt{1+\eps^2\alpha^2}}-\left\langle\nabla_{\X^{T,\eps}}\X^{T,\eps},\n^\eps\right\rangle_\eps+\frac{2\alpha\varphi_\epsilon\left\langle J(\X),\n^\eps\right\rangle_\eps}{\sqrt{1+\eps^2\alpha^2}}\\
        &\xrightarrow[\eps\to 0]{}\left\langle\Z,\vh+\alpha T\right\rangle_1-\Big\langle\nabla_{\X-\varphi\vh}\left(\X-\varphi\vh\right),\vh+\alpha T\Big\rangle_1-2\alpha\varphi\left\langle \X,J(\vh)\right\rangle
    \end{split}
\end{equation*}
uniformly on $S$. By the first part of the proof, we conclude that 
    \begin{equation*}
        \eps\mathrm{I^\eps}\xrightarrow[\eps\to 0]{}\delta\tmc^\hhh_f(S)\left[\left\langle\Z,\vh+\alpha T\right\rangle_1-2\left(\X-\varphi\vh\right)\varphi-\Big\langle\nabla_{\X-\varphi\vh}\left(\X-\varphi\vh\right),\vh+\alpha T\Big\rangle_1-2\alpha\varphi\left\langle \X,J(\vh)\right\rangle\right].
    \end{equation*}
    Moreover, by \eqref{jacepstojach},
    \begin{equation*}
        \eps\mathrm{II^\eps}=\eps \int_Sf''(H^\eps)\jacobi^\eps\varphi^\eps\,\jacobi^\eps\varphi^\eps\,d\sigma^\eps\xrightarrow[\eps\to 0]{}\int_Sf''(H^\hhh)\jacobi^\hhh\varphi\,\jacobi^\hhh\varphi\,d\sigma^\hhh=\int_S\varphi\,\jacobi^\hhh\left(f''(H^\hhh)\jacobi^\hhh\varphi\right)\,d\sigma^\hhh.
    \end{equation*}
    Next,
    \begin{equation*}
        \begin{split}
            \eps\mathrm{III^\eps}&=-\eps\int_S 2 f'(H^\eps)h^\eps\left(\nabla^{\eps,S}\varphi^\eps,\nabla^{\eps,S}\varphi^\eps\right)\,d\sigma^\eps\\
            \overset{\eqref{convhepsgradeps}}&{\xrightarrow[\eps\to 0]{}}-\int_S 2f'(H^\hhh)\Big(\tilde h^\hhh\left(\nabla^{\hhh,S}\varphi,\nabla^{\hhh,S}\varphi\right)+ 2J(\vh)\varphi\s\varphi\Big)\,d\sigma^\hhh\\
            \overset{\eqref{rapptrahetildacca2026}}&{=}
            -\int_S 2 \left\langle f'(H^\hhh)A^\hhh\left(\nabla^{\hhh,S}\varphi\right),\nabla^{\hhh,S}\varphi\right\rangle\,d\sigma^\hhh -\int_S 4f'(H^\hhh )J(\vh)\varphi\s\varphi\,d\sigma^\hhh\\
            \overset{\eqref{ibpsubriemequation},\eqref{ibpverticalformula}}&{=} \int_S 2\varphi\divv^{\hhh,S} \left\langle f'(H^\hhh)A^\hhh\left(\nabla^{\hhh,S}\varphi\right)\right)\,d\sigma^\hhh+\int_S4\varphi\alpha f'(H^\hhh)h^\hhh\left(\nabla^{\hhh,S}\varphi,J(\vh)\right)\,d\sigma^\hhh\\
            &\quad+\int_S 4\varphi J(\vh)\varphi f''(H^\hhh)\s H^\hhh\,d\sigma^\hhh+\int_S 4\varphi f'(H^\hhh)\s J(\vh)\varphi,d\sigma^\hhh-\int_S4\varphi\alpha f'(H^\hhh)H^\hhh J(\vh)\varphi\,d\sigma^\hhh\\
            \overset{\eqref{rapptrahetildacca2026}}&{=}\int_S 2\varphi\divv^{\hhh,S} \left\langle f'(H^\hhh)A^\hhh\left(\nabla^{\hhh,S}\varphi\right)\right)\,d\sigma^\hhh+\int_S4\varphi\alpha f'(H^\hhh)\Big(\tilde h^\hhh\left(\nabla^{\hhh,S}\varphi,J(\vh)\right)-H^\hhh J(\vh)\varphi\Big)\,d\sigma^\hhh\\
            &\quad+\int_S 4\varphi J(\vh)\varphi f''(H^\hhh)\s H^\hhh\,d\sigma^\hhh+\int_S 4\varphi f'(H^\hhh)\s J(\vh)\varphi,d\sigma^\hhh.
        \end{split}
    \end{equation*}
    Moreover,
    \begin{equation*}
        \eps\mathrm{IV}^\eps\xrightarrow[\eps\to 0]{}-\int_S\varphi\left(f'(H^\hhh)H^\hhh+f(H^\hhh)\right)\hat\Delta^{\hhh,S}\varphi\,d\sigma^\hhh.
    \end{equation*}
    In addition,
    \begin{equation*}
       \eps \mathrm{V}^\eps   \overset{\eqref{convhepsgradeps}}{\xrightarrow[\eps\to 0]{}}-\int_S 4\varphi f''(H^\hhh)\Big(\tilde h^\hhh\left(\nabla^{\hhh,S}H^\hhh,\nabla^{\hhh,S}\varphi\right)+ J(\vh)\varphi\s H^\hhh+ \s \varphi J(\vh)H^\hhh\Big)\,d\sigma^\hhh.
    \end{equation*}
    Moreover, 
    \begin{equation*}
        \eps\mathrm{VI^\eps}\xrightarrow[\eps\to 0]{}\int_S\varphi\left(f''(H^\hhh)H^\hhh-2f'(H^\hhh)\right)\left\langle \nabla^{\hhh,S} H^\hhh,\nabla^{\hhh,S}\varphi\right\rangle\,d\sigma^\hhh.
    \end{equation*}
    Furthermore, by \eqref{qbqequation},
    \begin{equation*}
        \eps\mathrm{VII}^\eps\xrightarrow[\eps\to 0]{}\int_S\varphi^2 f'(H^\hhh)\Big( 2\trace\left(\left(\tilde h^\hhh\right)^3\right)+12\tilde h^\hhh\left(\nabla^{\hhh,S}\alpha,J(\vh)\right)+8\s\alpha+6\tilde h^\hhh(J(\vh),J(\vh))\alpha^2+2 H^\hhh\alpha^2\Big)\,d\sigma^\hhh.
    \end{equation*}
    Finally, by \eqref{convergenzadelpotenzialejacobi},
    \begin{equation*}
        \eps\mathrm{VIII}^\eps\xrightarrow[\eps\to 0]{}\int_S\varphi^2 \Big ( f(H^\hhh)\left( H^\hhh\right)^2-\left(2f'(H^\hhh) H^\hhh+f(H^\hhh)\right)\left(|\tilde h^\hhh|^2+4J(\vh)\alpha+(2n+2)\alpha^2\right)\Big)\,d\sigma^\hhh.
    \end{equation*}
  The thesis follows combining the above computations.
    \end{proof}
\section*{Declarations}
\footnotesize{
%
\noindent{

\noindent{\textbf{Conflict of interest.}} The author has no financial interests or conflicts of interest related to the subject matter.
\smallskip

\noindent \textbf{Data availability statement.} Data sharing not applicable as no datasets were generated or analyzed during the current study.
\bibliographystyle{abbrvmr_macro}
\bibliography{biblio}

\begin{thebibliography}{10}

\bibitem{MR4037467}
V.~Agostiniani and L.~Mazzieri.
\newblock Monotonicity formulas in potential theory.
\newblock {\em Calc. Var. Partial Differential Equations}, 59(1):Paper No. 6,
  32, 2020.
\newblock \MR{4037467}.

\bibitem{MR3971262}
A.~Agrachev, D.~Barilari, and U.~Boscain.
\newblock {\em A comprehensive introduction to sub-{R}iemannian geometry},
  volume 181 of {\em Cambridge Studies in Advanced Mathematics}.
\newblock Cambridge University Press, Cambridge, 2020.
\newblock \MR{3971262}.

\bibitem{MR86338}
A.~D. Aleksandrov.
\newblock Uniqueness theorems for surfaces in the large. {I}.
\newblock {\em Vestnik Leningrad. Univ.}, 11(19):5--17, 1956.
\newblock \MR{86338}.

\bibitem{MR2021034}
Z.~M. Balogh.
\newblock Size of characteristic sets and functions with prescribed gradient.
\newblock {\em J. Reine Angew. Math.}, 564:63--83, 2003.
\newblock \MR{2021034}.

\bibitem{MR0731682}
J.~L. Barbosa and M.~do~Carmo.
\newblock Stability of hypersurfaces with constant mean curvature.
\newblock {\em Math. Z.}, 185(3):339--353, 1984.
\newblock \MR{731682}.

\bibitem{MR0917854}
J.~L. Barbosa, M.~do~Carmo, and J.~Eschenburg.
\newblock Stability of hypersurfaces of constant mean curvature in {R}iemannian
  manifolds.
\newblock {\em Math. Z.}, 197(1):123--138, 1988.
\newblock \MR{917854}.

\bibitem{MR2363343}
A.~Bonfiglioli, E.~Lanconelli, and F.~Uguzzoni.
\newblock {\em Stratified {L}ie groups and potential theory for their
  sub-{L}aplacians}.
\newblock Springer Monographs in Mathematics. Springer, Berlin, 2007.
\newblock \MR{2363343}.

\bibitem{MR2774306}
L.~Capogna, G.~Citti, and M.~Manfredini.
\newblock Smoothness of {L}ipschitz minimal intrinsic graphs in {H}eisenberg
  groups {$\Bbb H^n$}, {$n>1$}.
\newblock {\em J. Reine Angew. Math.}, 648:75--110, 2010.
\newblock \MR{2774306}.

\bibitem{MR4941125}
C.~Cederbaum and A.~Miehe.
\newblock A new proof of the {W}illmore inequality via a divergence inequality.
\newblock {\em Trans. Amer. Math. Soc.}, 378(9):6655--6676, 2025.
\newblock \MR{4941125}.

\bibitem{MR291994}
B.-y. Chen.
\newblock On a theorem of {F}enchel-{B}orsuk-{W}illmore-{C}hern-{L}ashof.
\newblock {\em Math. Ann.}, 194:19--26, 1971.
\newblock \MR{291994}.

\bibitem{MR278240}
B.-y. Chen.
\newblock On the total curvature of immersed manifolds. {I}. {A}n inequality of
  {F}enchel-{B}orsuk-{W}illmore.
\newblock {\em Amer. J. Math.}, 93:148--162, 1971.
\newblock \MR{278240}.

\bibitem{MR2165405}
J.-H. Cheng, J.-F. Hwang, A.~Malchiodi, and P.~Yang.
\newblock Minimal surfaces in pseudohermitian geometry.
\newblock {\em Ann. Sc. Norm. Super. Pisa Cl. Sci. (5)}, 4(1):129--177, 2005.
\newblock \MR{2165405}.

\bibitem{MR2262784}
J.-H. Cheng, J.-F. Hwang, and P.~Yang.
\newblock Existence and uniqueness for {$p$}-area minimizers in the
  {H}eisenberg group.
\newblock {\em Math. Ann.}, 337(2):253--293, 2007.
\newblock \MR{2262784}.

\bibitem{MR3544938}
J.~Dalphin, A.~Henrot, S.~Masnou, and T.~Takahashi.
\newblock On the minimization of total mean curvature.
\newblock {\em J. Geom. Anal.}, 26(4):2729--2750, 2016.
\newblock \MR{3544938}.

\bibitem{MR2354992}
D.~Danielli, N.~Garofalo, and D.~M. Nhieu.
\newblock Sub-{R}iemannian calculus on hypersurfaces in {C}arnot groups.
\newblock {\em Adv. Math.}, 215(1):292--378, 2007.
\newblock \MR{2354992}.

\bibitem{MR3921314}
M.~G. Delgadino and F.~Maggi.
\newblock Alexandrov's theorem revisited.
\newblock {\em Anal. PDE}, 12(6):1613--1642, 2019.
\newblock \MR{3921314}.

\bibitem{MR1138207}
M.~P. do~Carmo.
\newblock {\em Riemannian geometry}.
\newblock Mathematics: Theory \& Applications. Birkh\"{a}user Boston, Inc.,
  Boston, MA, portuguese edition, 1992.
\newblock \MR{1138207}.

\bibitem{MR3666332}
C.~M. Elliott, H.~Fritz, and G.~Hobbs.
\newblock Small deformations of {H}elfrich energy minimising surfaces with
  applications to biomembranes.
\newblock {\em Math. Models Methods Appl. Sci.}, 27(8):1547--1586, 2017.
\newblock \MR{3666332}.

\bibitem{MR2723818}
N.~Garofalo and C.~Selby.
\newblock Collapsing {R}iemannian metrics to sub-{R}iemannian and the geometry
  of hypersurfaces in {C}arnot groups.
\newblock In {\em Around the research of {V}ladimir {M}az'ya. {I}}, volume~11
  of {\em Int. Math. Ser. (N. Y.)}, pages 169--206. Springer, New York, 2010.
\newblock \MR{2723818}.

\bibitem{MR4761954}
G.~Giovannardi, A.~Pinamonti, J.~Pozuelo, and S.~Verzellesi.
\newblock The prescribed mean curvature equation for {$t$}-graphs in the
  sub-{F}insler {H}eisenberg group {$\Bbb{H}^n$}.
\newblock {\em Adv. Math.}, 451:Paper No. 109788, 43, 2024.
\newblock \MR{4761954}.

\bibitem{simons}
G.~Giovannardi, A.~Pinamonti, and S.~Verzellesi.
\newblock Curvature estimates for minimal hypersurfaces in the heisenberg
  group.
\newblock preprint, \url{https://doi.org/10.48550/arXiv.2409.20359}, 2024.

\bibitem{MR3962031}
A.~Gruber, M.~Toda, and H.~Tran.
\newblock On the variation of curvature functionals in a space form with
  application to a generalized {W}illmore energy.
\newblock {\em Ann. Global Anal. Geom.}, 56(1):147--165, 2019.
\newblock \MR{3962031}.

\bibitem{MR2522433}
P.~Guan and J.~Li.
\newblock The quermassintegral inequalities for {$k$}-convex starshaped
  domains.
\newblock {\em Adv. Math.}, 221(5):1725--1732, 2009.
\newblock \MR{2522433}.

\bibitem{MR533065}
E.~Heintze and H.~Karcher.
\newblock A general comparison theorem with applications to volume estimates
  for submanifolds.
\newblock {\em Ann. Sci. \'Ecole Norm. Sup. (4)}, 11(4):451--470, 1978.
\newblock \MR{533065}.

\bibitem{MR2262196}
V.~Magnani.
\newblock Characteristic points, rectifiability and perimeter measure on
  stratified groups.
\newblock {\em J. Eur. Math. Soc. (JEMS)}, 8(4):585--609, 2006.
\newblock \MR{2262196}.

\bibitem{MR1511220}
H.~Minkowski.
\newblock Volumen und {O}berfl\"ache.
\newblock {\em Math. Ann.}, 57(4):447--495, 1903.
\newblock \MR{1511220}.

\bibitem{MR1173047}
S.~Montiel and A.~Ros.
\newblock Compact hypersurfaces: the {A}lexandrov theorem for higher order mean
  curvatures.
\newblock In {\em Differential geometry}, volume~52 of {\em Pitman Monogr.
  Surveys Pure Appl. Math.}, pages 279--296. Longman Sci. Tech., Harlow, 1991.
\newblock \MR{1173047}.

\bibitem{MR2043961}
S.~D. Pauls.
\newblock Minimal surfaces in the {H}eisenberg group.
\newblock {\em Geom. Dedicata}, 104:201--231, 2004.
\newblock \MR{2043961}.

\bibitem{MR4923606}
A.~Pinamonti and S.~Verzellesi.
\newblock A characterization of horizontally totally geodesic hypersurfaces in
  {H}eisenberg groups.
\newblock {\em J. Geom. Anal.}, 35(8):Paper No. 235, 35, 2025.
\newblock \MR{4923606}.

\bibitem{MR4927775}
T.~Pires and W.~Santos.
\newblock On the variation of r-mean curvature functionals and application to
  the {$L^2$}-norm of the traceless second fundamental form.
\newblock {\em Ann. Global Anal. Geom.}, 68(2):Paper No. 5, 23, 2025.
\newblock \MR{4927775}.

\bibitem{MR5029815}
J.~Pozuelo and S.~Verzellesi.
\newblock Existence and uniqueness of {$t$}-graphs of prescribed mean curvature
  in {H}eisenberg groups.
\newblock {\em Trans. Amer. Math. Soc.}, 379(3):1883--1929, 2026.
\newblock \MR{5029815}.

\bibitem{MR341351}
R.~C. Reilly.
\newblock Variational properties of functions of the mean curvatures for
  hypersurfaces in space forms.
\newblock {\em J. Differential Geometry}, 8:465--477, 1973.
\newblock \MR{341351}.

\bibitem{MR4193432}
M.~Ritor\'{e}.
\newblock Tubular neighborhoods in the sub-{R}iemannian {H}eisenberg groups.
\newblock {\em Adv. Calc. Var.}, 14(1):1--36, 2021.
\newblock \MR{4193432}.

\bibitem{MR4676392}
M.~Ritor\'e.
\newblock {\em Isoperimetric inequalities in {R}iemannian manifolds}, volume
  348 of {\em Progress in Mathematics}.
\newblock Birkh\"auser/Springer, Cham, [2023] \copyright 2023.
\newblock \MR{4676392}.

\bibitem{MR996826}
A.~Ros.
\newblock Compact hypersurfaces with constant higher order mean curvatures.
\newblock {\em Rev. Mat. Iberoamericana}, 3(3-4):447--453, 1987.
\newblock \MR{996826}.

\bibitem{MR3702549}
C.~Xia and X.~Zhang.
\newblock A{BP} estimate and geometric inequalities.
\newblock {\em Comm. Anal. Geom.}, 25(3):685--708, 2017.
\newblock \MR{3702549}.

\end{thebibliography}
\end{document}